\documentclass[11pt]{article}%
\usepackage{amssymb}
\usepackage{amsmath}
\usepackage{amsfonts}
\usepackage{geometry}%
\usepackage{graphicx}
\providecommand{\U}[1]{\protect\rule{.1in}{.1in}}
\newtheorem{theorem}{Theorem}
\newtheorem{acknowledgement}[theorem]{Acknowledgement}

\newtheorem{corollary}[theorem]{Corollary}

\newtheorem{definition}[theorem]{Definition}

\newtheorem{lemma}[theorem]{Lemma}

\newtheorem{proposition}[theorem]{Proposition}
\newtheorem{remark}[theorem]{Remark}

\newenvironment{proof}[1][Proof]{\noindent\textbf{#1.} }{\ \rule{0.5em}{0.5em}}
\begin{document}

\title{Weak vorticity formulation of 2D Euler equations with white noise initial condition}
\author{Franco Flandoli}
\maketitle

\begin{abstract}
The 2D Euler equations with random initial condition distributed as a certain
Gaussian measure are considered. The theory developed by S. Albeverio and
A.-B. Cruzeiro in \cite{AlbCruz} is revisited, following the approach of weak
vorticity formulation. A solution is constructed as a limit of random point
vortices. This allows to prove that it is also limit of $L^{\infty}$-vorticity
solutions. The result is generalized to initial measures that have a
continuous bounded density with respect to the original Gaussian measure.

\end{abstract}

\section{Introduction}

We consider the 2D Euler equations on the torus $\mathbb{T}^{2}=\mathbb{R}%
^{2}/\mathbb{Z}^{2}$, formulated in terms of the vorticity $\omega$%
\begin{equation}
\partial_{t}\omega+u\cdot\nabla\omega=0\label{Euler}%
\end{equation}
where $u$ is the velocity, divergence free vector field such that
$\omega=\partial_{2}u_{1}-\partial_{1}u_{2}$. The classical theory (see for
instance \cite{Chemin}, \cite{Lions}, \cite{MajdaBert}, \cite{MarPulv})
includes the following results:

\begin{enumerate}
\item existence and uniqueness of weak solutions of class $L^{\infty}\left(
\left[  0,T\right]  \times\mathbb{T}^{2}\right)  \cap C\left(  \left[
0,T\right]  ;L^{p}\left(  \mathbb{T}^{2}\right)  \right)  $ for every
$p\in\lbrack1,\infty)$, satisfying%
\begin{equation}
\left\langle \omega_{t},\phi\right\rangle =\left\langle \omega_{0}%
,\phi\right\rangle +\int_{0}^{t}\left\langle \omega_{s},u_{s}\cdot\nabla
\phi\right\rangle ds\label{weak Euler}%
\end{equation}
for every $\phi\in C^{\infty}\left(  \mathbb{T}^{2}\right)  $, when the
initial condition $\omega_{0}$ is of class $L^{\infty}\left(  \mathbb{T}%
^{2}\right)  $ (\cite{Wolibner}, \cite{Yudovich}, \cite{MarPulv});

\item existence of weak solutions of class $C\left(  \left[  0,T\right]
;L^{p}\left(  \mathbb{T}^{2}\right)  \right)  $, satisfying (\ref{weak Euler}%
), when the initial condition $\omega_{0}$ is of class $L^{p}\left(
\mathbb{T}^{2}\right)  $, for some $p\in\lbrack1,\infty)$;

\item existence of measure-valued solutions $\omega_{t}\left(  dx\right)  $,
of class $L^{\infty}\left(  0,T;\mathcal{M}\left(  \mathbb{T}^{2}\right)  \cap
H^{-1}\left(  \mathbb{T}^{2}\right)  \right)  $, satisfying for every $\phi\in
C^{\infty}\left(  \mathbb{T}^{2}\right)  $ the so called weak vorticity
formulation%
\begin{equation}
\left\langle \omega_{t},\phi\right\rangle =\left\langle \omega_{0}%
,\phi\right\rangle +\int_{0}^{t}\int_{\mathbb{T}^{2}}\int_{\mathbb{T}^{2}%
}H_{\phi}\left(  x,y\right)  \omega_{s}\left(  dx\right)  \omega_{s}\left(
dy\right)  ds\label{weak vorticity formulation}%
\end{equation}
where%
\[
H_{\phi}\left(  x,y\right)  :=\frac{1}{2}K\left(  x-y\right)  \left(
\nabla\phi\left(  x\right)  -\nabla\phi\left(  y\right)  \right)
\]
and $K\left(  x\right)  $ is Biot-Savart kernel on $\mathbb{T}^{2}$, when the
initial condition is a measure of class $H^{-1}\left(  \mathbb{T}^{2}\right)
$ with a certain condition of preference for a single sign, see \cite{Delort},
\cite{ShochetII}, \cite{DiPernaMajda}; here we have denoted by $\mathcal{M}%
\left(  \mathbb{T}^{2}\right)  $ the space of finite signed measures and by
$H^{\alpha}\left(  \mathbb{T}^{2}\right)  $ the classical Sobolev spaces of
order $\alpha\in\mathbb{R}$ defined in Section \ref{sect notations};

\item existence and uniqueness of a measure-valued solution of the form
$\omega_{t}\left(  dx\right)  =\sum_{i=1}^{N}\xi_{i}\delta_{X_{t}^{i}}$,
fulfilling (\ref{weak vorticity formulation}), when the initial condition has
the form $\omega_{0}\left(  dx\right)  =\sum_{i=1}^{N}\xi_{i}\delta_{X_{0}%
^{i}}$, with real valued intensities $\xi_{1},...,\xi_{N}$, and $\left(
X_{0}^{1},...,X_{0}^{N}\right)  $ belonging to a set of full Lebesgue measure
in $\left(  \mathbb{T}^{2}\right)  ^{N}$, see \cite{MarPulv}.
\end{enumerate}

Obviously there are many other results, reported in the references above and
other works, including counterexamples to uniqueness like \cite{Shnir}. The
previous choice has been made to illustrate the attempt to include weaker and
weaker concepts of solutions. Very important for result n. 3 has been the
symmetrization step from (\ref{weak Euler}) to
(\ref{weak vorticity formulation}): the kernel $H_{\phi}\left(  x,y\right)  $
is bounded, smooth outside the diagonal, discontinuous along the
diagonal;\ hence a fine analysis of the concentration of $\omega_{t}\left(
dx\right)  $ around the diagonal is important but at least the singularity of
order $\frac{1}{\left\vert x\right\vert }$ of Biot-Savart kernel $K\left(
x\right)  $ has been removed.

In the present paper we discuss a probabilistic result for Euler equations,
interpreted in the form (\ref{weak vorticity formulation}). We like to state
it first in purely deterministic terms, here in the introduction, for the sake
of comparison with the "scale" of results above. Then, in the rest of the
paper, the probabilistic side will be stressed more. Denote by $H^{-1-}\left(
\mathbb{T}^{2}\right)  $ the space $%
%TCIMACRO{\dbigcap \limits_{\epsilon>0}}%
%BeginExpansion
{\displaystyle\bigcap\limits_{\epsilon>0}}
%EndExpansion
H^{-1-\epsilon}\left(  \mathbb{T}^{2}\right)  $, with the topology described
in Section \ref{sect notations} and notice that $\mathcal{M}\left(
\mathbb{T}^{2}\right)  \subset H^{-1-}\left(  \mathbb{T}^{2}\right)  $,
because by Sobolev embedding $H^{1+\epsilon}\left(  \mathbb{T}^{2}\right)
\subset C\left(  \mathbb{T}^{2}\right)  $. Moreover, denote by $K_{\epsilon}$
the smooth approximations of $K$ given by (\ref{def K eps}) below and, given a
sequence $\epsilon_{n}\rightarrow0$, set $H_{\phi}^{n}\left(  x,y\right)
:=\frac{1}{2}K_{\epsilon_{n}}\left(  x-y\right)  \left(  \nabla\phi\left(
x\right)  -\nabla\phi\left(  y\right)  \right)  $;\ by classical distribution
theory, $s\mapsto\left\langle \omega_{s}\otimes\omega_{s},H_{\phi}%
^{n}\right\rangle $ is well defined and continuous when $\omega\in C\left(
\left[  0,T\right]  ;H^{-1-}\left(  \mathbb{T}^{2}\right)  \right)  $.

\begin{theorem}
\label{Thm intro}There exist $\epsilon_{n}\rightarrow0$ and a large set
\[
\mathcal{IC}_{0}\subset H^{-1-}\left(  \mathbb{T}^{2}\right)  \backslash
\left(  H^{-1}\left(  \mathbb{T}^{2}\right)  \cup\mathcal{M}\left(
\mathbb{T}^{2}\right)  \right)
\]
of initial conditions such that for all $\omega_{0}\in\mathcal{IC}_{0}$ the
following properties hold.

i) there exists $\omega\in C\left(  \left[  0,T\right]  ;H^{-1-}\left(
\mathbb{T}^{2}\right)  \right)  $ such that, for every $\phi\in C^{\infty
}\left(  \mathbb{T}^{2}\right)  $, the sequence of functions $s\mapsto
\left\langle \omega_{s}\otimes\omega_{s},H_{\phi}^{n}\right\rangle $ is a
Cauchy sequence in $L^{2}\left(  0,T\right)  $ and, denoted by $s\mapsto
\left\langle \omega_{s}\otimes\omega_{s},H_{\phi}\right\rangle $ its limit,
one has the analog of (\ref{weak vorticity formulation}), namely%
\begin{equation}
\left\langle \omega_{t},\phi\right\rangle =\left\langle \omega_{0}%
,\phi\right\rangle +\int_{0}^{t}\left\langle \omega_{s}\otimes\omega
_{s},H_{\phi}\right\rangle ds\label{weak vorticity formulation 1}%
\end{equation}

ii) there is a sequence $\left\{  \omega^{\left(  n\right)  }\right\}  $ of
solutions of Euler equations of class $L^{\infty}\left(  \left[  0,T\right]
\times\mathbb{T}^{2}\right)  \cap C\left(  \left[  0,T\right]  ;L^{p}\left(
\mathbb{T}^{2}\right)  \right)  $ for every $p\in\lbrack1,\infty)$ (those of
point 1 above) such that $\left\langle \omega_{t}^{\left(  n\right)  }%
,\phi\right\rangle \rightarrow\left\langle \omega_{t},\phi\right\rangle $
uniformly in $t\in\left[  0,T\right]  $, for every $\phi\in C^{\infty}\left(
\mathbb{T}^{2}\right)  $.
\end{theorem}

\begin{remark}
How large is the set of initial conditions, it is clarified below in Section
\ref{Sect proof Thm intro}. It is a full measure set with respect to the
Gaussian measure $\mu$ introduced in Section \ref{subseq WN}
\end{remark}

\begin{remark}
In fact the set of initial conditions given by this theorem is included in a
more regular space $H^{-1-,\infty}\left(  \mathbb{T}^{2}\right)  $, where also
the solutions live, defined in Section \ref{subseq WN} below. We have not used
$H^{-1-,\infty}\left(  \mathbb{T}^{2}\right)  $ in place of $H^{-1-}\left(
\mathbb{T}^{2}\right)  $ because $\mathcal{M}\left(  \mathbb{T}^{2}\right)
\nsubseteq H^{-1-,\infty}\left(  \mathbb{T}^{2}\right)  $ and thus the
statement would be less clear. Moreover $H^{-1}\left(  \mathbb{T}^{2}\right)
$ and $\mathcal{M}\left(  \mathbb{T}^{2}\right)  $ are not included one in the
other, which again explains the statement.
\end{remark}

Part (i) of Theorem \ref{Thm intro} is a deterministic reformulation of
Theorem \ref{Thm AC} below, which states that Euler equations, interpreted in
the form (\ref{weak vorticity formulation 1}), has a stochastic solution, a
stationary stochastic process with time marginal given by the so called
\textit{white noise} on $\mathbb{T}^{2}$, defined in Section \ref{subseq WN}
below. This probabilistic result is due to Sergio Albeverio and Ana Bela
Cruzeiro \cite{AlbCruz}. Here we provide, with respect to that seminal work,
the so called weak vorticity formulation (\ref{weak vorticity formulation 1})
(opposite to a Fourier formulation), which fits more nicely in the scheme of
results 1-4 above;\ and we prove the existence of a solution as a limit of
random point vortices, a suitable random version of point 4 above. Opposite to
other schemes that can be used to prove existence, based on
\textit{approximated }equations (like the Galerkin scheme of \cite{AlbCruz},
or a Leray type scheme), point vortices are \textit{true solutions} of Euler
equations (see Section \ref{propos point vortices}) and thus establish a
bridge between $L^{\infty}$-vorticity solutions, those of point 1 above, and
Albeverio-Cruzeiro solution, via the result of approximation of point vortices
by vortex patches of Marchioro and Pulvirenti \cite{MarPulv93}. Nicolai
Tzvetkov suggested to investigate question (ii) of Theorem \ref{Thm intro},
which is similar to a question solved (in a stronger sense) for nonlinear wave
equations, see \cite{Tzvetkov}.

The point vortex approximation provides an interesting interpretation of the
white noise solution of Albeverio and Cruzeiro, as a limit of randomly
distributed vortices with positive and negative random vorticities. Under the
viewpoint of the weak vorticity formulation, having in mind the deep
discussions of the deterministic literature on concentration of solutions of
Euler equations (see for instance the works of Delort \cite{Delort}, Schochet
\cite{Shochet}, \cite{ShochetII}, Poupaud \cite{Poup}, Di Perna and Majda
\cite{DiPernaMajda}), a natural question is why the solution found here with
white noise distribution does not "concentrate on the diagonal", in the double
integration of the weak vorticity formulation, where the function $H_{\phi
}\left(  x,y\right)  $ is discontinuous. For the white noise solution the
absence of concentration is encoded in the results of Section
\ref{section nonlinear WN}, which show the power of Gaussian analysis but may
still look obscure. However, the approximation by point vortices provides a
clear intuition about the lack of concentration: at every time, vortices are
distributed at random uniformly in space, independently one of the other.

Among the reasons to reconsider Albeverio-Cruzeiro theory today, there is the
clear success of randomization of initial conditions in solving dispersive
equations, see for instance \cite{Bourg}, \cite{Burq-Tzv}, \cite{Burq-TzvII},
\cite{QuastelValko}, \cite{Richards} \cite{Oh} (the last two, for instance,
describe another PDE that leaves a Gaussian measure invariant) and in
particular the review of N. Tzvetkov \cite{Tzvetkov} on nonlinear wave
equation where Theorems 2.6, 2.7 are devoted to prove that solutions with poor
regularity (constructed for a.e. initial condition with respect to a Gaussian
measure) are the limit of more regular solutions belonging to the classical
theory. As a technical remark, the approximation result above in Theorem
\ref{Thm intro} is definitely weaker than Theorem 2.6 of \cite{Tzvetkov},
where any reasonable smooth approximation of initial conditions leads to
convergent solutions; it is more in the spirit of Theorem 2.7, where
particular approximations are considered. As a general remark, it is not
reasonable to expect for Euler equations the richness of results obtained in
dispersive equations, but nevertheless it may be of interest to make little
improvements. Another source of inspiration for the present work have been the
striking recent theories for certain stochastic nonlinear equations having
Gaussian measures invariant, see for instance \cite{Hairer}, \cite{Gubi},
\cite{GubiPerk}; however, the difficulties for such equations are much greater
than those solved here, although Gaussian analysis is a common core.

We prove existence of a stochastic solution also when the initial condition is
a random distribution with law that have a continuous bounded density with
respect to the original Gaussian measure. The solution has a density also at
time $t$, that satisfies a continuity equation; the results proved here in
this direction are quite elementary corollaries of the main results on white
noise solutions but we think it is of interest to state them for future
investigations in connection with deeper theories on continuity equations in
infinite dimensions, see for instance \cite{AmbrFigalli}, \cite{AmbrTrevisan},
\cite{Bogach}, \cite{Cruz}, \cite{DaPratoRoeckner}, \cite{DFR}, \cite{FangLuo}%
. The case with a density with respect to white noise arises an open question,
described in Section \ref{sect open}, concerning the approximation of smooth
solutions by white noise ones, a sort of dual problem to the one discussed above.

Let us finally mention several other works related to Gaussian invariant
measures for 2D Euler equations: see \cite{AlbFar}, \cite{AlbFer},
\cite{AlbFerCIME}, \cite{AlbKr}, \cite{Cipriano}, \cite{Symeonide}. Several
elements of these works may deserve further analysis.

\subsection{Notations\label{sect notations}}

We denote by $\left\{  e_{n}\right\}  $ the complete orthonormal system in
$L^{2}\left(  \mathbb{T}^{2};\mathbb{C}\right)  $ given by $e_{n}\left(
x\right)  =e^{2\pi in\cdot x}$, $n\in\mathbb{Z}^{2}$. Given a distribution
$\omega\in C^{\infty}\left(  \mathbb{T}^{2}\right)  ^{\prime}$ and a test
function $\phi\in C^{\infty}\left(  \mathbb{T}^{2}\right)  $, we denoted by
$\left\langle \omega,\phi\right\rangle $ the duality between $\omega$ and
$\phi$ (namely $\omega\left(  \phi\right)  $), and we use the same symbol for
the inner product of $L^{2}\left(  \mathbb{T}^{2}\right)  $. We set
$\widehat{\omega}\left(  n\right)  =\left\langle \omega,e_{n}\right\rangle $,
$n\in\mathbb{Z}^{2}$ and we define, for each $s\in\mathbb{R}$, the space
$H^{s}\left(  \mathbb{T}^{2}\right)  $ as the space of all distributions
$\omega\in C^{\infty}\left(  \mathbb{T}^{2}\right)  ^{\prime}$ such that
\[
\left\Vert \omega\right\Vert _{H^{s}}^{2}:=\sum_{n\in\mathbb{Z}^{2}}\left(
1+\left\vert n\right\vert ^{2}\right)  ^{s}\left\vert \widehat{\omega}\left(
n\right)  \right\vert ^{2}<\infty.
\]
We use similar definitions and notations for the space $H^{s}\left(
\mathbb{T}^{2},\mathbb{C}\right)  $ of complex valued functions. In the space
$H^{-1-}\left(  \mathbb{T}^{2}\right)  =%
%TCIMACRO{\dbigcap \limits_{\epsilon>0}}%
%BeginExpansion
{\displaystyle\bigcap\limits_{\epsilon>0}}
%EndExpansion
H^{-1-\epsilon}\left(  \mathbb{T}^{2}\right)  $ we consider the metric%
\[
d_{H^{-1-}}\left(  \omega,\omega^{\prime}\right)  =\sum_{n=1}^{\infty}%
2^{-n}\left(  \left\Vert \omega-\omega^{\prime}\right\Vert _{H^{-1-\frac{1}%
{n}}}\wedge1\right)  .
\]
Convergence in this metric is equivalent to convergence in $H^{-1-\epsilon
}\left(  \mathbb{T}^{2}\right)  $ for every $\epsilon>0$. The space
$H^{-1-}\left(  \mathbb{T}^{2}\right)  $ with this metric is complete and
separable. We denote by $\mathcal{X}:=C\left(  \left[  0,T\right]
;H^{-1-}\left(  \mathbb{T}^{2}\right)  \right)  $ the space of continuous
functions with values in this metric space;\ a function is in $\mathcal{X}$ if
and only if it is in $C\left(  \left[  0,T\right]  ;H^{-1-\epsilon}\left(
\mathbb{T}^{2}\right)  \right)  $ for every $\epsilon>0$. The distance in
$C\left(  \left[  0,T\right]  ;H^{-1-}\left(  \mathbb{T}^{2}\right)  \right)
$ is given by $d_{\mathcal{X}}\left(  \omega_{\cdot},\omega_{\cdot}^{\prime
}\right)  =\sup_{t\in\left[  0,T\right]  }d_{H^{-1-}}\left(  \omega_{t}%
,\omega_{t}^{\prime}\right)  $, which makes $\mathcal{X}$ a Polish space.

For $s>0$, the spaces $H^{s}\left(  \mathbb{T}^{2}\right)  $ and
$H^{-s}\left(  \mathbb{T}^{2}\right)  $ are dual each other. By $H^{s+}\left(
\mathbb{T}^{2}\right)  $ we shall therefore mean the space $%
%TCIMACRO{\dbigcup \limits_{\epsilon>0}}%
%BeginExpansion
{\displaystyle\bigcup\limits_{\epsilon>0}}
%EndExpansion
H^{s+\epsilon}\left(  \mathbb{T}^{2}\right)  $. We shall use this notation in
the case of the space $H^{2+}\left(  \mathbb{T}^{2}\times\mathbb{T}%
^{2}\right)  $, which is similarly defined.

\section{White noise vorticity distribution and the nonlinear term in the weak
vorticity formulation\label{WN and nonlinear term}}

\subsection{White noise\label{subseq WN}}

We start recalling the well known notion of white noise, reviewing some of its
main properties used in the sequel.

White noise on $\mathbb{T}^{2}$ is by definition a Gaussian
distributional-valued stochastic process $\omega:\Xi\rightarrow C^{\infty
}\left(  \mathbb{T}^{2}\right)  ^{\prime}$, defined on some probability space
$\left(  \Xi,\mathcal{F},\mathbb{P}\right)  $, such that
\begin{equation}
\mathbb{E}\left[  \left\langle \omega,\phi\right\rangle \left\langle
\omega,\psi\right\rangle \right]  =\left\langle \phi,\psi\right\rangle
\label{def WN}%
\end{equation}
for all $\phi,\psi\in C^{\infty}\left(  \mathbb{T}^{2}\right)  $ (Gaussian
means that the real valued r.v. $\left\langle \omega,\phi\right\rangle $ is
Gaussian, for every $\phi\in C^{\infty}\left(  \mathbb{T}^{2}\right)  $). We
have denoted by $\left\langle \omega\left(  \theta\right)  ,\phi\right\rangle
$ the duality between the distribution $\omega\left(  \theta\right)  $ (for
some $\theta\in\Xi$)\ and the test function $\phi\in C^{\infty}\left(
\mathbb{T}^{2}\right)  $. These properties uniquely characterize the law of
$\omega$. In more heuristic terms, as it is often written in the Physics
literature,%
\[
\mathbb{E}\left[  \omega\left(  x\right)  \omega\left(  y\right)  \right]
=\delta\left(  x-y\right)
\]
since double integration of this identity against $\phi\left(  x\right)
\psi\left(  y\right)  $ gives (\ref{def WN}). White noise exists: it is
sufficient to take the complete orthonormal system $\left\{  e_{n}\right\}
_{n\in\mathbb{Z}^{2}}$ of $L^{2}\left(  \mathbb{T}^{2},\mathbb{C}\right)  $
introduced in Section \ref{sect notations}, a probability space $\left(
\Xi,\mathcal{F},\mathbb{P}\right)  $ supporting a sequence of independent
standard Gaussian variables $\left\{  G_{n}\right\}  _{n\in\mathbb{Z}^{2}}$,
and consider the series%
\[
\omega=\sqrt{2}\operatorname{Re}\sum_{n\in\mathbb{Z}^{2}}G_{n}e_{n}\text{.}%
\]
The partial sums $\omega_{N}^{\mathbb{C}}\left(  \theta,x\right)
=\sum_{\left\vert n\right\vert \leq N}G_{n}\left(  \theta\right)  e_{n}\left(
x\right)  $ are well defined complex valued random fields with square
integrable paths, $\omega_{N}:\Xi\rightarrow L^{2}\left(  \mathbb{T}%
^{2},\mathbb{C}\right)  $. For every $\epsilon>0$, $\left\{  \omega
_{N}^{\mathbb{C}}\right\}  _{N\in\mathbb{N}}$ is a Cauchy sequence in
$L^{2}\left(  \Xi;H^{-1-\epsilon}\left(  \mathbb{T}^{2},\mathbb{C}\right)
\right)  $, because
\[
\mathbb{E}\left[  \left\Vert \omega_{N}^{\mathbb{C}}\left(  \theta,x\right)
-\omega_{M}^{\mathbb{C}}\left(  \theta,x\right)  \right\Vert _{H^{-1-\epsilon
}}^{2}\right]  =\mathbb{E}\left[  \sum_{M<\left\vert n\right\vert \leq
N}\left(  1+\left\vert n\right\vert ^{2}\right)  ^{-1-\epsilon}\left\vert
G_{n}\right\vert ^{2}\right]  =\sum_{M<\left\vert n\right\vert \leq N}\left(
1+\left\vert n\right\vert ^{2}\right)  ^{-1-\epsilon}.
\]
The limit $\omega^{\mathbb{C}}$ in $L^{2}\left(  \Xi;H^{-1-\epsilon}\left(
\mathbb{T}^{2},\mathbb{C}\right)  \right)  $ thus exists, and $\omega=\sqrt
{2}\operatorname{Re}\omega^{\mathbb{C}}$ is a white noise because (doing
rigorously the computation on the partial sums and then taking the limit) it
is centered and for $\phi,\psi\in C^{\infty}\left(  \mathbb{T}^{2}\right)  $,
\begin{align*}
\mathbb{E}\left[  \left\langle \omega,\phi\right\rangle \left\langle
\omega,\psi\right\rangle \right]   & =\operatorname{Re}\mathbb{E}\left[
\left\langle \omega_{C},\phi\right\rangle \left\langle \overline{\omega_{C}%
},\psi\right\rangle \right]  =\operatorname{Re}\sum_{n,m\in\mathbb{Z}^{2}%
}\left\langle e_{n},\phi\right\rangle \overline{\left\langle e_{m}%
,\psi\right\rangle }\mathbb{E}\left[  G_{n}G_{m}\right] \\
& =\operatorname{Re}\sum_{n\in\mathbb{Z}^{2}}\left\langle e_{n},\phi
\right\rangle \overline{\left\langle e_{n},\psi\right\rangle }=\left\langle
\phi,\psi\right\rangle .
\end{align*}
[One obtains the same result by taking $\omega=\sum_{n\in\mathbb{Z}^{2}}%
G_{n}e_{n}$ where $\mathbb{Z}^{2}\backslash\left\{  0\right\}  $ is
partitioned as $\mathbb{Z}^{2}=\Lambda\cup\left(  -\Lambda\right)  $, $G_{n} $
are i.i.d. $N\left(  0,1\right)  $ on $\Lambda\cup\left\{  0\right\}  $ and
$G_{-n}=\overline{G_{n}}$ for $n\in\Lambda$.] The law $\mu$ of the measurable
map $\omega:\Xi\rightarrow H^{-1-\epsilon}\left(  \mathbb{T}^{2}\right)  $ is
a Gaussian measure (it is sufficient to check that $\left\langle \omega
,\phi\right\rangle $ is Gaussian for every $\phi\in C^{\infty}\left(
\mathbb{T}^{2}\right)  $, and this is true since $\left\langle \omega
,\phi\right\rangle $ is the $L^{2}\left(  \Xi\right)  $-limit of the Gaussian
variables $\sum_{\left\vert n\right\vert \leq N}G_{n}\left\langle e_{n}%
,\phi\right\rangle $). The measure $\mu$ is supported by $H^{-1-}\left(
\mathbb{T}^{2}\right)  $ but not by $H^{-1}\left(  \mathbb{T}^{2}\right)  $,
namely we have%
\[
\mu\left(  H^{-1}\left(  \mathbb{T}^{2}\right)  \right)  =0.
\]
It follows from
\[
\mathbb{E}\left[  \left\Vert \omega^{\mathbb{C}}\right\Vert _{H^{-1}}%
^{2}\right]  =\sum_{n\in\mathbb{Z}^{2}}\left(  1+\left\vert n\right\vert
^{2}\right)  ^{-1}=+\infty.
\]

The measure $\mu$ is sometimes denoted heuristically as%
\[
\mu\left(  d\omega\right)  =\frac{1}{Z}\exp\left(  -\frac{1}{2}\int%
_{\mathbb{T}^{2}}\omega^{2}dx\right)  d\omega
\]
and called the \textit{enstrophy measure}. The notation "$d\omega$" has no
meaning (unless interpreted as a limit of measures on finite dimensional
Euclidean spaces), just reminds the structure of centered nonsingular Gaussian
measures in $\mathbb{R}^{n}$, that is $\mu_{n}\left(  d\omega_{n}\right)
=\frac{1}{Z_{n}}\exp\left(  -\frac{1}{2}\left\langle Q_{n}^{-1}\omega
_{n},\omega_{n}\right\rangle _{\mathbb{R}^{n}}\right)  d\omega_{n}$ where
$d\omega_{n}$ is Lebesgue measure in $\mathbb{R}^{n}$ and $Q_{n}$ is the
covariance matrix. The notation $\int_{\mathbb{T}^{2}}\omega^{2}dx$ alludes to
the fact that $\mu$, heuristically considered as a Gaussian measure on
$L^{2}\left(  \mathbb{T}^{2}\right)  $ (this is not possible, $\mu\left(
L^{2}\left(  \mathbb{T}^{2}\right)  \right)  =0$), has covariance equal to the
identity: if $Q=Id$, then $\left\langle Q^{-1}\omega,\omega\right\rangle
_{L^{2}}=\int_{\mathbb{T}^{2}}\omega^{2}dx$. The fact that in $L^{2}\left(
\mathbb{T}^{2}\right)  $ the covariance operator $Q$, heuristically defined as%
\[
\left\langle Q\omega,\omega\right\rangle _{L^{2}}=\mathbb{E}\left[
\left\langle \omega,\phi\right\rangle _{L^{2}}\left\langle \omega
,\psi\right\rangle _{L^{2}}\right]
\]
is the identity in the case of the law $\mu$ of white noise, is a simple
"consequence" (the argument is not rigorous ab initio) of the definition
(\ref{def WN}) of white noise.

White noise realizations are in fact more regular than $H^{-1-}\left(
\mathbb{T}^{2}\right)  $. The general idea, used several times in
investigations of this kind, is that when a Gaussian field is $L^{2}$ it is
also more regular, because higher order moments are simply related to second
moments and Kolmogorov regularity theorem applies. Let us see this fact in the
case of white noise $\omega$. Given $\epsilon>0$, we know that $\omega\in
H^{-1-\epsilon}\left(  \mathbb{T}^{2}\right)  $ with probability one. Consider
the random field%
\[
\psi\left(  \theta,x\right)  :=\left(  \left(  1+\Delta\right)  ^{-\frac
{1+\epsilon}{2}}\omega\left(  \theta\right)  \right)  \left(  x\right)
=\left\langle \omega\left(  \theta\right)  ,\left(  1+\Delta\right)
^{-\frac{1+\epsilon}{2}}\delta_{x}\right\rangle .
\]
We have $\psi\in L^{2}\left(  \mathbb{T}^{2}\right)  $ with probability one.
But $\psi$ is a Gaussian field. We have in particular%
\[
\mathbb{E}\left[  \left\vert \psi\left(  x\right)  -\psi\left(  y\right)
\right\vert ^{p}\right]  \leq C_{p}\mathbb{E}\left[  \left\vert \psi\left(
x\right)  -\psi\left(  y\right)  \right\vert ^{2}\right]  ^{p/2}%
\]
and, denoting $\left(  1+\Delta\right)  ^{-\frac{1+\epsilon}{2}}\delta_{x}$
and $\left(  1+\Delta\right)  ^{-\frac{1+\epsilon}{2}}\delta_{y}$ respectively
by $f_{x}$, $f_{y}$,
\begin{align*}
& \mathbb{E}\left[  \left\vert \psi\left(  x\right)  -\psi\left(  y\right)
\right\vert ^{2}\right] \\
& =\mathbb{E}\left[  \psi\left(  x\right)  \psi\left(  x\right)  \right]
-2\mathbb{E}\left[  \psi\left(  x\right)  \psi\left(  y\right)  \right]
+\mathbb{E}\left[  \psi\left(  y\right)  ^{2}\right] \\
& =\mathbb{E}\left[  \left\langle \omega,f_{x}\right\rangle \left\langle
\omega,f_{x}\right\rangle \right]  -2\mathbb{E}\left[  \left\langle
\omega,f_{x}\right\rangle \left\langle \omega,f_{y}\right\rangle \right]
+\mathbb{E}\left[  \left\langle \omega,f_{y}\right\rangle \left\langle
\omega,f_{y}\right\rangle \right]
\end{align*}
and now we use definition (\ref{def WN})%
\begin{align*}
& =\left\langle f_{x},f_{x}\right\rangle -2\left\langle f_{x},f_{y}%
\right\rangle +\left\langle f_{y},f_{y}\right\rangle \\
& =\left\Vert f_{x}-f_{y}\right\Vert _{L^{2}}^{2}=\left\Vert \left(
1+\Delta\right)  ^{-\frac{1+\epsilon}{2}}\delta_{x}-\left(  1+\Delta\right)
^{-\frac{1+\epsilon}{2}}\delta_{y}\right\Vert _{L^{2}}^{2}\\
& \leq C_{\epsilon}\left\vert x-y\right\vert ^{\alpha\left(  \epsilon\right)
}%
\end{align*}
for a suitable number $\alpha\left(  \epsilon\right)  >0$ and a constant
$C_{\epsilon}>0$;\ the last inequality can be proved as
\begin{align*}
\sup_{\left\Vert \phi\right\Vert _{L^{2}}\leq1}\left\vert \left\langle \left(
1+\Delta\right)  ^{-\frac{1+\epsilon}{2}}\left(  \delta_{x}-\delta_{y}\right)
,\phi\right\rangle \right\vert  & =\sup_{\left\Vert \phi\right\Vert _{L^{2}%
}\leq1}\left\vert \left(  \left(  1+\Delta\right)  ^{-\frac{1+\epsilon}{2}%
}\phi\right)  \left(  x\right)  -\left(  \left(  1+\Delta\right)
^{-\frac{1+\epsilon}{2}}\phi\right)  \left(  y\right)  \right\vert \\
& \leq\sup_{\left\Vert \phi\right\Vert _{L^{2}}\leq1}\left\Vert \left(
1+\Delta\right)  ^{-\frac{1+\epsilon}{2}}\phi\right\Vert _{C^{\alpha\left(
\epsilon\right)  }}\left\vert x-y\right\vert ^{\alpha\left(  \epsilon\right)
}\\
& \leq C_{\epsilon}\sup_{\left\Vert \phi\right\Vert _{L^{2}}\leq1}\left\Vert
\phi\right\Vert _{L^{2}}\left\vert x-y\right\vert ^{\alpha\left(
\epsilon\right)  }%
\end{align*}
due to the fact that $H^{1+\epsilon}\left(  \mathbb{T}^{2}\right)  $ is
embedded in a space of H\"{o}lder continuous functions. Therefore%
\[
\mathbb{E}\left[  \left\vert \psi\left(  x\right)  -\psi\left(  y\right)
\right\vert ^{p}\right]  \leq C_{p}C_{\epsilon}^{p/2}\left\vert x-y\right\vert
^{\frac{p\alpha\left(  \epsilon\right)  }{2}}.
\]
Taking $p$ so large that $\frac{p\alpha\left(  \epsilon\right)  }{2}>2$, we
may apply Kolmogorov regularity theorem and deduce that the random field
$\psi\left(  x\right)  $ has a version with continuous paths. It means that,
up to a modification, $\left(  1+\Delta\right)  ^{-\frac{1+\epsilon}{2}}%
\omega\in C\left(  \mathbb{T}^{2}\right)  $ with probability one, not only
$\left(  1+\Delta\right)  ^{-\frac{1+\epsilon}{2}}\omega$ belongs to
$L^{2}\left(  \mathbb{T}^{2}\right)  $. Let us summarize this fact by the
notation%
\[
\mathbb{P}\left(  \omega\in H^{-1-,\infty}\left(  \mathbb{T}^{2}\right)
\right)  =1.
\]

In spite of this additional regularity, we are not in the realm of signed
measures, that received so much attention in the case of the vorticity of 2D
fluids. One has Among the properties, one has%
\[
\mu\left(  \mathcal{M}\left(  \mathbb{T}^{2}\right)  \right)  =0.
\]
(see \cite{Grotto}, Proposition A2 for a concise proof). To understand this
property, think to the analogy with the more classical 1-dimensional case, on
$[0,\infty$) instead of the torus. In such case, white noise is the
distributional derivative of Brownian motion. It is well known that, with
probability one, trajectories of a continuous version of Brownian motion are
not of bounded variation (because they have finite non zero quadratic
variation, and are continuous). Therefore their derivatives are not signed
measures. The gap in regularity between white noise and signed measures is
thus comparable to the gap between total variation and quadratic variation.

\subsection{Colored noise}

For technical reasons, sometimes it is convenient to consider a smooth
approximation of white noise. A simple one is $\omega_{N}\left(
\theta,x\right)  =\operatorname{Re}\sum_{\left\vert n\right\vert \leq N}%
G_{n}\left(  \theta\right)  e_{n}\left(  x\right)  $ but, although the
difference is really minor, for the PDE approach followed here the use of
mollifiers looks a bit more natural. We set, for $\epsilon>0$,%
\[
\omega_{\epsilon}\left(  x\right)  =\left\langle \omega,\theta_{\epsilon
}\left(  x-\cdot\right)  \right\rangle
\]
formally written also as $\left(  \theta_{\epsilon}\ast\omega\right)  \left(
x\right)  =\int_{\mathbb{T}^{2}}\theta_{\epsilon}\left(  x-y\right)
\omega\left(  y\right)  dy$, where $\theta_{\epsilon}\left(  x\right)
=\epsilon^{-2}\theta\left(  \epsilon^{-1}x\right)  $, and $\theta$ is a smooth
probability density on $\mathbb{T}^{2}$ with a small support around $x=0$.
Assume $\theta$ symmetric. We have%
\[
\mathbb{E}\left[  \left\langle \omega_{\epsilon},\phi\right\rangle
\left\langle \omega_{\epsilon},\psi\right\rangle \right]  =\mathbb{E}\left[
\left\langle \omega,\theta_{\epsilon}\ast\phi\right\rangle \left\langle
\omega,\theta_{\epsilon}\ast\psi\right\rangle \right]  =\left\langle
\theta_{\epsilon}\ast\phi,\theta_{\epsilon}\ast\psi\right\rangle
\]%
\begin{align*}
\mathbb{E}\left[  \omega_{\epsilon}\left(  x\right)  \omega_{\epsilon}\left(
y\right)  \right]   & =\mathbb{E}\left[  \left\langle \omega,\theta_{\epsilon
}\left(  x-\cdot\right)  \right\rangle \left\langle \omega,\theta_{\epsilon
}\left(  y-\cdot\right)  \right\rangle \right]  =\left\langle \theta
_{\epsilon}\left(  x-\cdot\right)  ,\theta_{\epsilon}\left(  y-\cdot\right)
\right\rangle \\
& =\int_{\mathbb{T}^{2}}\theta_{\epsilon}\left(  x-y-z\right)  \theta
_{\epsilon}\left(  z\right)  dz=\left(  \theta_{\epsilon}\ast\theta_{\epsilon
}\right)  \left(  x-y\right)  =:\delta_{x-y}^{\epsilon}%
\end{align*}
where we have used the notation $\delta_{a}^{\epsilon}$ to denote $\left(
\theta_{\epsilon}\ast\theta_{\epsilon}\right)  \left(  a\right)  $ because it
is an approximation of the Dirac delta distribution.

Notice that $\omega_{\epsilon}\in C^{\infty}\left(  \mathbb{T}^{2}\right)  $
with probability one. Moreover, since $\left\langle \omega_{\epsilon}%
,\phi\right\rangle =\left\langle \omega,\theta_{\epsilon}\ast\phi\right\rangle
$ and $\theta_{\epsilon}\ast\phi\rightarrow\phi$ in $H^{1+\gamma}\left(
\mathbb{T}^{2}\right)  $ for every $\phi\in H^{1+\gamma}\left(  \mathbb{T}%
^{2}\right)  $ and given $\gamma>0$, we have the following statement:

\begin{lemma}
\label{lemma convergence smooth noise}$\mathbb{P}$-almost surely, for every
$\phi\in H^{1+\gamma}\left(  \mathbb{T}^{2}\right)  $ we have
\[
\lim_{\epsilon\rightarrow0}\left\langle \omega_{\epsilon},\phi\right\rangle
=\left\langle \omega,\phi\right\rangle .
\]

\end{lemma}

\subsection{Weak vorticity formulation, preliminaries}

Let us first recall the weak vorticity formulation in the case of
measure-valued vorticities. First, one rewrites equation (\ref{Euler}) against
test functions $\phi\in C^{\infty}\left(  \mathbb{T}^{2}\right)  $, using
$\operatorname{div}u=0$:%

\[
\left\langle \omega_{t},\phi\right\rangle =\left\langle \omega_{0}%
,\phi\right\rangle +\int_{0}^{t}\left\langle \omega_{s},u_{s}\cdot\nabla
\phi\right\rangle ds.
\]
Then recall that Biot-Savart law gives us%
\[
u_{t}\left(  x\right)  =\int_{\mathbb{T}^{2}}K\left(  x-y\right)  \omega
_{t}\left(  dy\right)
\]
where $K\left(  x,y\right)  $ is the Biot-Savart kernel; in full space it is
given by $K\left(  x-y\right)  =\frac{1}{2\pi}\frac{\left(  x-y\right)
^{\perp}}{\left\vert x-y\right\vert ^{2}}$; on the torus its form is less
simple but we still have $K$ smooth for $x\neq y$, $K\left(  y-x\right)
=-K\left(  x-y\right)  $,
\[
\left\vert K\left(  x-y\right)  \right\vert \leq\frac{C}{\left\vert
x-y\right\vert }%
\]
for small values of $\left\vert x-y\right\vert $. See for instance
\cite{Shochet} for details. Thus we write the weak formulation in the more
explicit form%
\[
\left\langle \omega_{t},\phi\right\rangle =\left\langle \omega_{0}%
,\phi\right\rangle +\int_{0}^{t}\int_{\mathbb{T}^{2}}\int_{\mathbb{T}^{2}%
}K\left(  x-y\right)  \nabla\phi\left(  x\right)  \omega_{s}\left(  dx\right)
\omega_{s}\left(  dy\right)  ds.
\]
Since the double space integral, when we rename $x$ by $y$ and $y$ by $x$, is
the same (the renaming doesn't affect the value), and $K\left(  y-x\right)
=-K\left(  x-y\right)  $, we get (\ref{weak vorticity formulation}). Identity
(\ref{weak vorticity formulation}) is the weak vorticity formulation of Euler
equations. Depending on the assumptions on the measures $\omega_{s}$ (whether
or not they have concentrated masses), one has to specify the value of
$K\left(  0\right)  $, which is not given a priori, and thus the value of
$H_{\phi}\left(  x,x\right)  $; in the analysis of point vortices, for
instance, it is usually set equal to zero, to avoid self-interaction. The weak
vorticity formulation of Euler equations proved to be a fundamental tool in
the investigation of limits of solutions, especially in the context of
measures. Below we shall follow a similar path in the case of white noise
distributional solutions.

\subsection{The nonlinear term for white noise
vorticity\label{section nonlinear WN}}

Our purpose now is to define
\[
\int_{\mathbb{T}^{2}}\int_{\mathbb{T}^{2}}H_{\phi}\left(  x,y\right)
\omega\left(  x\right)  \omega\left(  y\right)  dxdy
\]
when $\omega:\Xi\rightarrow C^{\infty}\left(  \mathbb{T}^{2}\right)  ^{\prime
}$ is a white noise.

Preliminarily, notice that if $\omega\in C^{\infty}\left(  \mathbb{T}%
^{2}\right)  ^{\prime}$ is a distribution, we can define a distribution
$\omega\otimes\omega\in C^{\infty}\left(  \mathbb{T}^{2}\times\mathbb{T}%
^{2}\right)  ^{\prime}$ which satisfies%
\[
\left\langle \omega\otimes\omega,\phi\otimes\psi\right\rangle =\left\langle
\omega,\phi\right\rangle \left\langle \omega,\psi\right\rangle
\]
for all $\phi,\psi\in C^{\infty}\left(  \mathbb{T}^{2}\right)  $, where
$\phi\otimes\psi$ denotes the function $\left(  \phi\otimes\psi\right)
\left(  x,y\right)  =\phi\left(  x\right)  \psi\left(  y\right)  $. The
definition of $\omega\otimes\omega$ can be based on limits of test functions
of the form $\sum_{i=1}^{n}\phi_{i}\left(  x\right)  \psi_{i}\left(  y\right)
$, or more directly on the following argument. Given $f\in C^{\infty}\left(
\mathbb{T}^{2}\times\mathbb{T}^{2}\right)  $, for each $x\in\mathbb{T}^{2}$ we
have $f\left(  x,\cdot\right)  \in C^{\infty}\left(  \mathbb{T}^{2}\right)  $,
hence $\left\langle \omega,f\left(  x,\cdot\right)  \right\rangle $ is well
defined. The function $g\left(  x\right)  =\left\langle \omega,f\left(
x,\cdot\right)  \right\rangle $ belongs to $C^{\infty}\left(  \mathbb{T}%
^{2}\right)  $, as one can verify using the continuity properties of
distributions on test functions. Then we can set%
\begin{equation}
\left\langle \omega\otimes\omega,f\right\rangle =\left\langle \omega
,g\right\rangle \text{, \qquad where }g\left(  x\right)  =\left\langle
\omega,f\left(  x,\cdot\right)  \right\rangle .\label{def tensor distrib}%
\end{equation}
If $\omega\in H^{-s}\left(  \mathbb{T}^{2}\right)  $ for some $s>0$, one can
check that $\omega\otimes\omega\in H^{-2s}\left(  \mathbb{T}^{2}%
\times\mathbb{T}^{2}\right)  $.

Let us go back to white noise. First notice that, being $\omega\in
H^{-1-}\left(  \mathbb{T}^{2}\right)  $ with probability one, we have at
least
\[
\omega\otimes\omega\in H^{-2-}\left(  \mathbb{T}^{2}\times\mathbb{T}%
^{2}\right)  \text{ with probability one.}%
\]
Hence $\int_{\mathbb{T}^{2}}\int_{\mathbb{T}^{2}}f\left(  x,y\right)
\omega\left(  x\right)  \omega\left(  y\right)  dxdy$, or more properly the
duality
\[
\left\langle \omega\otimes\omega,f\right\rangle
\]
is well defined when $f\in H^{2+}\left(  \mathbb{T}^{2}\times\mathbb{T}%
^{2}\right)  $. The question is:\ can we define
\[
\left\langle \omega\otimes\omega,H_{\phi}\right\rangle
\]
for the function $H_{\phi}$, which is smooth outside the diagonal, and
bounded, but discontinuous along the diagonal and thus not of class $H^{2+}$?
We have the following results, over which all our analysis is based. The first
result is concerned with the smooth approximations $\omega_{\epsilon}\left(
x\right)  =\left\langle \omega,\theta_{\epsilon}\left(  x-\cdot\right)
\right\rangle $, the second one with white noise.

\begin{lemma}
i) If $\omega:\Xi\rightarrow C^{\infty}\left(  \mathbb{T}^{2}\right)
^{\prime}$ is a white noise and $f$ is bounded measurable on $\mathbb{T}%
^{2}\times\mathbb{T}^{2}$, then for every $p\geq1$ there is a constant
$C_{p}>0$ such that, for all $\epsilon>0$,
\[
\mathbb{E}\left[  \left\vert \left\langle \omega_{\epsilon}\otimes
\omega_{\epsilon},f\right\rangle \right\vert ^{p}\right]  \leq C_{p}\left\Vert
f\right\Vert _{\infty}^{p}.
\]

ii)\ We have $\mathbb{E}\left[  \left\langle \omega_{\epsilon}\otimes
\omega_{\epsilon},f\right\rangle \right]  =\int_{\mathbb{T}^{2}}%
\int_{\mathbb{T}^{2}}\delta_{x-y}^{\epsilon}f\left(  x,y\right)  dxdy$.

iii)\ If $f$ is symmetric, then%
\[
\mathbb{E}\left[  \left\vert \left\langle \omega_{\epsilon}\otimes
\omega_{\epsilon},f\right\rangle -\mathbb{E}\left[  \left\langle
\omega_{\epsilon}\otimes\omega_{\epsilon},f\right\rangle \right]  \right\vert
^{2}\right]  =2\int_{\left(  \mathbb{T}^{2}\right)  ^{4}}\delta_{x_{1}-x_{2}%
}^{\epsilon}\delta_{y_{1}-y_{2}}^{\epsilon}f\left(  x_{1},y_{1}\right)
f\left(  x_{2},y_{2}\right)  dx_{1}dy_{1}dx_{2}dy_{2}.
\]

\end{lemma}

\begin{proof}
i) It is sufficient to prove the claim for integer values of $p$. We have%
\[
\left\langle \omega_{\epsilon}\otimes\omega_{\epsilon},f\right\rangle
=\int_{\mathbb{T}^{2}}\int_{\mathbb{T}^{2}}\omega_{\epsilon}\left(  x\right)
\omega_{\epsilon}\left(  y\right)  f\left(  x,y\right)  dxdy
\]%
\[
\mathbb{E}\left[  \left\vert \left\langle \omega_{\epsilon}\otimes
\omega_{\epsilon},f\right\rangle \right\vert ^{p}\right]  =\int_{\left(
\mathbb{T}^{2}\right)  ^{2p}}\mathbb{E}\left[
%TCIMACRO{\dprod \limits_{i=1}^{p}}%
%BeginExpansion
{\displaystyle\prod\limits_{i=1}^{p}}
%EndExpansion
\left(  \omega_{\epsilon}\left(  x_{i}\right)  \omega_{\epsilon}\left(
y_{i}\right)  \right)  \right]
%TCIMACRO{\dprod \limits_{i=1}^{p}}%
%BeginExpansion
{\displaystyle\prod\limits_{i=1}^{p}}
%EndExpansion
f\left(  x_{i},y_{i}\right)  dx_{1}dy_{1}\cdot\cdot\cdot dx_{p}dy_{p}.
\]
From Isserlis-Wick theorem,
\[
\mathbb{E}\left[
%TCIMACRO{\dprod \limits_{i=1}^{p}}%
%BeginExpansion
{\displaystyle\prod\limits_{i=1}^{p}}
%EndExpansion
\left(  \omega_{\epsilon}\left(  x_{i}\right)  \omega_{\epsilon}\left(
y_{i}\right)  \right)  \right]  =\sum_{\pi}%
%TCIMACRO{\dprod \limits_{\left(  a,b\right)  \in\pi}}%
%BeginExpansion
{\displaystyle\prod\limits_{\left(  a,b\right)  \in\pi}}
%EndExpansion
\mathbb{E}\left[  \omega_{\epsilon}\left(  a\right)  \omega_{\epsilon}\left(
b\right)  \right]  =\sum_{\pi}%
%TCIMACRO{\dprod \limits_{\left(  a,b\right)  \in\pi}}%
%BeginExpansion
{\displaystyle\prod\limits_{\left(  a,b\right)  \in\pi}}
%EndExpansion
\delta_{a-b}^{\epsilon}%
\]
where the sum is over all partitions $\pi$ of $\left(  x_{1},y_{1}%
,...,x_{p},y_{p}\right)  $ in pairs, generically denoted by $\left(
a,b\right)  $. Therefore
\begin{align*}
\mathbb{E}\left[  \left\vert \left\langle \omega_{\epsilon}\otimes
\omega_{\epsilon},f\right\rangle \right\vert ^{p}\right]   & =\sum_{\pi}%
\int_{\left(  \mathbb{T}^{2}\right)  ^{2p}}%
%TCIMACRO{\dprod \limits_{\left(  a,b\right)  \in\pi}}%
%BeginExpansion
{\displaystyle\prod\limits_{\left(  a,b\right)  \in\pi}}
%EndExpansion
\delta_{a-b}^{\epsilon}%
%TCIMACRO{\dprod \limits_{i=1}^{p}}%
%BeginExpansion
{\displaystyle\prod\limits_{i=1}^{p}}
%EndExpansion
f\left(  x_{i},y_{i}\right)  dx_{1}dy_{1}\cdot\cdot\cdot dx_{p}dy_{p}\\
& \leq\left\Vert f\right\Vert _{\infty}^{p}\sum_{\pi}\int_{\left(
\mathbb{T}^{2}\right)  ^{2p}}%
%TCIMACRO{\dprod \limits_{\left(  a,b\right)  \in\pi}}%
%BeginExpansion
{\displaystyle\prod\limits_{\left(  a,b\right)  \in\pi}}
%EndExpansion
\delta_{a-b}^{\epsilon}dx_{1}dy_{1}\cdot\cdot\cdot dx_{p}dy_{p}\\
& =\left\Vert f\right\Vert _{\infty}^{p}\sum_{\pi}\left(  \int_{\mathbb{T}%
^{2}}\int_{\mathbb{T}^{2}}\delta_{a-b}^{\epsilon}dadb\right)  ^{p}\\
& =\left\Vert f\right\Vert _{\infty}^{p}\sum_{\pi}\left(  \int_{\mathbb{T}%
^{2}}\int_{\mathbb{T}^{2}}\left\langle \theta_{\epsilon}\left(  a-\cdot
\right)  ,\theta_{\epsilon}\left(  b-\cdot\right)  \right\rangle dadb\right)
^{p}\\
& =\left\Vert f\right\Vert _{\infty}^{p}\sum_{\pi}\left\vert \mathbb{T}%
^{2}\right\vert ^{p}=:C_{p}\left\Vert f\right\Vert _{\infty}^{p}%
\end{align*}
(the sum has $\left(  2p\right)  !/\left(  2^{p}p!\right)  $ terms).

ii) We simply have%
\[
\mathbb{E}\left[  \left\langle \omega_{\epsilon}\otimes\omega_{\epsilon
},f\right\rangle \right]  =\int_{\mathbb{T}^{2}}\int_{\mathbb{T}^{2}%
}\mathbb{E}\left[  \omega_{\epsilon}\left(  x\right)  \omega_{\epsilon}\left(
y\right)  \right]  f\left(  x,y\right)  dxdy=\int_{\mathbb{T}^{2}}%
\int_{\mathbb{T}^{2}}\delta_{x-y}^{\epsilon}f\left(  x,y\right)  dxdy.
\]

iii)\ We just develop more carefully%
\[
\mathbb{E}\left[  \left\langle \omega_{\epsilon}\otimes\omega_{\epsilon
},f\right\rangle ^{2}\right]  =\int_{\left(  \mathbb{T}^{2}\right)  ^{4}%
}\mathbb{E}\left[
%TCIMACRO{\dprod \limits_{i=1}^{2}}%
%BeginExpansion
{\displaystyle\prod\limits_{i=1}^{2}}
%EndExpansion
\left(  \omega_{\epsilon}\left(  x_{i}\right)  \omega_{\epsilon}\left(
y_{i}\right)  \right)  \right]
%TCIMACRO{\dprod \limits_{i=1}^{2}}%
%BeginExpansion
{\displaystyle\prod\limits_{i=1}^{2}}
%EndExpansion
f\left(  x_{i},y_{i}\right)  dx_{1}dy_{1}dx_{2}dy_{2}.
\]
We have, again from Isserlis-Wick theorem,
\begin{align*}
\mathbb{E}\left[
%TCIMACRO{\dprod \limits_{i=1}^{2}}%
%BeginExpansion
{\displaystyle\prod\limits_{i=1}^{2}}
%EndExpansion
\left(  \omega_{\epsilon}\left(  x_{i}\right)  \omega_{\epsilon}\left(
y_{i}\right)  \right)  \right]   & =\mathbb{E}\left[  \omega_{\epsilon}\left(
x_{1}\right)  \omega_{\epsilon}\left(  x_{2}\right)  \right]  \mathbb{E}%
\left[  \omega_{\epsilon}\left(  y_{1}\right)  \omega_{\epsilon}\left(
y_{2}\right)  \right] \\
& +\mathbb{E}\left[  \omega_{\epsilon}\left(  x_{1}\right)  \omega_{\epsilon
}\left(  y_{2}\right)  \right]  \mathbb{E}\left[  \omega_{\epsilon}\left(
y_{1}\right)  \omega_{\epsilon}\left(  x_{2}\right)  \right] \\
& +\mathbb{E}\left[  \omega_{\epsilon}\left(  x_{1}\right)  \omega_{\epsilon
}\left(  y_{1}\right)  \right]  \mathbb{E}\left[  \omega_{\epsilon}\left(
x_{2}\right)  \omega_{\epsilon}\left(  y_{2}\right)  \right]
\end{align*}%
\[
=\delta_{x_{1}-x_{2}}^{\epsilon}\delta_{y_{1}-y_{2}}^{\epsilon}+\delta
_{x_{1}-y_{2}}^{\epsilon}\delta_{y_{1}-x_{2}}^{\epsilon}+\delta_{x_{1}-y_{1}%
}^{\epsilon}\delta_{x_{2}-y_{2}}^{\epsilon}.
\]
Hence, using the symmetry,%
\begin{align*}
& \mathbb{E}\left[  \left\langle \omega_{\epsilon}\otimes\omega_{\epsilon
},f\right\rangle ^{2}\right] \\
& =\int_{\left(  \mathbb{T}^{2}\right)  ^{4}}\left(  \delta_{x_{1}-x_{2}%
}^{\epsilon}\delta_{y_{1}-y_{2}}^{\epsilon}+\delta_{x_{1}-y_{2}}^{\epsilon
}\delta_{y_{1}-x_{2}}^{\epsilon}+\delta_{x_{1}-y_{1}}^{\epsilon}\delta
_{x_{2}-y_{2}}^{\epsilon}\right)  f\left(  x_{1},y_{1}\right)  f\left(
x_{2},y_{2}\right)  dx_{1}dy_{1}dx_{2}dy_{2}\\
& =2\int_{\left(  \mathbb{T}^{2}\right)  ^{4}}\delta_{x_{1}-x_{2}}^{\epsilon
}\delta_{y_{1}-y_{2}}^{\epsilon}f\left(  x_{1},y_{1}\right)  f\left(
x_{2},y_{2}\right)  dx_{1}dy_{1}dx_{2}dy_{2}+\left(  \int_{\mathbb{T}^{2}}%
\int_{\mathbb{T}^{2}}\delta_{x-y}^{\epsilon}f\left(  x,y\right)  dxdy\right)
^{2}.
\end{align*}
We have found%
\[
\mathbb{E}\left[  \left\langle \omega_{\epsilon}\otimes\omega_{\epsilon
},f\right\rangle ^{2}\right]  -\mathbb{E}\left[  \left\langle \omega
_{\epsilon}\otimes\omega_{\epsilon},f\right\rangle \right]  ^{2}%
=2\int_{\left(  \mathbb{T}^{2}\right)  ^{4}}\delta_{x_{1}-x_{2}}^{\epsilon
}\delta_{y_{1}-y_{2}}^{\epsilon}f\left(  x_{1},y_{1}\right)  f\left(
x_{2},y_{2}\right)  dx_{1}dy_{1}dx_{2}dy_{2}.
\]

\end{proof}

\begin{corollary}
\label{corollary WN}i) If $\omega:\Xi\rightarrow C^{\infty}\left(
\mathbb{T}^{2}\right)  ^{\prime}$ is a white noise and $f\in H^{2+}\left(
\mathbb{T}^{2}\times\mathbb{T}^{2}\right)  $, then for every $p\geq1$ there is
a constant $C_{p}>0$ such that
\[
\mathbb{E}\left[  \left\vert \left\langle \omega\otimes\omega,f\right\rangle
\right\vert ^{p}\right]  \leq C_{p}\left\Vert f\right\Vert _{\infty}^{p}.
\]

ii)\ We have $\mathbb{E}\left[  \left\langle \omega\otimes\omega
,f\right\rangle \right]  =\int_{\mathbb{T}^{2}}f\left(  x,x\right)  dx$.

iii)\ If $f$ is symmetric, then%
\[
\mathbb{E}\left[  \left\vert \left\langle \omega\otimes\omega,f\right\rangle
-\mathbb{E}\left[  \left\langle \omega\otimes\omega,f\right\rangle \right]
\right\vert ^{2}\right]  =2\int_{\mathbb{T}^{2}}\int_{\mathbb{T}^{2}}f\left(
x,y\right)  ^{2}dxdy.
\]

\end{corollary}

\begin{proof}
Notice that $f$ is continuous and thus bounded and uniformly continuous, on
$\mathbb{T}^{2}$, by Sobolev embedding theorem. Thus we may apply the previous
lemma to $\left\langle \omega_{\epsilon}\otimes\omega_{\epsilon}%
,f\right\rangle $; and we have%
\begin{align*}
\lim_{\epsilon\rightarrow0}\int_{\mathbb{T}^{2}}\int_{\mathbb{T}^{2}}%
\delta_{x-y}^{\epsilon}f\left(  x,y\right)  dxdy  & =\int_{\mathbb{T}^{2}%
}f\left(  x,x\right)  dx\\
\lim_{\epsilon\rightarrow0}\int_{\left(  \mathbb{T}^{2}\right)  ^{4}}%
\delta_{x_{1}-x_{2}}^{\epsilon}\delta_{y_{1}-y_{2}}^{\epsilon}f\left(
x_{1},y_{1}\right)  f\left(  x_{2},y_{2}\right)  dx_{1}dy_{1}dx_{2}dy_{2}  &
=\int_{\mathbb{T}^{2}}\int_{\mathbb{T}^{2}}f\left(  x_{1},y_{1}\right)
^{2}dx_{1}dy_{1}.
\end{align*}
From the identity%
\[
\left\langle \omega_{\epsilon}\otimes\omega_{\epsilon},f\right\rangle
=\int_{\mathbb{T}^{2}}\int_{\mathbb{T}^{2}}\omega_{\epsilon}\left(  x\right)
\omega_{\epsilon}\left(  y\right)  f\left(  x,y\right)  dxdy=\left\langle
\omega\otimes\omega,\left(  \theta_{\epsilon}\otimes\theta_{\epsilon}\right)
\ast f\right\rangle
\]
we see that $\mathbb{P}$-almost surely, for every $f\in H^{2+}\left(
\mathbb{T}^{2}\times\mathbb{T}^{2}\right)  $ we have
\[
\lim_{\epsilon\rightarrow0}\left\langle \omega_{\epsilon}\otimes
\omega_{\epsilon},f\right\rangle =\left\langle \omega\otimes\omega
,f\right\rangle .
\]
We can pass to the limit in all expectations written in the statement of the
corollary, due to uniform integrability of $\left\vert \left\langle
\omega_{\epsilon}\otimes\omega_{\epsilon},f\right\rangle \right\vert $ (Vitali
theorem), coming from property (i) of the lemma. The corollary then follows
from these limit properties and the lemma.
\end{proof}

\begin{remark}
In the non symmetric case we simply have%
\[
\mathbb{E}\left[  \left\vert \left\langle \omega\otimes\omega,f\right\rangle
-\mathbb{E}\left[  \left\langle \omega\otimes\omega,f\right\rangle \right]
\right\vert ^{2}\right]  =\int\int f^{2}\left(  x,y\right)  dxdy+\int%
_{\mathbb{T}^{2}}\int_{\mathbb{T}^{2}}f\left(  x,y\right)  f\left(
y,x\right)  dxdy.
\]

\end{remark}

Based on the previous key facts we can give a definition of $\left\langle
\omega\otimes\omega,H_{\phi}\right\rangle $ when $\omega$ is white noise.

\begin{theorem}
\label{Thm Cauchy}Let $\omega:\Xi\rightarrow C^{\infty}\left(  \mathbb{T}%
^{2}\right)  ^{\prime}$ be a white noise and $\phi\in C^{\infty}\left(
\mathbb{T}^{2}\right)  $ be given. Assume that $H_{\phi}^{n}\in H^{2+}\left(
\mathbb{T}^{2}\times\mathbb{T}^{2}\right)  $ are symmetric and approximate
$H_{\phi}$ in the following sense:%
\begin{align*}
\lim_{n\rightarrow\infty}\int\int\left(  H_{\phi}^{n}-H_{\phi}\right)
^{2}\left(  x,y\right)  dxdy  & =0\\
\lim_{n\rightarrow\infty}\int H_{\phi}^{n}\left(  x,x\right)  dx  & =0.
\end{align*}
Then the sequence of r.v.'s $\left\langle \omega\otimes\omega,H_{\phi}%
^{n}\right\rangle $ is a Cauchy sequence in mean square. We denote by
\[
\left\langle \omega\otimes\omega,H_{\phi}\right\rangle
\]
its limit. Moreover, the limit is the same if $H_{\phi}^{n}$ is replaced by
$\widetilde{H}_{\phi}^{n}$ with the same properties and such that
$\lim_{n\rightarrow\infty}\int\int\left(  H_{\phi}^{n}-\widetilde{H}_{\phi
}^{n}\right)  ^{2}\left(  x,y\right)  dxdy=0$.
\end{theorem}

\begin{proof}
Since $\lim_{n\rightarrow\infty}\int H_{\phi}^{n}\left(  x,x\right)  dx=0$, it
is equivalent to show that $\left\langle \omega\otimes\omega,H_{\phi}%
^{n}\right\rangle -\int H_{\phi}^{n}\left(  x,x\right)  dx$ is a Cauchy
sequence in mean square. We have%
\begin{align*}
& \mathbb{E}\left[  \left\vert \left\langle \omega\otimes\omega,H_{\phi}%
^{n}\right\rangle -\int H_{\phi}^{n}\left(  x,x\right)  dx-\left\langle
\omega\otimes\omega,H_{\phi}^{m}\right\rangle +\int H_{\phi}^{m}\left(
x,x\right)  dx\right\vert ^{2}\right] \\
& =\mathbb{E}\left[  \left\vert \left\langle \omega\otimes\omega,\left(
H_{\phi}^{n}-H_{\phi}^{m}\right)  \right\rangle -\int\left(  H_{\phi}%
^{n}-H_{\phi}^{m}\right)  \left(  x,x\right)  dx\right\vert ^{2}\right]
\end{align*}
and now we use properties (ii-iii)\ of the Corollary%
\[
=2\int_{\mathbb{T}^{2}}\int_{\mathbb{T}^{2}}\left(  H_{\phi}^{n}-H_{\phi}%
^{m}\right)  ^{2}\left(  x,y\right)  dxdy.
\]
Due to our assumption, this implies the Cauchy property. Hence $\left\langle
\omega\otimes\omega,H_{\phi}\right\rangle $ is well defined. The invariance
property is prove in a similar way.
\end{proof}

\begin{remark}
\label{remark existence}It is easy to construct a sequence $H_{\phi}%
^{n}\left(  x,y\right)  $ with the properties above. Recall that $H_{\phi
}\left(  x,y\right)  :=\frac{1}{2}K\left(  x-y\right)  \left(  \nabla
\phi\left(  x\right)  -\nabla\phi\left(  y\right)  \right)  $, where $K$
smooth for $x\neq y$, $K\left(  y-x\right)  =-K\left(  x-y\right)  $,
\[
\left\vert K\left(  x-y\right)  \right\vert \leq\frac{C}{\left\vert
x-y\right\vert }%
\]
for small values of $\left\vert x-y\right\vert $. We set, for $\epsilon>0$,%
\begin{equation}
K_{\epsilon}\left(  x\right)  =\left\{
\begin{array}
[c]{ccc}%
K\left(  x\right)  \left(  1-\theta_{\epsilon}\left(  x\right)  \right)  &
\text{for} & x\neq0\\
0 & \text{for} & x=0
\end{array}
\right. \label{def K eps}%
\end{equation}
where $\theta_{\epsilon}\left(  x\right)  =\theta\left(  \epsilon
^{-1}x\right)  $, $0\leq\theta\leq1$, $\theta$ is smooth, with support a small
ball $B\left(  0,r\right)  $, equal to 1 in $B\left(  0,r/2\right)  $; and,
given any sequence $\epsilon_{n}\rightarrow0$ we set%
\[
H_{\phi}^{n}\left(  x,y\right)  =\frac{1}{2}K_{\epsilon_{n}}\left(
x-y\right)  \left(  \nabla\phi\left(  x\right)  -\nabla\phi\left(  y\right)
\right)  .
\]
Then $H_{\phi}^{n}$ is smooth; $H_{\phi}^{n}\left(  x,x\right)  =0$ hence
$\int H_{\phi}^{n}\left(  x,x\right)  dx=0$; and \
\begin{align*}
\lim_{n\rightarrow\infty}\int\int\left(  H_{\phi}^{n}-H_{\phi}\right)
^{2}\left(  x,y\right)  dxdy  & =\lim_{n\rightarrow\infty}\int\int H_{\phi
}^{2}\left(  x,y\right)  \theta_{\epsilon_{n}}^{2}\left(  x-y\right)  dxdy\\
& \leq\lim_{n\rightarrow\infty}\int\int_{\left\vert x-y\right\vert
\leq\epsilon_{n}r}H_{\phi}^{2}\left(  x,y\right)  dxdy=0
\end{align*}
(because $H_{\phi}^{2}\left(  x,y\right)  $ is bounded above, $\theta
_{\epsilon_{n}}^{2}\leq1$, and $\theta_{\epsilon_{n}}^{2}\neq0$ only in
$B\left(  0,\epsilon_{n}r\right)  $).
\end{remark}

In fact, what we need in Definition \ref{def WN sol}\ below is a definition of
$\int_{0}^{t}\left\langle \omega_{s}\otimes\omega_{s},H_{\phi}\right\rangle
ds$ and for such purpose the previous result is not so strong; it would allow
for instance to define such integral as a Bochner integral in the Hilbert
space $L^{2}\left(  \Xi\right)  $. We prefer to have a stronger meaning and
for this purpose we refine the previous result.

\begin{theorem}
Let $\omega_{\cdot}:\Xi\times\left[  0,T\right]  \rightarrow C^{\infty}\left(
\mathbb{T}^{2}\right)  ^{\prime}$ be a measurable map with trajectories of
class $C\left(  \left[  0,T\right]  ;H^{-1-}\right)  $. Assume that
$\omega_{t}$ is a white noise at every time $t\in\left[  0,T\right]  $. Let
$H_{\phi}^{n}$ be an approximation of $H_{\phi}$ as above, of class
$H^{2+}\left(  \mathbb{T}^{2}\times\mathbb{T}^{2}\right)  $. Then the well
defined sequence of real valued process $\left\{  s\mapsto\left\langle
\omega_{s}\otimes\omega_{s},H_{\phi}^{n}\right\rangle ;s\in\left[  0,T\right]
\right\}  _{n\in\mathbb{N}}$ is a Cauchy sequence in $L^{2}\left(  \Xi
;L^{2}\left(  0,T\right)  \right)  $.
\end{theorem}

\begin{proof}
The proof is the same as the one of Theorem \ref{Thm Cauchy}, but we repeat
it, due to the importance of the present result. We have%
\begin{align*}
& \mathbb{E}\left[  \int_{0}^{T}\left\vert \left\langle \omega_{s}%
\otimes\omega_{s},H_{\phi}^{n}\right\rangle -\int H_{\phi}^{n}\left(
x,x\right)  dx-\left\langle \omega_{s}\otimes\omega_{s},H_{\phi}%
^{m}\right\rangle +\int H_{\phi}^{m}\left(  x,x\right)  dx\right\vert
^{2}ds\right] \\
& =\int_{0}^{T}\mathbb{E}\left[  \int_{0}^{T}\left\vert \left\langle
\omega_{s}\otimes\omega_{s},H_{\phi}^{n}\right\rangle -\int H_{\phi}%
^{n}\left(  x,x\right)  dx-\left\langle \omega_{s}\otimes\omega_{s},H_{\phi
}^{m}\right\rangle +\int H_{\phi}^{m}\left(  x,x\right)  dx\right\vert
^{2}\right]  ds\\
& =T\cdot\mathbb{E}\left[  \left\vert \left\langle \omega_{0}\otimes\omega
_{0},\left(  H_{\phi}^{n}-H_{\phi}^{m}\right)  \right\rangle -\int\left(
H_{\phi}^{n}-H_{\phi}^{m}\right)  \left(  x,x\right)  dx\right\vert
^{2}\right]
\end{align*}
and now we use properties (ii-iii)\ of the Corollary%
\[
=2\int_{\mathbb{T}^{2}}\int_{\mathbb{T}^{2}}\left(  H_{\phi}^{n}-H_{\phi}%
^{m}\right)  ^{2}\left(  x,y\right)  dxdy.
\]
Due to our assumption, this implies the Cauchy property.
\end{proof}

\begin{definition}
\label{Def nonlin in t}Under the assumptions of the previous theorem, we
denote by
\[
\left\{  s\mapsto\left\langle \omega_{s}\otimes\omega_{s},H_{\phi
}\right\rangle ;s\in\left[  0,T\right]  \right\}
\]
or more simply by $\left\langle \omega_{\cdot}\otimes\omega_{\cdot},H_{\phi
}\right\rangle $ the process of class $L^{2}\left(  \Xi;L^{2}\left(
0,T\right)  \right)  $, limit of the sequence $\left\{  s\mapsto\left\langle
\omega_{s}\otimes\omega_{s},H_{\phi}^{n}\right\rangle ;s\in\left[  0,T\right]
\right\}  _{n\in\mathbb{N}}$.
\end{definition}

\begin{remark}
By the identification $L^{2}\left(  \Xi;L^{2}\left(  0,T\right)  \right)
=L^{2}\left(  0,T;L^{2}\left(  \Xi\right)  \right)  $, we may see
$\left\langle \omega_{\cdot}\otimes\omega_{\cdot},H_{\phi}\right\rangle $ as
an element of the class $L^{2}\left(  0,T;L^{2}\left(  \Xi\right)  \right)  $;
its value at time $s$ is, for a.e. $s$, an element of $L^{2}\left(
\Xi\right)  $; one may check that it is the same element of $L^{2}\left(
\Xi\right)  $ given by Theorem \ref{Thm Cauchy}.
\end{remark}

\subsection{The nonlinear term for modified white noise vorticity}

We may generalize a little bit the previous construction. Assume $\omega
:\Xi\rightarrow C^{\infty}\left(  \mathbb{T}^{2}\right)  ^{\prime}$ is a
random distribution with the property that
\[
\mathbb{E}\left[  \Phi\left(  \omega\right)  \right]  =\mathbb{E}\left[
\rho\left(  \omega_{WN}\right)  \Phi\left(  \omega_{WN}\right)  \right]
\]
for every measurable function $\Phi:H^{-1-}\left(  \mathbb{T}^{2}\right)
\rightarrow\lbrack0,\infty)$, where $\omega_{WN}:\Xi\rightarrow C^{\infty
}\left(  \mathbb{T}^{2}\right)  ^{\prime}$ is a white noise and $\rho
:H^{-1-}\left(  \mathbb{T}^{2}\right)  \rightarrow\lbrack0,\infty)$ is a
measurable function such that%
\[
k_{q}:=\mathbb{E}\left[  \rho^{q}\left(  \omega_{WN}\right)  \right]  <\infty
\]
for some $q>1$, and $\int\rho d\mu=1$. This is equivalent to say that the law
of $\omega$ is absolutely continuous with respect to $\mu$ with density $\rho$
satisfying $\int\rho^{q}d\mu<\infty$.

\begin{lemma}
Under the previous assumptions, if $f\in H^{2+}\left(  \mathbb{T}^{2}%
\times\mathbb{T}^{2}\right)  $, then:

i) for every $r\geq1$ there is a constant $C_{r}>0$ such that
\[
\mathbb{E}\left[  \left\vert \left\langle \omega\otimes\omega,f\right\rangle
\right\vert ^{r}\right]  \leq C_{r}\left\Vert f\right\Vert _{\infty}^{r}.
\]

ii)\ If $f$ is symmetric, then there exists a constant $C_{q}>0$ such that
\[
\mathbb{E}\left[  \left\vert \left\langle \omega\otimes\omega,f\right\rangle
-\int_{\mathbb{T}^{2}}f\left(  x,x\right)  dx\right\vert \right]  \leq
C_{q}\left\Vert f\right\Vert _{L^{2}\left(  \mathbb{T}^{2}\times\mathbb{T}%
^{2}\right)  }^{1/p}%
\]
where $p$ is the number such that $\frac{1}{p}+\frac{1}{q}=1$.
\end{lemma}

\begin{proof}
i) We deduce the claim from
\begin{align*}
\mathbb{E}\left[  \left\vert \left\langle \omega\otimes\omega,f\right\rangle
\right\vert ^{r}\right]   & =\mathbb{E}\left[  \rho\left(  \omega_{WN}\right)
\left\vert \left\langle \omega_{WN}\otimes\omega_{WN},f\right\rangle
\right\vert ^{r}\right] \\
& \leq\mathbb{E}\left[  \rho^{q}\left(  \omega_{WN}\right)  \right]
^{1/q}\mathbb{E}\left[  \left\vert \left\langle \omega_{WN}\otimes\omega
_{WN},f\right\rangle \right\vert ^{rp}\right]  ^{1/p}.
\end{align*}

ii) One has
\begin{align*}
\mathbb{E}\left[  \left\vert \left\langle \omega\otimes\omega,f\right\rangle
-\int_{\mathbb{T}^{2}}f\left(  x,x\right)  dx\right\vert \right]   &
=\mathbb{E}\left[  \rho\left(  \omega_{WN}\right)  \left\vert \left\langle
\omega_{WN}\otimes\omega_{WN},f\right\rangle -\int_{\mathbb{T}^{2}}f\left(
x,x\right)  dx\right\vert \right] \\
& \leq\mathbb{E}\left[  \rho^{q}\left(  \omega_{WN}\right)  \right]
^{1/q}\mathbb{E}\left[  \left\vert \left\langle \omega_{WN}\otimes\omega
_{WN},\phi\right\rangle -\int_{\mathbb{T}^{2}}f\left(  x,x\right)
dx\right\vert ^{p}\right]  ^{1/p}.
\end{align*}
Moreover,%
\begin{align*}
& \mathbb{E}\left[  \left\vert \left\langle \omega_{WN}\otimes\omega
_{WN},f\right\rangle -\int_{\mathbb{T}^{2}}f\left(  x,x\right)  dx\right\vert
^{p}\right] \\
& \leq\mathbb{E}\left[  \left\vert \left\langle \omega_{WN}\otimes\omega
_{WN},f\right\rangle -\int_{\mathbb{T}^{2}}f\left(  x,x\right)  dx\right\vert
^{2}\right]  ^{1/2}\mathbb{E}\left[  \left\vert \left\langle \omega
_{WN}\otimes\omega_{WN},f\right\rangle -\int_{\mathbb{T}^{2}}f\left(
x,x\right)  dx\right\vert ^{2p-2}\right]  ^{1/2}\\
& =C_{q}^{0}\left(  \int_{\mathbb{T}^{2}}\int_{\mathbb{T}^{2}}f\left(
x,y\right)  ^{2}dxdy\right)  ^{1/2}%
\end{align*}
where%
\[
C_{q}^{0}:=2\mathbb{E}\left[  \left\vert \left\langle \omega_{WN}\otimes
\omega_{WN},f\right\rangle -\int_{\mathbb{T}^{2}}f\left(  x,x\right)
dx\right\vert ^{2p-2}\right]  ^{1/2}%
\]
is a finite constant, due to property (i) of a previous corollary. We set
$C_{q}=k_{q}^{1/q}\left(  C_{q}^{0}\right)  ^{1/p}$.
\end{proof}

The next results are the same as those above in the white noise case except
that we have a lower order of integrability, nevertheless sufficient for our aims.

\begin{theorem}
Under the previous assumptions, assume that $H_{\phi}^{n}\in H^{2+}\left(
\mathbb{T}^{2}\times\mathbb{T}^{2}\right)  $ are symmetric and approximate
$H_{\phi}$ as in Theorem \ref{Thm Cauchy}. Then the sequence of r.v.'s
$\left\langle \omega\otimes\omega,H_{\phi}^{n}\right\rangle $ is a Cauchy
sequence in $L^{1}\left(  \Xi\right)  $. We denote by $\left\langle
\omega\otimes\omega,H_{\phi}\right\rangle $ its limit. It is the same if
$H_{\phi}^{n}$ is replaced by $\widetilde{H}_{\phi}^{n}$ with the properties
described in Theorem \ref{Thm Cauchy}.
\end{theorem}

\begin{proof}
Since $\lim_{n\rightarrow\infty}\int H_{\phi}^{n}\left(  x,x\right)  dx=0$, it
is equivalent to show that $\left\langle \omega\otimes\omega,H_{\phi}%
^{n}\right\rangle -\int H_{\phi}^{n}\left(  x,x\right)  dx$ is a Cauchy
sequence in $L^{1}\left(  \Xi\right)  $. We have%
\begin{align*}
& \mathbb{E}\left[  \left\vert \left\langle \omega\otimes\omega,H_{\phi}%
^{n}\right\rangle -\int H_{\phi}^{n}\left(  x,x\right)  dx-\left\langle
\omega\otimes\omega,H_{\phi}^{m}\right\rangle +\int H_{\phi}^{m}\left(
x,x\right)  dx\right\vert \right] \\
& =\mathbb{E}\left[  \left\vert \left\langle \omega\otimes\omega,\left(
H_{\phi}^{n}-H_{\phi}^{m}\right)  \right\rangle -\int\left(  H_{\phi}%
^{n}-H_{\phi}^{m}\right)  \left(  x,x\right)  dx\right\vert \right]
\end{align*}
and now we use property (ii)\ of the Corollary%
\[
\leq C_{q}\left\Vert H_{\phi}^{n}-H_{\phi}^{m}\right\Vert _{L^{2}\left(
\mathbb{T}^{2}\times\mathbb{T}^{2}\right)  }^{1/p}.
\]
Due to our assumptions, this implies the Cauchy property. Hence $\left\langle
\omega\otimes\omega,H_{\phi}\right\rangle $ is well defined. The invariance
property is prove in a similar way.
\end{proof}

\begin{theorem}
Let $\rho:\left[  0,T\right]  \times H^{-1-}\left(  \mathbb{T}^{2}\right)
\rightarrow\lbrack0,\infty)$ be a function such that $\int\rho_{t}^{q}d\mu\leq
C$ for some constants $C>0$, $q>1$, where $\mu$ is the law of white noise; and
$\int\rho_{t}d\mu=1$ for every $t\in\left[  0,T\right]  $. Let $\omega_{\cdot
}:\Xi\times\left[  0,T\right]  \rightarrow C^{\infty}\left(  \mathbb{T}%
^{2}\right)  ^{\prime}$ be a measurable map with trajectories of class
$C\left(  \left[  0,T\right]  ;H^{-1-}\right)  $. Assume that the law of
$\omega_{t}$ is $\rho_{t}d\mu$, at every time $t\in\left[  0,T\right]  $. Let
$H_{\phi}^{n}$ be an approximation of $H_{\phi}$ as above, of class
$H^{2+}\left(  \mathbb{T}^{2}\times\mathbb{T}^{2}\right)  $. Then the well
defined sequence of real valued process $\left\{  s\mapsto\left\langle
\omega_{s}\otimes\omega_{s},H_{\phi}^{n}\right\rangle ;s\in\left[  0,T\right]
\right\}  _{n\in\mathbb{N}}$ is a Cauchy sequence in $L^{1}\left(  \Xi
;L^{1}\left(  0,T\right)  \right)  $.
\end{theorem}

\begin{proof}
As in previous proofs, we have%
\begin{align*}
& \mathbb{E}\left[  \int_{0}^{T}\left\vert \left\langle \omega_{s}%
\otimes\omega_{s},H_{\phi}^{n}\right\rangle -\int H_{\phi}^{n}\left(
x,x\right)  dx-\left\langle \omega_{s}\otimes\omega_{s},H_{\phi}%
^{m}\right\rangle +\int H_{\phi}^{m}\left(  x,x\right)  dx\right\vert
ds\right] \\
& =\int_{0}^{T}\mathbb{E}\left[  \left\vert \left\langle \omega_{s}%
\otimes\omega_{s},\left(  H_{\phi}^{n}-H_{\phi}^{m}\right)  \right\rangle
-\int\left(  H_{\phi}^{n}-H_{\phi}^{m}\right)  \left(  x,x\right)
dx\right\vert \right]  ds\\
& \leq C_{q}T\left\Vert H_{\phi}^{n}-H_{\phi}^{m}\right\Vert _{L^{2}\left(
\mathbb{T}^{2}\times\mathbb{T}^{2}\right)  }^{1/p}%
\end{align*}

\end{proof}

\begin{definition}
\label{Def nonlin in t rho}Under the assumptions of the previous theorem, we
denote by $\left\langle \omega_{\cdot}\otimes\omega_{\cdot},H_{\phi
}\right\rangle $ the process of class $L^{1}\left(  \Xi;L^{1}\left(
0,T\right)  \right)  $, limit of the sequence $\left\{  s\mapsto\left\langle
\omega_{s}\otimes\omega_{s},H_{\phi}^{n}\right\rangle ;s\in\left[  0,T\right]
\right\}  _{n\in\mathbb{N}}$.
\end{definition}

\subsection{Weak vorticity formulation for white noise vorticity}

\begin{definition}
\label{def WN sol}We say that a measurable map $\omega_{\cdot}:\Xi
\times\left[  0,T\right]  \rightarrow C^{\infty}\left(  \mathbb{T}^{2}\right)
^{\prime}$ with trajectories of class $C\left(  \left[  0,T\right]
;H^{-1-}\left(  \mathbb{T}^{2}\right)  \right)  $ is a white noise solution of
Euler equations if $\omega_{t}$ is a white noise at every time $t\in\left[
0,T\right]  $ and for every $\phi\in C^{\infty}\left(  \mathbb{T}^{2}\right)
$, we have the following identity $P$-a.s., uniformly in time,
\[
\left\langle \omega_{t},\phi\right\rangle =\left\langle \omega_{0}%
,\phi\right\rangle +\int_{0}^{t}\left\langle \omega_{s}\otimes\omega
_{s},H_{\phi}\right\rangle ds.
\]

\end{definition}

Here $\left\langle \omega_{t},\phi\right\rangle $ is a.s. a continuous
function of time because we assume that trajectories of\ $\omega$ are of class
$C\left(  \left[  0,T\right]  ;H^{-1-}\left(  \mathbb{T}^{2}\right)  \right)
$, and $\int_{0}^{t}\left\langle \omega_{s}\otimes\omega_{s},H_{\phi
}\right\rangle ds$ is the continuous process obtained by integration of the
$L^{2}\left(  0,T\right)  $-process provided by Definition
\ref{Def nonlin in t}.

In the case of the previous definition, in addition, we may require that
$\omega_{\cdot}$ is a time-stationary process. In a sense, the law of white
noise is an invariant measure, although we do not have a proper Markov
structure allowing us to talk about invariant measures in the classical sense.

Using Definition \ref{Def nonlin in t rho} we may generalize the previous
definition to the following case:

\begin{definition}
\label{def rho WN sol}Let $\rho:\left[  0,T\right]  \times H^{-1-}\left(
\mathbb{T}^{2}\right)  \rightarrow\lbrack0,\infty)$ satisfy $\int\rho_{t}%
^{q}d\mu\leq C$ for some constants $C>0$, $q>1$, where $\mu$ is the law of
white noise; and $\int\rho_{t}d\mu=1$ for every $t\in\left[  0,T\right]  $.
Let $\omega_{\cdot}:\Xi\times\left[  0,T\right]  \rightarrow C^{\infty}\left(
\mathbb{T}^{2}\right)  ^{\prime}$ be a measurable map with trajectories of
class $C\left(  \left[  0,T\right]  ;H^{-1-}\left(  \mathbb{T}^{2}\right)
\right)  $, such that $\omega_{t}$ has law $\rho_{t}d\mu$, for every
$t\in\left[  0,T\right]  $. We say that $\omega$ is a $\rho-$white noise
solution of Euler equations if for every $\phi\in C^{\infty}\left(
\mathbb{T}^{2}\right)  $, $t\mapsto\left\langle \omega_{t},\phi\right\rangle $
is continuous and we have the following identity $P$-a.s., uniformly in time,
\[
\left\langle \omega_{t},\phi\right\rangle =\left\langle \omega_{0}%
,\phi\right\rangle +\int_{0}^{t}\left\langle \omega_{s}\otimes\omega
_{s},H_{\phi}\right\rangle ds.
\]

\end{definition}

\section{Random point vortex dynamics\label{section random point vortices}}

Let us introduce some notations. In $\left(  \mathbb{T}^{2}\right)  ^{N}$,
denote by $\Delta_{N}$ the generalized diagonal%
\[
\Delta_{N}=\left\{  \left(  x^{1},...,x^{N}\right)  \in\left(  \mathbb{T}%
^{2}\right)  ^{N}:x^{i}=x^{j}\text{ for some }i\neq j\text{, }%
i,j=1,...,n\right\}  .
\]
Then introduce the set of unlabelled and labelled finite sequences of
different points%
\[
F_{N}\mathbb{T}^{2}=\left\{  \left(  x_{1},...,x_{n}\right)  \in\left(
\mathbb{T}^{2}\right)  ^{N}:\left(  x^{1},...,x^{N}\right)  \in\Delta_{N}%
^{c}\right\}
\]%
\[
\mathcal{L}F_{N}\mathbb{T}^{2}=\left\{  \left(  \left(  \xi_{1},x_{1}\right)
,...,\left(  \xi_{N},x_{N}\right)  \right)  \in\left(  \mathbb{R}%
\times\mathbb{T}^{2}\right)  ^{N}:\left(  x^{1},...,x^{N}\right)  \in
\Delta_{N}^{c}\right\}
\]
and the unlabelled and labelled configuration space%
\[
C_{N}\mathbb{T}^{2}=F_{N}\mathbb{T}^{2}/\Sigma_{N}%
\]%
\[
\mathcal{L}C_{N}\mathbb{T}^{2}=\mathcal{L}F_{N}\mathbb{T}^{2}/\Sigma_{N}%
\]
where $\Sigma_{N}$ is the group of permutations of coordinates. This set,
$\mathcal{L}C_{N}\mathbb{T}^{2}$, is in bijection with the set of discrete
signed measures with $n$-point support:%
\[
\mathcal{M}_{N}\left(  \mathbb{T}^{2}\right)  =\left\{  \mu\in\mathcal{M}%
\left(  \mathbb{T}^{2}\right)  :\exists X\in C_{N}\mathbb{T}^{2}:\left\vert
\mu\right\vert \left(  X^{c}\right)  =0,\mu\left(  x\right)  \neq0\text{ for
every }x\in X\right\}  .
\]
We do not use extensively these notations but they may help to formalize
further the topics we are going to describe.

\subsection{Definition for a.e. initial condition}

Consider, for every $N\in\mathbb{N}$, the finite dimensional dynamics in
$\left(  \mathbb{T}^{2}\right)  ^{N}$
\begin{equation}
\frac{dX_{t}^{i,N}}{dt}=\sum_{j=1}^{N}\frac{1}{\sqrt{N}}\xi_{j}K\left(
X_{t}^{i,N}-X_{t}^{j,N}\right)  \qquad i=1,...,N\label{vortex system}%
\end{equation}
with initial condition $\left(  X_{0}^{1,N},...,X_{0}^{N,N}\right)  \in\left(
\mathbb{T}^{2}\right)  ^{N}\backslash\Delta_{N}$, where as above $K$ is the
Biot-Savart kernel on $\mathbb{T}^{2}$; we set $K\left(  0\right)  =0$ so that
the self-interaction (namely when $j=i$) in the sum does not count. The
intensities $\xi_{1},...,\xi_{N}$ are (random) numbers of any sign. One can
consider (\ref{vortex system}) as a dynamics on the configuration space
$C_{N}\mathbb{T}^{2}$. This system corresponds also to the time-evolution of a
vorticity distribution concentrated at positions $\left(  X_{t}^{1,N}%
,...,X_{t}^{N,N}\right)  $:%
\[
\omega_{t}^{N}=\frac{1}{\sqrt{N}}\sum_{n=1}^{N}\xi_{n}\delta_{X_{t}^{n}}.
\]
There are various ways in which one can relate this finite dimensional
dynamics to Euler equations, see \cite{MarPulv}; under our assumptions made
below we shall clarify one of these connections.

In \cite{MarPulv} it is shown an example with $N=3$ and $\xi_{1},\xi_{2}%
,\xi_{3}$ of different signs such that, starting from different initial
positions $X_{0}^{1,3},X_{0}^{2,3},X_{0}^{3,3}$, in finite time $X_{t}%
^{1,3},X_{t}^{2,3},X_{t}^{3,3}$ coincide; this vortex collapse corresponds to
a blow-up in the finite dimensional dynamics (because $K\left(  x-y\right)  $
diverges as $\frac{1}{\left\vert x-y\right\vert }$ as $\left\vert
x-y\right\vert \rightarrow0$) and provokes troubles also at the level of a
PDE\ reformulation of the dynamics of $\omega_{t}^{N}$ (having in mind the
weak vorticity formulation above, the measure $\omega_{t}^{N}\left(
dx\right)  $ concentrates on the diagonal, where $H_{\phi}$ is discontinuous).
These difficulties do not happen for constant sign vortices, but they are not
interesting for our investigation. Let $\otimes_{N}Leb_{\mathbb{T}^{2}}$ be
Lebesgue measure on $\left(  \mathbb{T}^{2}\right)  ^{N}$. The main result we
use below, proved in \cite{MarPulv} is that, independently of the sign of
$\xi_{1},...,\xi_{N}$, for $\otimes_{N}Leb_{\mathbb{T}^{2}}$-a.e. initial
condition $\left(  X_{0}^{1,N},...,X_{0}^{N,N}\right)  \in\left(
\mathbb{T}^{2}\right)  ^{N}$ the positions $\left(  X_{t}^{1,N},...,X_{t}%
^{N,N}\right)  $ remain different for all times; and in addition the measure
$\otimes_{N}Leb_{\mathbb{T}^{2}}$ is invariant, in the sense that $\left(
X_{t}^{1,N},...,X_{t}^{N,N}\right)  $ is distributed as $\otimes
_{N}Leb_{\mathbb{T}^{2}}$ for all $t\geq0$. The precise statement is:

\begin{theorem}
\label{Thm foundational vortices}For every $\left(  \xi_{1},...,\xi
_{N}\right)  \in\mathbb{R}^{N}$ and for $\otimes_{N}Leb_{\mathbb{T}^{2}}$-
almost every $\left(  X_{0}^{1,N},...,X_{0}^{N,N}\right)  \in\Delta_{N}^{c}$,
there is a unique solution $\left(  X_{t}^{1,N},...,X_{t}^{N,N}\right)  $ of
system (\ref{vortex system}), with the property that $\left(  X_{t}%
^{1,N},...,X_{t}^{N,N}\right)  \in\Delta_{N}^{c}$ for all $t\geq0$. Moreover,
considering the initial condition as a random variable with distribution
$\otimes_{N}Leb_{\mathbb{T}^{2}}$, the stochastic process $\left(  X_{t}%
^{1,N},...,X_{t}^{N,N}\right)  $ is stationary, with invariant marginal law
$\otimes_{N}Leb_{\mathbb{T}^{2}}$.
\end{theorem}

When this occurs, the measure-valued process $\omega_{t}^{N}=\frac{1}{\sqrt
{N}}\sum_{n=1}^{N}\xi_{n}\delta_{X_{t}^{n}}$ satisfies, for every $\phi\in
C^{\infty}\left(  \mathbb{T}^{2}\right)  $, the identity%
\begin{align*}
\frac{d}{dt}\left\langle \omega_{t}^{N},\phi\right\rangle  & =\frac{1}%
{\sqrt{N}}\sum_{n=1}^{N}\xi_{n}\frac{d}{dt}\phi\left(  X_{t}^{n}\right)
=\frac{1}{\sqrt{N}}\sum_{n=1}^{N}\xi_{n}\nabla\phi\left(  X_{t}^{n}\right)
\cdot\sum_{j=1}^{N}\frac{1}{\sqrt{N}}\xi_{j}K\left(  X_{t}^{n}-X_{t}%
^{j}\right) \\
& =\int_{\mathbb{T}^{2}}\int_{\mathbb{T}^{2}}\nabla\phi\left(  x\right)  \cdot
K\left(  x-y\right)  \omega_{t}^{N}\left(  dx\right)  \omega_{t}^{N}\left(
dy\right)
\end{align*}
and therefore%
\[
\left\langle \omega_{t}^{N},\phi\right\rangle =\left\langle \omega_{0}%
^{N},\phi\right\rangle +\int_{0}^{t}\left\langle \omega_{s}^{N}\otimes
\omega_{s}^{N},H_{\phi}\right\rangle ds.
\]

\subsection{Random point vortices, at time $t=0$, converging to white noise,
and their time evolution\label{section random vortices}}

On a probability space $\left(  \Xi,\mathcal{F},\mathbb{P}\right)  $, let
$\left(  \xi_{n}\right)  $ be an i.i.d. sequence of $N\left(  0,1\right)  $
r.v.'s and $\left(  X_{0}^{n}\right)  $ be an i.i.d. sequence of
$\mathbb{T}^{2}$-valued r.v.'s, independent of $\left(  \xi_{n}\right)  $ and
uniformly distributed. Denote by
\[
\lambda_{N}^{0}:=\otimes_{N}\left(  N\left(  0,1\right)  \otimes
Leb_{\mathbb{T}^{2}}\right)
\]
the law of the random vector%
\[
\left(  \left(  \xi_{1},X_{0}^{1}\right)  ,...,\left(  \xi_{N},X_{0}%
^{N}\right)  \right)  .
\]
For every $N\in\mathbb{N}$, let us consider also the measure-valued vorticity
field%
\[
\omega_{0}^{N}=\frac{1}{\sqrt{N}}\sum_{n=1}^{N}\xi_{n}\delta_{X_{0}^{n}}.
\]

\begin{remark}
\label{remark main measure on vortices}Since product Lebesgue measure does not
charge the generalized diagonal $\Delta_{N}$, the law $\lambda_{N}^{0}$ can be
seen as a probability measure on the set of labelled ordered different points
$\mathcal{L}F_{N}\mathbb{T}^{2}$ (see the beginning of Section
\ref{section random point vortices}). It is an exchangeable measure (namely
invariant by permutations) and thus it induces a probability measure on the
labelled configuration space $\mathcal{L}C_{N}\mathbb{T}^{2}$. It also induces
a probability measure on $\mathcal{M}_{N}\left(  \mathbb{T}^{2}\right)  $ or,
what we need below, on $H^{-1-}\left(  \mathbb{T}^{2}\right)  $. We shall
denote this induced measure on discrete measures or on distributions by
$\mu_{N}^{0}\left(  d\omega\right)  $. Defined the measurable map
$\mathcal{T}_{N}:\left(  \mathbb{R}\times\mathbb{T}^{2}\right)  ^{N}%
\rightarrow H^{-1-}\left(  \mathbb{T}^{2}\right)  $ as%
\[
\left(  \left(  \xi_{1},X_{0}^{1}\right)  ,...,\left(  \xi_{N},X_{0}%
^{N}\right)  \right)  \overset{\mathcal{T}_{N}}{\mapsto}\frac{1}{\sqrt{N}}%
\sum_{n=1}^{N}\xi_{n}\delta_{X_{0}^{n}}%
\]
we have (with the push-forward notation)%
\[
\mu_{N}^{0}=\left(  \mathcal{T}_{N}\right)  _{\ast}\lambda_{N}^{0}.
\]

\end{remark}

The random distribution $\omega_{0}^{N}$ is centered, becuase%
\[
\mathbb{E}\left[  \xi_{n}\left\langle \delta_{X_{0}^{n}},\varphi\right\rangle
\right]  =0
\]
(true since $\xi_{n}$ and $\left\langle \delta_{X_{0}^{n}},\varphi
\right\rangle $ are independent and $\xi_{n}$ is centered). Let us denote by
$Q_{N}$ the covariance operator of $\omega_{0}^{N}$, defined as%
\[
\left\langle Q_{N}\varphi,\psi\right\rangle =\mathbb{E}\left[  \left\langle
\omega_{0}^{N},\varphi\right\rangle \left\langle \omega_{0}^{N},\psi
\right\rangle \right]
\]
for all $\varphi,\psi\in C^{\infty}\left(  \mathbb{T}^{2}\right)  $. We have
\begin{align*}
\left\langle Q_{N}\varphi,\psi\right\rangle  & =\frac{1}{N}\sum_{n,m=1}%
^{N}\mathbb{E}\left[  \xi_{n}\xi_{m}\left\langle \delta_{X_{0}^{n}}%
,\varphi\right\rangle \left\langle \delta_{X_{0}^{m}},\psi\right\rangle
\right] \\
& =\frac{1}{N}\sum_{n=1}^{N}\mathbb{E}\left[  \xi_{n}^{2}\right]
\mathbb{E}\left[  \left\langle \delta_{X_{0}^{n}},\varphi\right\rangle
\left\langle \delta_{X_{0}^{n}},\psi\right\rangle \right] \\
& =\mathbb{E}\left[  \xi_{1}^{2}\right]  \mathbb{E}\left[  \varphi\left(
X_{0}^{1}\right)  \psi\left(  X_{0}^{1}\right)  \right] \\
& =\int_{\mathbb{T}^{2}}\varphi\left(  x\right)  \psi\left(  x\right)  dx
\end{align*}
hence $\omega_{0}^{N}$ has the same covariance as white noise, but obviously
it is not Gaussian. However, a Hilbert-valued version of the Central Limit
Theorem gives us

\begin{proposition}
\label{Prop CLT}If $\omega_{WN}$ denotes white noise, then
\[
\omega_{0}^{N}\overset{Law}{\rightharpoonup}\omega_{WN}%
\]
where convergence takes place in $H^{-1-\delta}$ for every $\delta>0$.
\end{proposition}

\begin{proof}
The condition for the validity of this claim, a part from the computation
above on the covariance, is that the space is Hilbert and the second moment is
finite:
\begin{equation}
\mathbb{E}\left[  \left\Vert \xi_{n}\delta_{X_{0}^{n}}\right\Vert
_{H^{-1-\delta}}^{2}\right]  <\infty\label{bound for CLT}%
\end{equation}
(see \cite{LedTal}). Condition (\ref{bound for CLT}) is true because
$\mathbb{E}\left[  \left\Vert \xi_{n}\delta_{X_{0}^{n}}\right\Vert
_{H^{-1-\delta}}^{2}\right]  =\mathbb{E}\left[  \left\Vert \delta_{X_{0}^{n}%
}\right\Vert _{H^{-1-\delta}}^{2}\right]  $ and%
\begin{align*}
\left\Vert \delta_{X_{0}^{n}}\right\Vert _{H^{-1-\delta}}  & =\sup_{\left\Vert
\phi\right\Vert _{H^{1+\delta}}\leq1}\left\langle \delta_{X_{0}^{n}}%
,\phi\right\rangle =\sup_{\left\Vert \phi\right\Vert _{H^{1+\delta}}\leq1}%
\phi\left(  X_{0}^{n}\right) \\
& \leq\sup_{\left\Vert \phi\right\Vert _{H^{1+\delta}}\leq1}\left\Vert
\phi\right\Vert _{\infty}\leq C\sup_{\left\Vert \phi\right\Vert _{H^{1+\delta
}}\leq1}\left\Vert \phi\right\Vert _{H^{1+\delta}}=C
\end{align*}
where we have used Sobolev embedding theorem $H^{1+\delta}\left(
\mathbb{T}^{2}\right)  \subset C\left(  \mathbb{T}^{2}\right)  $.
\end{proof}

Obviously, using proper versions of the Central Limit Theorem, one can provide
much more general random point vortices that converge in law to $\omega_{WN}$;
our aim here is not the generality but the construction of an approximation
scheme for our main existence theorem.

As a consequence of Theorem \ref{Thm foundational vortices} we have:

\begin{proposition}
\label{propos point vortices}Consider the vortex dynamics with random
intensities $\left(  \xi_{1},...,\xi_{N}\right)  $ and random initial
positions $\left(  X_{0}^{1},...,X_{0}^{N}\right)  $ distributed as
$\lambda_{N}^{0}$. For a.e. value of $\left(  \left(  \xi_{1},X_{0}%
^{1}\right)  ,...,\left(  \xi_{N},X_{0}^{N}\right)  \right)  $ the dynamics
$\left(  X_{t}^{1,N},...,X_{t}^{N,N}\right)  $ is well defined in $\Delta
_{N}^{c}$ for all $t\geq0$, and the associated measure-valued vorticity
$\omega_{t}^{N}$ satisfies the weak vorticity formulation. The stochastic
process $\omega_{t}^{N}$ is stationary in time and space-homogeneous; in
particular the law of $\left(  \left(  \xi_{1},X_{t}^{1}\right)  ,...,\left(
\xi_{N},X_{t}^{N}\right)  \right)  $ is $\lambda_{N}^{0}$ at any time $t\geq0$.
\end{proposition}

\begin{proof}
The first claims are obvious consequences of Theorem
\ref{Thm foundational vortices}. Given $\left(  \xi_{1},...,\xi_{N}\right)  $,
the process $\left(  X_{t}^{1,N},...,X_{t}^{N,N}\right)  $ is stationary.
Hence, denoted $\left(  \xi_{1},...,\xi_{N}\right)  $ by $\xi$ and $\left(
X_{t}^{1,N},...,X_{t}^{N,N}\right)  $ by $X_{t}$, for every $0\leq t_{1}%
\leq...\leq t_{n}$ and bounded measurable $F$, the random variable
(conditional expectation given the $\sigma$-field generated by $\xi$)
\[
\mathbb{E}\left[  F\left(  \left(  \xi,X_{t_{1}+h}\right)  ,...,\left(
\xi,X_{t_{n}+h}\right)  \right)  |\xi\right]
\]
is independent of $h$ (in the equivalence class of conditional expectation).
Therefore its expectation, namely $\mathbb{E}\left[  F\left(  \left(
\xi,X_{t_{1}+h}\right)  ,...,\left(  \xi,X_{t_{n}+h}\right)  \right)  \right]
$, is independent of $h$, which implies that $\left(  \xi,X_{t}\right)  $ (and
therefore $\omega_{t}^{N}$) is a stationary process. Space homogeneity is not
used below and thus we do not prove it, but it is not difficult due to the
symmetries of the system.
\end{proof}

\subsection{Integrability properties of the random point vortices}

Let $\omega_{t}^{N}$ be given by Proposition \ref{propos point vortices}. It
satisfies estimates similar to those of white noise.

\begin{lemma}
\label{lemma integrabil pont vort}Assume $f:\mathbb{T}^{2}\times\mathbb{T}%
^{2}\rightarrow\mathbb{R}$ is symmetric, bounded and measurable. Then, for
every $p\geq1$ and $\delta>0$ there are constants $C_{p},C_{p,\delta}>0$ such
that
\[
\mathbb{E}\left[  \left\langle \omega_{t}^{N}\otimes\omega_{t}^{N}%
,f\right\rangle ^{p}\right]  \leq C_{p}\left\Vert f\right\Vert _{\infty}^{p}%
\]%
\[
\mathbb{E}\left[  \left\Vert \omega_{t}^{N}\right\Vert _{H^{-1-\delta}}%
^{p}\right]  \leq C_{p,\delta}%
\]
and moreover%
\[
\mathbb{E}\left[  \left\langle \omega_{t}^{N}\otimes\omega_{t}^{N}%
,f\right\rangle ^{2}\right]  =\frac{3}{N}\int f^{2}\left(  x,x\right)
dx+\left(  \int f\left(  x,x\right)  dx\right)  ^{2}+2\int\int f^{2}\left(
x,y\right)  dxdy.
\]

\end{lemma}

\begin{proof}
\textbf{Step 1}. It is sufficient to consider integer values of $p$. One has%
\begin{align*}
\mathbb{E}\left[  \left\langle \omega_{t}^{N}\otimes\omega_{t}^{N}%
,f\right\rangle ^{p}\right]   & =\mathbb{E}\left(  \int_{\mathbb{T}^{2}}%
\int_{\mathbb{T}^{2}}f\left(  x,y\right)  \omega_{t}^{N}\left(  dx\right)
\omega_{t}^{N}\left(  dy\right)  \right)  ^{p}\\
& =\int_{\left(  \mathbb{T}^{2}\right)  ^{2p}}\mathbb{E}\left[
%TCIMACRO{\dprod \limits_{i=1}^{p}}%
%BeginExpansion
{\displaystyle\prod\limits_{i=1}^{p}}
%EndExpansion
f\left(  x_{i},y_{i}\right)
%TCIMACRO{\dprod \limits_{i=1}^{p}}%
%BeginExpansion
{\displaystyle\prod\limits_{i=1}^{p}}
%EndExpansion
\left(  \omega_{t}^{N}\left(  dx_{i}\right)  \omega_{t}^{N}\left(
dy_{i}\right)  \right)  \right] \\
& =\frac{1}{N^{p}}\sum_{k_{1},h_{1},...,k_{p},h_{p}=1}^{N}\mathbb{E}\left[
%TCIMACRO{\dprod \limits_{i=1}^{p}}%
%BeginExpansion
{\displaystyle\prod\limits_{i=1}^{p}}
%EndExpansion
\xi_{k_{i}}\xi_{h_{i}}\right]  \mathbb{E}\left[
%TCIMACRO{\dprod \limits_{i=1}^{p}}%
%BeginExpansion
{\displaystyle\prod\limits_{i=1}^{p}}
%EndExpansion
f\left(  X_{t}^{k_{i}},X_{t}^{h_{i}}\right)  \right]  .
\end{align*}
We replace here Isserlis-Wick theorem by a combinatorial argument based on the
independence of the r.v.'s $\xi_{i}$. Denote by $\mathcal{P}_{p}$ the family
of all $(2p$)-ples $\left(  k_{1},h_{1},...,k_{p},h_{p}\right)  $ that are
"paired", namely such that we may split $\left(  k_{1},h_{1},...,k_{p}%
,h_{p}\right)  $ in $p$ pairs such that in each pair the two elements have the
same value; an example is when $h_{1}=k_{1}$, ... , $h_{p}=k_{p}$. Notice that
we do not require that the values in different pairs are different. One has
$\mathbb{E}\left[
%TCIMACRO{\dprod \limits_{i=1}^{p}}%
%BeginExpansion
{\displaystyle\prod\limits_{i=1}^{p}}
%EndExpansion
\xi_{k_{i}}\xi_{h_{i}}\right]  =0$ if $\left(  k_{1},h_{1},...,k_{p}%
,h_{p}\right)  \notin\mathcal{P}_{p}$, hence%
\begin{align*}
\mathbb{E}\left[  \left\langle \omega_{t}^{N}\otimes\omega_{t}^{N}%
,f\right\rangle ^{p}\right]   & =\frac{1}{N^{p}}\sum_{\left(  k_{1}%
,h_{1},...,k_{p},h_{p}\right)  \in\mathcal{P}_{p}}\mathbb{E}\left[
%TCIMACRO{\dprod \limits_{i=1}^{p}}%
%BeginExpansion
{\displaystyle\prod\limits_{i=1}^{p}}
%EndExpansion
\xi_{k_{i}}\xi_{h_{i}}\right]  \mathbb{E}\left[
%TCIMACRO{\dprod \limits_{i=1}^{p}}%
%BeginExpansion
{\displaystyle\prod\limits_{i=1}^{p}}
%EndExpansion
f\left(  X_{t}^{k_{i}},X_{t}^{h_{i}}\right)  \right] \\
& \leq\left\Vert f\right\Vert _{\infty}^{p}\frac{C_{p}^{\prime}}{N^{p}%
}Card\left(  \mathcal{P}_{p}\right)
\end{align*}
where $C_{p}^{\prime}$ is a constant that bounds from above $\mathbb{E}\left[
%
%TCIMACRO{\dprod \limits_{i=1}^{p}}%
%BeginExpansion
{\displaystyle\prod\limits_{i=1}^{p}}
%EndExpansion
\xi_{k_{i}}\xi_{h_{i}}\right]  $ independently of the index. The cardinality
of $\mathcal{P}_{p}$ is bounded above by $C_{p}^{\prime\prime}N^{p}$ for
another constant $C_{p}^{\prime\prime}>0$ (the idea is that given any one of
the $N$ values of $k_{1}$, either $h_{1}$ or $k_{2}$ or one of the next
indexes is equal to $k_{1}$, and this constraints the variability of that
index to one value;\ then repeat $p$ times this argument). Therefore
$\mathbb{E}\left[  \left\langle \omega_{t}^{N}\otimes\omega_{t}^{N}%
,f\right\rangle ^{p}\right]  \leq\left\Vert f\right\Vert _{\infty}^{p}%
C_{p}^{\prime}C_{p}^{\prime\prime}$. This proves the first claim of the lemma,
with $C_{p}=C_{p}^{\prime}C_{p}^{\prime\prime}$.

\textbf{Step 2}. Similarly,
\begin{align*}
\mathbb{E}\left[  \left\Vert \frac{1}{\sqrt{N}}\sum_{n=1}^{N}\xi_{n}%
\delta_{X_{t}^{n}}\right\Vert _{H^{-1-\delta/2}}^{2p}\right]   &
=\mathbb{E}\left[  \left(  \left\langle \frac{1}{\sqrt{N}}\sum_{n=1}^{N}%
\xi_{n}\delta_{X_{t}^{n}},\frac{1}{\sqrt{N}}\sum_{n=1}^{N}\xi_{n}\delta
_{X_{t}^{n}}\right\rangle _{H^{-1-\delta/2}}\right)  ^{p}\right] \\
& =\frac{1}{N^{p}}\mathbb{E}\left[  \left(  \sum_{n,m=1}^{N}\xi_{n}\xi
_{M}\left\langle \delta_{X_{t}^{n}},\delta_{X_{t}^{m}}\right\rangle
_{H^{-1-\delta/2}}\right)  ^{p}\right] \\
& =\frac{1}{N^{p}}\sum_{k_{1},h_{1},...,k_{p},h_{p}=1}^{N}\mathbb{E}\left[
%TCIMACRO{\dprod \limits_{i=1}^{p}}%
%BeginExpansion
{\displaystyle\prod\limits_{i=1}^{p}}
%EndExpansion
\xi_{k_{i}}\xi_{h_{i}}\right]  \mathbb{E}\left[
%TCIMACRO{\dprod \limits_{i=1}^{p}}%
%BeginExpansion
{\displaystyle\prod\limits_{i=1}^{p}}
%EndExpansion
\left\langle \delta_{X_{t}^{k_{i}}},\delta_{X_{t}^{h_{i}}}\right\rangle
_{H^{-1-\delta/2}}\right] \\
& =\frac{1}{N^{p}}\sum_{\left(  k_{1},h_{1},...,k_{p},h_{p}\right)
\in\mathcal{P}_{p}}\mathbb{E}\left[
%TCIMACRO{\dprod \limits_{i=1}^{p}}%
%BeginExpansion
{\displaystyle\prod\limits_{i=1}^{p}}
%EndExpansion
\xi_{k_{i}}\xi_{h_{i}}\right]  \mathbb{E}\left[
%TCIMACRO{\dprod \limits_{i=1}^{p}}%
%BeginExpansion
{\displaystyle\prod\limits_{i=1}^{p}}
%EndExpansion
\left\langle \delta_{X_{0}^{k_{i}}},\delta_{X_{0}^{h_{i}}}\right\rangle
_{H^{-1-\delta/2}}\right] \\
& \leq C_{p,\delta}%
\end{align*}
because we use the same bounds above for $\mathbb{E}\left[
%TCIMACRO{\dprod \limits_{i=1}^{p}}%
%BeginExpansion
{\displaystyle\prod\limits_{i=1}^{p}}
%EndExpansion
\xi_{k_{i}}\xi_{h_{i}}\right]  $ and $Card\left(  \mathcal{P}_{p}\right)  $
and a trivial uniform bound on $\mathbb{E}\left[
%TCIMACRO{\dprod \limits_{i=1}^{p}}%
%BeginExpansion
{\displaystyle\prod\limits_{i=1}^{p}}
%EndExpansion
\left\langle \delta_{X_{0}^{k_{i}}},\delta_{X_{0}^{h_{i}}}\right\rangle
_{H^{-1-\delta/2}}\right]  $ due to the property $\left\Vert \delta_{X_{0}%
^{i}}\right\Vert _{H^{-1-\delta/2}}\leq C $ showed in the proof of Proposition
\ref{Prop CLT}.

\textbf{Step 3}.%
\begin{align*}
\mathbb{E}\left[  \left\langle \omega_{t}^{N}\otimes\omega_{t}^{N}%
,f\right\rangle ^{2}\right]   & =\mathbb{E}\left(  \int_{\mathbb{T}^{2}}%
\int_{\mathbb{T}^{2}}f\left(  x,y\right)  \omega_{t}^{N}\left(  dx\right)
\omega_{t}^{N}\left(  dy\right)  \right)  ^{2}\\
& =\mathbb{E}\int_{\left(  \mathbb{T}^{2}\right)  ^{4}}f\left(  x,y\right)
f\left(  x^{\prime},y^{\prime}\right)  \omega_{t}^{N}\left(  dx\right)
\omega_{t}^{N}\left(  dy\right)  \omega_{t}^{N}\left(  dx^{\prime}\right)
\omega_{t}^{N}\left(  dy^{\prime}\right) \\
& =\frac{1}{N^{2}}\sum_{ijkh=1}^{N}\mathbb{E}\left[  f\left(  X_{t}^{i}%
,X_{t}^{j}\right)  f\left(  X_{t}^{k},X_{t}^{h}\right)  \right]
\mathbb{E}\left[  \xi_{i}\xi_{j}\xi_{k}\xi_{h}\right]  .
\end{align*}
In this sum there are various terms. The term with $i=j=k=h$ is%
\[
\frac{1}{N^{2}}\sum_{i=1}^{N}\mathbb{E}\left[  f\left(  X_{t}^{i},X_{t}%
^{i}\right)  f\left(  X_{t}^{i},X_{t}^{i}\right)  \right]  \mathbb{E}\left[
\xi_{i}^{4}\right]  =\frac{\mathbb{E}\left[  \xi^{4}\right]  }{N}\int
f^{2}\left(  x,x\right)  dx.
\]
Then there are terms with $j=i$, $h=k$:\
\begin{align*}
& \frac{1}{N^{2}}\sum_{i\neq k=1}^{N}E\left[  \xi_{i}^{2}\right]  E\left[
\xi_{k}^{2}\right]  \mathbb{E}\left[  f\left(  X_{t}^{i},X_{t}^{i}\right)
f\left(  X_{t}^{k},X_{t}^{k}\right)  \right] \\
& =\frac{E\left[  \xi^{2}\right]  ^{2}}{N^{2}}\sum_{i\neq k=1}^{N}%
\mathbb{E}\left[  f\left(  X_{t}^{i},X_{t}^{i}\right)  \right]  \mathbb{E}%
\left[  f\left(  X_{t}^{k},X_{t}^{k}\right)  \right] \\
& \leq E\left[  \xi^{2}\right]  ^{2}\left(  \int f\left(  x,x\right)
dx\right)  ^{2}.
\end{align*}
Then there are terms with $k=i$, $h=j$:%
\[
\frac{E\left[  \xi^{2}\right]  ^{2}}{N^{2}}\sum_{i\neq j=1}^{N}\mathbb{E}%
\left[  f\left(  X_{t}^{i},X_{t}^{j}\right)  f\left(  X_{t}^{i},X_{t}%
^{j}\right)  \right]  \mathbb{\leq}E\left[  \xi^{2}\right]  ^{2}\int\int
f^{2}\left(  x,y\right)  dxdy.
\]
Finally, then there are terms with $k=j$, $h=i$: (here we use symmetry)%
\[
\frac{E\left[  \xi^{2}\right]  ^{2}}{N^{2}}\sum_{i\neq j=1}^{N}\mathbb{E}%
\left[  f\left(  X_{t}^{i},X_{t}^{j}\right)  f\left(  X_{t}^{j},X_{t}%
^{i}\right)  \right]  \mathbb{\leq}E\left[  \xi^{2}\right]  ^{2}\int\int
f^{2}\left(  x,y\right)  dxdy.
\]

\end{proof}

\section{Main results}

Denote by $\mu$ the law of White Noise. We first formulate our version of
Albeverio-Cruzeiro result \cite{AlbCruz}.

\begin{theorem}
\label{Thm AC}There exists a probability space $\left(  \Xi,\mathcal{F}%
,P\right)  $ with the following properties.

i) There exists a measurable map $\omega_{\cdot}:\Xi\times\left[  0,T\right]
\rightarrow C^{\infty}\left(  \mathbb{T}^{2}\right)  ^{\prime}$ such that
$\omega_{\cdot}$ is a time-stationary white noise solution of Euler equations,
in the sense of Definition \ref{def WN sol}.

ii) On $\left(  \Xi,\mathcal{F},P\right)  $ one can define the random point
vortex system described in Section \ref{section random vortices}; it has a
subsequence which converges $P$-a.s. to the solution of point (i) in $C\left(
\left[  0,T\right]  ;H^{-1-}\left(  \mathbb{T}^{2}\right)  \right)  $.
\end{theorem}

We prove also a generalization to $\rho-$white noise solutions; the assumption
on $\rho_{0}$ is presumably too restrictive but further investigation is
needed for more generality.

\begin{theorem}
\label{Thm AC rho}Given $\rho_{0}\in C_{b}\left(  H^{-1-}\left(
\mathbb{T}^{2}\right)  \right)  $ such that $\rho_{0}\geq0$ and $\int\rho
_{0}d\mu=1$, there exist a probability space $\left(  \Xi,\mathcal{F}%
,P\right)  $, a bounded measurable function $\rho:\left[  0,T\right]  \times
H^{-1-}\left(  \mathbb{T}^{2}\right)  \rightarrow\left[  0,\left\Vert \rho
_{0}\right\Vert _{\infty}\right]  $ and a measurable map $\omega_{\cdot}%
:\Xi\times\left[  0,T\right]  \rightarrow C^{\infty}\left(  \mathbb{T}%
^{2}\right)  ^{\prime}$ such that $\omega_{\cdot}$ is a $\rho-$white noise
solution of Euler equations, in the sense of Definition \ref{def rho WN sol}.
It is also the limit $P$-a.s. in $C\left(  \left[  0,T\right]  ;H^{-1-}\left(
\mathbb{T}^{2}\right)  \right)  $ of a suitable sequence of random point vortices.
\end{theorem}

\subsection{Remarks on disintegration, uniqueness an Gaussianity}

In this section we discuss several limits of the previous results and open
problems arising from them.

Consider the law $Q$, on path space $C\left(  \left[  0,T\right]
;H^{-1-}\left(  \mathbb{T}^{2}\right)  \right)  $, of a solutions provided by
Theorem \ref{Thm AC} (similarly for Theorem \ref{Thm AC rho}). If we
disintegrate $Q$ with respect to the marginal law at time $t=0$ (namely the
white noise law $\mu$ for Theorem \ref{Thm AC} or law $\rho_{0}d\mu$ for
Theorem \ref{Thm AC rho}), we find a probability kernel $Q\left(  \cdot
,\omega_{0}\right)  $, indexed by $\omega_{0}\in H^{-1-}\left(  \mathbb{T}%
^{2}\right)  $, such that for $\mu$-a.e. $\omega_{0}\in H^{-1-}\left(
\mathbb{T}^{2}\right)  $ the probability measure $Q\left(  \cdot,\omega
_{0}\right)  $ is concentrated on solutions of Euler equations (in the sense
described above). But $Q\left(  \cdot,\omega_{0}\right)  $ is not of the form
$\delta_{\omega_{\cdot}^{\omega_{0}}}$, namely it is not concentrated on a
single solution $\omega_{t}^{\omega_{0}}$ with initial condition $\omega_{0}$;
or at least we do not know this information. In the language of
\cite{Ambrosio}, we have a superposition solution that we do not know to be a
graph. For $\mu$-a.e. $\omega_{0}\in H^{-1-}\left(  \mathbb{T}^{2}\right)  $,
we have at least one solution $\omega$ of Euler equations, but we could have
many;\ also in the sense of the Lagrangian flows described in \cite{Ambrosio},
see below.

In the case of Theorem \ref{Thm AC rho} on $\rho-$white noise solutions, we
are certainly far away from any uniqueness claim, even in law. Presumably one
should try first to investigate uniqueness of $\rho_{t}$, maybe with tools
related to those of \cite{AmbrFigalli}, \cite{AmbrTrevisan}, \cite{DFR},
\cite{FangLuo}, which already looks a formidable task.

In the case however of Theorem \ref{Thm AC}, due to fact that the law at any
time $t$ is uniquely determined, it could seem that a statement of uniqueness
in law is not far (notice that uniqueness in law would also imply that the
full sequence of point vortices converges to it, in law). And perhaps a
statement of uniqueness of Lagrangian flows. These are however open problems,
potentially of very difficult solution. Let us mention where two approaches,
both based on uniqueness of the 1-dimensional marginals, meet essential difficulties.

One approach is by the criteria of uniqueness for martingale solutions of
stochastic equations (applicable in principle to deterministic equations with
random solutions). Take as an example Theorem 6.2.3 of \cite{Stroock-Varadhan}%
. It does not apply here, at the present stage of our understanding, since we
do not have any information of uniqueness of 1-dimensional marginals starting
from generic deterministic initial conditions. As remarked above, by
disintegration we may construct solutions $Q\left(  \cdot,\omega_{0}\right)  $
(in the sense of the martingale problem;\ we do not develop the details) for
$\mu$-a.e. $\omega_{0}\in H^{-1-}\left(  \mathbb{T}^{2}\right)  $, but we do
not know the uniqueness of their 1-point marginals.

A second approach is described in \cite{Ambrosio}, see Theorem 16. It requires
the validity of comparison principle, a variant of 1-point marginal
uniqueness, for the associated continuity equation. The comparison principle
should hold in a convex class of solutions (denoted by $\mathcal{L}_{b}$ in
\cite{Ambrosio}); if only this, one could take the class defined by the rule
that it is white noise at every time. However, the class $\mathcal{L}_{b}$ in
\cite{Ambrosio} has to satisfy also a monotonicity property (see (14) in
\cite{Ambrosio}, used in essential way in Theorem 18), which is not satisfied
by the trivial class defined by being white noise at every time. If we enlarge
the class to have the monotonicity property, we are faced with a very
difficult question of uniqueness - or comparison principle - for weak
solutions of the continuity equation associated to Euler equations, which is
an open problem.

The $k$-dimensional time marginals are not easily identified by the Euler
equations or by the random point vortex dynamics. The question is, given
$0\leq t_{1}<\cdot\cdot\cdot<t_{k}\leq T$, to understand the limit as
$N\rightarrow\infty$ of the marginal $\left(  \omega_{t_{1}}^{N}%
,...,\omega_{t_{k}}^{N}\right)  $, given by%
\[
\left(  \omega_{t_{1}}^{N},...,\omega_{t_{k}}^{N}\right)  =\frac{1}{\sqrt{N}%
}\sum_{n=1}^{N}\xi_{n}\left(  \delta_{X_{t_{1}}^{n}},...,\delta_{X_{t_{k}}%
^{n}}\right)  .
\]
This is an open problem.

For Burgers equations with white noise initial conditions, thanks to special
representation formulae, it was possible to compute \ the two-point
distribution, see \cite{FracheMartin}. Here we do not see yet a method. But,
also due to the comparison with \cite{FracheMartin}, one should be aware that
there is no reason why $k$-dimensional time marginals are Gaussian!
Nonlinearity, still preserving a Gaussian initial condition, should distroy
Gaussianity at the level of the process.

Another example of nonlinear equation with stationary solutions having
Gaussian 1-dimensional marginals is KPZ equation or the stochastic Burgers
equations, see \cite{Hairer}, \cite{Gubi}, \cite{GubiPerk}.

\subsection{Proof of Theorem \ref{Thm AC}}

Consider the Polish space $\mathcal{X}=C\left(  \left[  0,T\right]
;H^{-1-}\left(  \mathbb{T}^{2}\right)  \right)  $ with the metric
$d_{\mathcal{X}}\left(  \omega_{\cdot},\omega_{\cdot}^{\prime}\right)  $
defined in Section \ref{sect notations}. J. Simon \cite{Simon}, in Corollary
8, gives a useful class of compact sets in this space, generalizing the more
classical Aubin-Lions compactness lemma (and Ascoli-Arzel\`{a} criterion). Let
us explain the result of Simon in our context. Take $\delta\in\left(
0,1\right)  $, $\gamma>3$ (this special choice of $\gamma$ is due to the
estimates below) and consider the spaces
\[
X=H^{-1-\delta/2}\left(  \mathbb{T}^{2}\right)  ,\qquad B=H^{-1-\delta}\left(
\mathbb{T}^{2}\right)  ,\qquad Y=H^{-\gamma}\left(  \mathbb{T}^{2}\right)  .
\]
We have%
\[
X\subset B\subset Y
\]
with compact dense embeddings and we also have, for a suitable constant $C>0$
and for
\[
\theta=\frac{\delta/2}{\gamma-1-\delta/2}%
\]
the interpolation inequality%
\[
\left\Vert \omega\right\Vert _{B}\leq C\left\Vert \omega\right\Vert
_{X}^{1-\theta}\left\Vert \omega\right\Vert _{Y}^{\theta}%
\]
for all $\omega\in X$. These are preliminary assumptions of Corollary 8 of
\cite{Simon}. Then such Corollary, in the second part, in the particular case
$r_{1}=2$, states that a bounded family $F$ in
\[
L^{p_{0}}\left(  0,T;X\right)  \cap W^{1,2}\left(  0,T;Y\right)
\]
is relatively compact in
\[
C\left(  \left[  0,T\right]  ;B\right)
\]
if%
\[
\frac{\theta}{2}>\frac{1-\theta}{p_{0}}.
\]
Here $p_{0}$ is any number in $\left[  1,\infty\right]  $. We apply this
result to our spaces $X,B,Y$, taking $p_{0}$ large enough to have the previous
inequality. More precisely, we use the following statement (notice that
$\frac{1-\theta}{\theta}=\frac{\gamma-1-\delta}{\delta/2}$):

\begin{lemma}
Let $\delta>0$, $\gamma>3$ be given. If
\[
p_{0}>\frac{\gamma-1-\delta}{\delta/2}%
\]
then
\[
L^{p_{0}}\left(  0,T;H^{-1-\delta/2}\left(  \mathbb{T}^{2}\right)  \right)
\cap W^{1,2}\left(  0,T;H^{-\gamma}\left(  \mathbb{T}^{2}\right)  \right)
\]
is compactly embedded into
\[
C\left(  \left[  0,T\right]  ;H^{-1-\delta}\left(  \mathbb{T}^{2}\right)
\right)  .
\]

\end{lemma}

In fact we need compactness in $\mathcal{X}$. Denote by $L^{\infty-}\left(
0,T;H^{-1-}\left(  \mathbb{T}^{2}\right)  \right)  $ the space of all
functions of class $L^{p_{0}}\left(  0,T;H^{-1-\delta}\left(  \mathbb{T}%
^{2}\right)  \right)  $ for any $p_{0}>0$ and $\delta>0$, endowed with the
metric%
\[
d_{L_{t}^{\infty-}\left(  H^{-1-}\right)  }\left(  \omega_{\cdot}%
,\omega_{\cdot}^{\prime}\right)  =\sum_{n=1}^{\infty}2^{-n}\left(  \left(
\int_{0}^{T}\left\Vert \omega_{t}-\omega_{t}^{\prime}\right\Vert
_{H^{-1-\frac{1}{n}}}^{n}\right)  ^{1/n}\wedge1\right)  .
\]
It is a simple exercise to check that:

\begin{corollary}
Let $\gamma>3$ be given. Then
\[
\mathcal{Y}:=L^{\infty-}\left(  0,T;H^{-1-}\left(  \mathbb{T}^{2}\right)
\right)  \cap W^{1,2}\left(  0,T;H^{-\gamma}\left(  \mathbb{T}^{2}\right)
\right)
\]
is compactly embedded into $\mathcal{X}$.
\end{corollary}

Let $Q^{N}$ be the law of $\omega^{N}$ on Borel subsets of $\mathcal{X}$. We
want to prove that the family $\left\{  Q^{N}\right\}  _{N\in\mathbb{N}}$ is
tight in this space. In order to prove this, it is sufficient to prove that
the family $\left\{  Q^{N}\right\}  _{N\in\mathbb{N}}$ is bounded in
probability in the space $\mathcal{Y}$ given by the previous corollary. For
this purpose, it is sufficient to prove that $\left\{  Q^{N}\right\}
_{N\in\mathbb{N}}$ is bounded in probability in $W^{1,2}\left(  0,T;H^{-\gamma
}\left(  \mathbb{T}^{2}\right)  \right)  $ and in each $L^{p_{0}}\left(
0,T;H^{-1-\delta}\left(  \mathbb{T}^{2}\right)  \right)  $, for any $p_{0}>0$
and $\delta>0$. Let us prove these conditions.

The family $\left\{  Q^{N}\right\}  _{N\in\mathbb{N}}$ is bounded in
probability in $L^{p_{0}}\left(  0,T;H^{-1-\delta}\left(  \mathbb{T}%
^{2}\right)  \right)  $ (by Chebyshev inequality)\ because%
\[
\sup_{N\in\mathbb{N}}\mathbb{E}\left[  \int_{0}^{T}\left\Vert \omega_{t}%
^{N}\right\Vert _{H^{-1-\delta}}^{p_{0}}dt\right]  <\infty.
\]
This inequality (that we could conceptually summarize as the "compactness in
space") comes from stationarity of $\omega_{t}^{N}$:%

\[
\mathbb{E}\left[  \int_{0}^{T}\left\Vert \omega_{t}^{N}\right\Vert
_{H^{-1-\delta}}^{p_{0}}dt\right]  =\int_{0}^{T}\mathbb{E}\left[  \left\Vert
\omega_{t}^{N}\right\Vert _{H^{-1-\delta}}^{p_{0}}\right]  dt\leq
C_{p_{0},\delta}T
\]
by Lemma \ref{lemma integrabil pont vort}.

To prove "compactness in time", namely the property that the family $\left\{
Q^{N}\right\}  _{N\in\mathbb{N}}$ is bounded in probability in $W^{1,2}\left(
0,T;H^{-\gamma}\left(  \mathbb{T}^{2}\right)  \right)  $, we use the equation,
in its weak vorticity formulation. We have, for all $\phi\in C^{\infty}\left(
\mathbb{T}^{2}\right)  $,
\[
\left\langle \omega_{t}^{N},\phi\right\rangle =\left\langle \omega_{0}%
^{N},\phi\right\rangle +\int_{0}^{t}\left\langle \omega_{s}^{N}\otimes
\omega_{s}^{N},H_{\phi}\right\rangle ds
\]
where $P$-a.s. the function $s\mapsto\left\langle \omega_{s}^{N}\otimes
\omega_{s}^{N},H_{\phi}\right\rangle $ is continuous (the trajectories of
point vortices are continuous and never touch the diagonal), hence, $P$-a.s.,
the function $t\mapsto\left\langle \omega_{t}^{N},\phi\right\rangle $ is
continuously differentiable and $\partial_{t}\left\langle \omega_{t}^{N}%
,\phi\right\rangle =\left\langle \omega_{t}^{N}\otimes\omega_{t}^{N},H_{\phi
}\right\rangle $. Thus%
\begin{align*}
\mathbb{E}\left[  \left\vert \partial_{t}\left\langle \omega_{t}^{N}%
,\phi\right\rangle \right\vert ^{2}\right]   & =\mathbb{E}\left[  \left\vert
\left\langle \omega_{t}^{N}\otimes\omega_{t}^{N},H_{\phi}\right\rangle
\right\vert ^{2}\right] \\
& \leq C\left\Vert H_{\phi}\right\Vert _{\infty}^{2}\leq C\left\Vert D^{2}%
\phi\right\Vert _{\infty}^{2}%
\end{align*}
by Lemma \ref{lemma integrabil pont vort}. Then we apply this inequality to
$\phi=e_{k}$ and get%
\[
\mathbb{E}\left[  \left\vert \partial_{t}\left\langle \omega_{t}^{N}%
,e_{k}\right\rangle \right\vert ^{2}\right]  \leq C\left\vert k\right\vert
^{4}.
\]
Therefore%
\begin{align*}
\mathbb{E}\left[  \int_{0}^{T}\left\Vert \partial_{t}\omega_{t}^{N}\right\Vert
_{H^{-\gamma}}^{2}dt\right]   & =\mathbb{E}\left[  \int_{0}^{T}\sum_{k}\left(
1+\left\vert k\right\vert ^{2}\right)  ^{-\gamma}\left\vert \left\langle
\partial_{t}\omega_{t}^{N},e_{k}\right\rangle \right\vert ^{2}dt\right] \\
& \leq C\mathbb{E}\left[  \int_{0}^{T}\sum_{k}\left(  1+\left\vert
k\right\vert ^{2}\right)  ^{-\gamma}\left\vert k\right\vert ^{4}dt\right]
<\infty
\end{align*}
for $2\gamma-4>2$, hence $\gamma>3$. The estimate for $\mathbb{E}\left[
\int_{0}^{T}\left\Vert \omega_{t}^{N}\right\Vert _{H^{-\gamma}}^{2}dt\right]
$ is similar to the one for "compactness in space" above. By Chebyshev
inequality, $\left\{  Q^{N}\right\}  _{N\in\mathbb{N}}$ is bounded in
probability in $W^{1,2}\left(  0,T;H^{-\gamma}\left(  \mathbb{T}^{2}\right)
\right)  $.

We have proved that the family $\left\{  Q^{N}\right\}  _{N\in\mathbb{N}}$ is
bounded in probability in $\mathcal{Y}$ and thus it is tight in $\mathcal{X}$.
From Prohorov theorem, it is relatively compact in $\mathcal{X}$. Let
$\left\{  Q^{N_{k}}\right\}  _{k\in\mathbb{N}}$ be a subsequence which
converges weakly, in $\mathcal{X}$, to a Borel probability measure $Q$. First,
convergence in $\mathcal{X}$ implies that $Q$ is invariant by time-shift
(because $Q^{N}$ is; by shift we mean shift of finite dimensional
distributions such that all involved time points are in $\left[  0,T\right]
$) and the marginal at any time is the law of white noise, by Proposition
\ref{Prop CLT} (recall that $\omega_{t}^{N}$ is stationary, hence this
proposition applies at every time).

By Skorokhod representation theorem, there exist a new probability space
$\left(  \widehat{\Xi},\widehat{\mathcal{F}},\widehat{P}\right)  $ and r.v.'s
$\widehat{\omega}^{N_{k}}$, $\widehat{\omega}$ with values in $\mathcal{X}$,
such that the laws of $\widehat{\omega}^{N_{k}}$ and $\widehat{\omega}$ are
$Q^{N_{k}}$ and $Q$ respectively, and $\widehat{\omega}^{N_{k}}$ converges
$P$-a.s. to $\widehat{\omega}$ in the topology of $\mathcal{X}$; since
$\mathcal{X}$ is made of functions of time, we may see $\widehat{\omega
}^{N_{k}}$ and $\widehat{\omega}$ as stochastic processes, $\widehat{\omega
}_{t}^{N_{k}}$ and $\widehat{\omega}_{t}$ being the result of application of
the projection at time $t$. We are going to check that $\widehat{\omega}$, or
more precisely another process closely defined, is the solution claimed by the
theorem. We already know it has trajectories of class $C\left(  \left[
0,T\right]  ;H^{-1-}\left(  \mathbb{T}^{2}\right)  \right)  $, it is time
stationary and with marginal being a white noise. We have to show that it
satisfies the equation, in the sense specified by the definitions.

We have to enlarge the probability space $\left(  \widehat{\Xi}%
,\widehat{\mathcal{F}},\widehat{P}\right)  $ to be sure it contains certain
independent r.v.'s we need in the construction. Denote by $\left(
\widetilde{\Xi},\widetilde{\mathcal{F}},\widetilde{P}\right)  $ a probability
space where, for every $N$, it is defined a random permutation $\widetilde{s}%
_{N}:\widetilde{\Xi}\rightarrow\Sigma_{N}$, uniformly distributed. Define the
new probability space%
\[
\left(  \Xi,\mathcal{F},P\right)  :=\left(  \widehat{\Xi}\times\widetilde{\Xi
},\widehat{\mathcal{F}}\otimes\widetilde{\mathcal{F}},\widehat{P}%
\otimes\widetilde{P}\right)
\]
and the new processes%
\[
\omega^{N_{k}}=\widehat{\omega}^{N_{k}}\circ\pi_{1},\qquad\omega
=\widehat{\omega}\circ\pi_{1},\qquad s_{N}=\widetilde{s}_{N}\circ\pi_{2}%
\]
where $\pi_{1}$ and $\pi_{2}$ are the projections on $\widehat{\Xi}%
\times\widetilde{\Xi}$. We adopt a little abuse of notation here, because we
indicate the final spaces and processes like the original ones, but we shall
try to clarify everywhere which ones we are investigating. Notice that the
properties of convergence and of the laws of the processes $\omega^{N_{k}}$
and $\omega$ are the same as those of $\widehat{\omega}^{N_{k}}$ and
$\widehat{\omega}$.

\begin{lemma}
\label{lemma representation}The process $\omega_{t}^{N_{k}}$ (the one on the
new probability space) can be represented in the form $\frac{1}{\sqrt{N_{k}}%
}\sum_{i=1}^{N_{k}}\xi_{i}\delta_{X_{t}^{i,N_{k}}}$, where%
\begin{equation}
\left(  \left(  \xi_{1},X_{0}^{1,N_{k}}\right)  ,...,\left(  \xi_{N_{k}}%
,X_{0}^{N_{k},N_{k}}\right)  \right) \label{initial cond}%
\end{equation}
is a random vector with law $\lambda_{N}^{0}$ and $\left(  X_{t}^{1,N_{k}%
},...,X_{t}^{N_{k},N_{k}}\right)  $ solves system (\ref{vortex system}) with
initial condition $\left(  X_{0}^{1,N_{k}},...,X_{0}^{N_{k},N_{k}}\right)  $.
\end{lemma}

\begin{proof}
\textbf{Step 1}. Let us list a few preliminary facts; we omit some detail in
the proofs; we extensively use the notations at the beginning of Section
\ref{section random point vortices}.

Identify for a second $\mathbb{T}^{2}$ with $[0,1)^{2}$. On $[0,1)^{2}$,
consider the lexicographic order: $x=\left(  a,b\right)  $ is smaller than
$y=\left(  c,d\right)  $ either if $a<c$ or if $a=c$ but $b<d$. It is a total
order. We write $<_{L}$ for the strict lexicographic order just defined. Let
us denote by $\mathcal{L}\Lambda_{N}^{1}\subset\mathcal{L}F_{N}\mathbb{T}^{2}$
the set of strings $\left(  \left(  \xi_{1},x_{1}\right)  ,...,\left(  \xi
_{N},x_{N}\right)  \right)  $ such that $x_{1}<_{L}...<_{L}x_{N}$, with
$x_{i}$ seen as elements of $[0,1)^{2}$. The set $\mathcal{L}F_{N}%
\mathbb{T}^{2}$ is partitioned in $N!$ subsets $\mathcal{L}\Lambda_{N}%
^{1},...,\mathcal{L}\Lambda_{N}^{N!}$ obtained applying to $\mathcal{L}%
\Lambda_{N}^{1}$ each one of the $N!$ permutations of indexes.

Given $\omega\in\mathcal{M}_{N}\left(  \mathbb{T}^{2}\right)  $, there is a
unique element $\left\{  \left(  \xi_{i},x_{i}\right)  ,i=1,..,N\right\}
\in\mathcal{L}C_{N}\mathbb{T}^{2}=\mathcal{L}F_{N}\mathbb{T}^{2}/\Sigma_{N}$
such that $\omega=\frac{1}{\sqrt{N}}\sum_{i=1}^{N}\xi_{i}\delta_{x_{i}}$.
Notice that the indexing $i=1,..,N$ here, a priori, is not canonical. However,
we may use the\ lexicographic order, and the fact that point are disjoint,
to\ attribute the indexes $i=1,..,N$ to the elements of the set $\left\{
\left(  \xi_{i},x_{i}\right)  ,i=1,..,N\right\}  $, in such a way that
$\left(  \left(  \xi_{1},x_{1}\right)  ,...,\left(  \xi_{N},x_{N}\right)
\right)  \in\mathcal{L}\Lambda_{N}^{1}$. This way, we have uniquely defined
maps $\omega\overset{h_{1}}{\mapsto}\left(  \xi_{1},x_{1}\right)  $, ...,
$\omega\overset{h_{N}}{\mapsto}\left(  \xi_{N},x_{N}\right)  $, from
$\mathcal{M}_{N}\left(  \mathbb{T}^{2}\right)  $ to $\mathbb{R\times T}^{2}$.

On $\mathcal{M}_{N}\left(  \mathbb{T}^{2}\right)  \subset H^{-1-}\left(
\mathbb{T}^{2}\right)  $ let us put the topology induced by $d_{H^{-1-}}$ and
consider the functions of class $C\left(  \left[  0,T\right]  ;\mathcal{M}%
_{N}\left(  \mathbb{T}^{2}\right)  \right)  $. The set $\mathcal{M}_{N}\left(
\mathbb{T}^{2}\right)  $ is measurable in $H^{-1-}\left(  \mathbb{T}%
^{2}\right)  $, and the set $C\left(  \left[  0,T\right]  ;\mathcal{M}%
_{N}\left(  \mathbb{T}^{2}\right)  \right)  $ is measurable in $C\left(
\left[  0,T\right]  ;H^{-1-}\left(  \mathbb{T}^{2}\right)  \right)  $ (the
proof is not difficult arguing on suitable close subfamilies of $\mathcal{M}%
_{N}\left(  \mathbb{T}^{2}\right)  $, constrained by the minimal distance
between elements in the support).

If $\omega_{t}^{N}=\frac{1}{\sqrt{N}}\sum_{i=1}^{N}\xi_{i}\delta_{X_{t}^{i,N}%
}$ comes from the vortex point dynamics with an initial condition such that
coalescence does not occur, then $\omega_{\cdot}^{N}\in C\left(  \left[
0,T\right]  ;\mathcal{M}_{N}\left(  \mathbb{T}^{2}\right)  \right)  $: to
prove this, one has to use the embedding of $H^{-1-}\left(  \mathbb{T}%
^{2}\right)  $ into H\"{o}lder continuous functions, in evaluating
\[
\sup_{\left\Vert \phi\right\Vert _{H^{-1-\delta}}\leq1}\left\vert \sum
_{i=1}^{N}\xi_{i}\left(  \phi\left(  X_{t}^{i,N}\right)  -\phi\left(
X_{s}^{i,N}\right)  \right)  \right\vert .
\]

Conversely, if $\omega_{\cdot}^{N}\in C\left(  \left[  0,T\right]
;\mathcal{M}_{N}\left(  \mathbb{T}^{2}\right)  \right)  $, then there exist
functions $x_{\cdot}^{i,N}\in C\left(  \left[  0,T\right]  ;\mathbb{T}%
^{2}\right)  $ and numbers $\xi_{i}$, $i=1,...,N$, such that $\omega_{t}%
^{N}=\frac{1}{\sqrt{N}}\sum_{i=1}^{N}\xi_{i}\delta_{x_{t}^{i,N}}$; the lengthy
proof requires identification of these functions locally in time by means of
very concentrated test functions. The indexing $i=1,...,N$ of this functions
however cannot correspond to lexicographic order: to have lexicographic order
at every time we should accept jumps in time (these jumps occur every time the
first coordinates of two points exchange their order, also due to the
difference between $\mathbb{T}^{2}$ and $[0,1)^{2}$). Let us impose
lexicographic order only at time $t=0$ (in doing so there is no problem to
identify $\mathbb{T}^{2}$ with $[0,1)^{2}$) and then accept that particles
exchange lexicographic order later in time, with the advantage that $x_{\cdot
}^{i,N}\in C\left(  \left[  0,T\right]  ;\mathbb{T}^{2}\right)  $. Thus we
have uniquely defined the maps $\omega_{\cdot}^{N}$ $\overset{\widetilde{h}%
_{1}}{\mapsto}\left(  \xi_{1},x_{\cdot}^{1,N}\right)  $, ..., $\omega_{\cdot
}^{N}$ $\overset{\widetilde{h}_{N}}{\mapsto}\left(  \xi_{N},x_{\cdot}%
^{N,N}\right)  $ from $C\left(  \left[  0,T\right]  ;\mathcal{M}_{N}\left(
\mathbb{T}^{2}\right)  \right)  $\ to $\mathbb{R\times}C\left(  \left[
0,T\right]  ;\mathbb{T}^{2}\right)  $: at time zero we impose $x_{0}%
^{1,N}<_{L}...<_{L}x_{0}^{N,N}$ (at later times this may be not true anymore).
These maps are measurable.

Finally let us discuss the last preliminary fact we need below. Given a
probability measure $\rho$ on $\mathcal{L}F_{N}\mathbb{T}^{2}$, assume it is
exchangeable, namely its law is invariant by permutation of the indexes; it is
thus uniquely determined by its restriction to $\mathcal{L}\Lambda_{N}^{1}$.
Consider $\rho$ restricted to $\mathcal{L}\Lambda_{N}^{1}$, remormalized by
$N!$ so to be a probability measure; call $\widehat{\rho}$ such measure. We
have a one-to-one correspondence between $\rho$ and $\widehat{\rho}$, measures
on $\mathcal{L}F_{N}\mathbb{T}^{2}$ and $\mathcal{L}\Lambda_{N}^{1}$ respectively.

In particular, given a measure $\widehat{\rho}$ on $\mathcal{L}\Lambda_{N}%
^{1}$, we may reconstruct an exchangeable measure on $\mathcal{L}%
F_{N}\mathbb{T}^{2}$, the unique one that restricted to $\mathcal{L}%
\Lambda_{N}^{1}$ gives values proportional to $\widehat{\rho}$ up to $N!$.
Assume more, namely that $\widehat{\rho}$ on $\mathcal{L}\Lambda_{N}^{1}$ is
the law of a vector $\left(  \left(  \widehat{\xi}_{1},\widehat{X}_{1}\right)
,...,\left(  \widehat{\xi}_{N},\widehat{X}_{N}\right)  \right)  $, defined on
a probability space $\left(  \widehat{\Xi},\widehat{\mathcal{F}}%
,\widehat{P}\right)  $. Enlarge the probability space as described before the
lemma, incorporating independent permutations $\widetilde{s}_{N}%
:\widetilde{\Xi}\rightarrow\Sigma_{N}$. On the product space $\left(
\Xi,\mathcal{F},P\right)  $, with the notations above plus $\left(  \xi
_{i},X_{i}\right)  =\left(  \widehat{\xi}_{i},\widehat{X}_{i}\right)  \circ
\pi_{1}$, consider the new vector%
\[
\left(  \left(  \xi_{1}^{\ast},X_{1}^{\ast}\right)  ,...,\left(  \xi_{N}%
^{\ast},X_{N}^{\ast}\right)  \right)  :=\left(  \left(  \xi_{\widetilde{s}%
_{N}\left(  1\right)  },X_{\widetilde{s}_{N}\left(  1\right)  }\right)
,...,\left(  \xi_{\widetilde{s}_{N}\left(  N\right)  },X_{\widetilde{s}%
_{N}\left(  N\right)  }\right)  \right)  .
\]
This vector takes values in $\mathcal{L}F_{N}\mathbb{T}^{2}$, not in
$\mathcal{L}\Lambda_{N}^{1}$ as the previous one $\left(  \left(
\widehat{\xi}_{1},\widehat{X}_{1}\right)  ,...,\left(  \widehat{\xi}%
_{N},\widehat{X}_{N}\right)  \right)  $. We claim its law is $\rho$, in the
correspondence $\rho\leftrightarrow\widehat{\rho}$ described above. Indeed,
$\left(  \left(  \xi_{1}^{\ast},X_{1}^{\ast}\right)  ,...,\left(  \xi
_{N}^{\ast},X_{N}^{\ast}\right)  \right)  $ is exchangeable, because given a
single deterministic permutation $s$, $\widetilde{s}_{N}\circ s$ is uniformly
distributed. And conditioning to have $X_{\widetilde{s}_{N}\left(  1\right)
}<_{L}...<_{L}X_{\widetilde{s}_{N}\left(  N\right)  }$ is like conditioning to
have $\widetilde{s}_{N}=id$, which gives $\widehat{\rho}$. Let us call
\textit{shuffling} the procedure illustrated here of composition with
independent permutations, to get the exchangeable distribution from a
distribution on $\mathcal{L}\Lambda_{N}^{1}$.

\textbf{Step 2}. Now let us prove the lemma. The law of $\widehat{\omega
}_{\cdot}^{N_{k}}$, being the same as the law of the original process, is
concentrated on $C\left(  \left[  0,T\right]  ;\mathcal{M}_{N_{k}}\left(
\mathbb{T}^{2}\right)  \right)  $. Hence, by the measurable maps
$\widetilde{h}_{i}$ described above, it defines random elements $\left(
\widehat{\xi}_{1},\widehat{X}_{\cdot}^{1,N_{k}}\right)  $, ..., $\left(
\widehat{\xi}_{N_{k}},\widehat{X}_{\cdot}^{N_{k},N_{k}}\right)  $ in
$\mathbb{R\times}C\left(  \left[  0,T\right]  ;\mathbb{T}^{2}\right)  $. One
has $\widehat{\omega}_{t}^{N_{k}}=\frac{1}{\sqrt{N_{k}}}\sum_{i=1}^{N_{k}%
}\widehat{\xi}_{i}\delta_{\widehat{X}_{t}^{i,N_{k}}}$; therefore we have
proved a first claim of the lemma (in fact we shall redefine the random vector
but the redefinition will not change this statement). We still have to prove
that $\lambda_{N}^{0}$ is the law of (\ref{initial cond}) (in fact we still
have to define properly (\ref{initial cond})) and $\left(  \widehat{X}%
_{t}^{1,N_{k}},...,\widehat{X}_{t}^{1,N_{k}}\right)  $ solves system
(\ref{vortex system}).

Since the original process $\omega_{\cdot}^{N_{k}}$ had the property that%
\[
\mathbb{E}\left[  \sup_{t\in\left[  0,T\right]  }\left\vert \left\langle
\omega_{t}^{N_{k}},\phi\right\rangle -\left\langle \omega_{0}^{N_{k}}%
,\phi\right\rangle -\int_{0}^{t}\int_{\mathbb{T}^{2}}\int_{\mathbb{T}^{2}%
}\nabla\phi\left(  x\right)  \cdot K\left(  x-y\right)  \omega_{s}^{N_{k}%
}\left(  dx\right)  \omega_{s}^{N_{k}}\left(  dy\right)  ds\right\vert
\wedge1\right]  =0
\]
for every $\phi\in C^{\infty}\left(  \mathbb{T}^{2}\right)  $, the same
property holds for the new process $\widehat{\omega}_{t}^{N_{k}}$ (because
they have the same law), hence $\widehat{P}$-a.s. it holds%
\[
\sup_{t\in\left[  0,T\right]  }\left\vert \left\langle \widehat{\omega}%
_{t}^{N_{k}},\phi\right\rangle -\left\langle \widehat{\omega}_{0}^{N_{k}}%
,\phi\right\rangle -\int_{0}^{t}\int_{\mathbb{T}^{2}}\int_{\mathbb{T}^{2}%
}\nabla\phi\left(  x\right)  \cdot K\left(  x-y\right)  \widehat{\omega}%
_{s}^{N_{k}}\left(  dx\right)  \widehat{\omega}_{s}^{N_{k}}\left(  dy\right)
ds\right\vert =0
\]
on a dense countable set of $\phi\in C^{\infty}\left(  \mathbb{T}^{2}\right)
$, which implies (using the structure $\widehat{\omega}_{t}^{N_{k}}=\frac
{1}{\sqrt{N_{k}}}\sum_{i=1}^{N_{k}}\widehat{\xi}_{i}\delta_{\widehat{X}%
_{t}^{i,N_{k}}}$)\ that $\left(  \widehat{X}_{t}^{1,N_{k}},...,\widehat{X}%
_{t}^{N_{k},N_{k}}\right)  $ satisfies (\ref{vortex system}). Below we shall
redefine this process but the redefinition will not change this property.

It remains to understand the law of (\ref{initial cond}). We have constructed
the random vector $\left(  \left(  \widehat{\xi}_{1},\widehat{X}_{0}^{1,N_{k}%
}\right)  ,...,\left(  \widehat{\xi}_{N_{k}},\widehat{X}_{0}^{N_{k},N_{k}%
}\right)  \right)  $, with $\widehat{X}_{0}^{1,N}<_{L}...<_{L}\widehat{X}%
_{0}^{N,N}$. We apply the shuffling procedure described at the end of Step 1,
hence redefining all r.v.'s and processes by composition with random
permutations. The result is an initial random vector of the form
(\ref{initial cond}) and the associated process $\left(  X_{t}^{1,N_{k}%
},...,X_{t}^{N_{k},N_{k}}\right)  $. The modifications introduced by shuffling
do not change the representation $\omega_{t}^{N_{k}}=\frac{1}{\sqrt{N_{k}}%
}\sum_{i=1}^{N_{k}}\xi_{i}\delta_{X_{t}^{i,N_{k}}}$ (now $\omega_{t}^{N_{k}}$
is the process defined before the lemma) and the fact that $\left(
X_{t}^{1,N_{k}},...,X_{t}^{N_{k},N_{k}}\right)  $ solves system
(\ref{vortex system}). We claim that the new initial random vector
(\ref{initial cond}) has law $\lambda_{N}^{0}$. By construction the vector
(\ref{initial cond}) is exchangeable and its law is the unique exchangeable
law on $\mathcal{L}F_{N}\mathbb{T}^{2}$ corresponding to a certain probability
measure $\widehat{\rho}$ on $\mathcal{L}\Lambda_{N}^{1}$ that we now describe.
Since $\lambda_{N}^{0}$ has this property, we deduce that $\lambda_{N}^{0}$ is
the law of (\ref{initial cond}). Let us describe $\widehat{\rho}$. It is the
law of $\left(  \left(  \widehat{\xi}_{1},\widehat{X}_{0}^{1,N_{k}}\right)
,...,\left(  \widehat{\xi}_{N_{k}},\widehat{X}_{0}^{N_{k},N_{k}}\right)
\right)  $, random vector constructed through the unique maps $h_{i}$, hence
$\widehat{\rho}$ is the push forward under $\left(  h_{1},...,h_{N}\right)  $
of the law of $\omega_{0}^{N_{k}}$; call it $\pi_{t=0}Q^{N_{k}}$. These
correspondences are bijections and, as already said, if we start by
$\lambda_{N}^{0}$ and push it forward (in opposite direction) to a law on
$\omega_{0}^{N_{k}}$ we find $\pi_{t=0}Q^{N_{k}}$. Thus we have the identification.
\end{proof}

Given $\phi\in C^{\infty}\left(  \mathbb{T}^{2}\right)  $ and $t\in\left[
0,T\right]  $, we are going to prove that
\[
E\left[  \left\vert \left\langle \omega_{t},\phi\right\rangle -\left\langle
\omega_{0},\phi\right\rangle -\int_{0}^{t}\left\langle H_{\phi},\omega
_{s}\otimes\omega_{s}\right\rangle ds\right\vert \wedge1\right]  =0.
\]
This implies that $\left\langle \omega_{t},\phi\right\rangle =\left\langle
\omega_{0},\phi\right\rangle -\int_{0}^{t}\left\langle H_{\phi},\omega
_{s}\otimes\omega_{s}\right\rangle ds$ with $P$-probability one, at time $t $.
Since the processes involved are continuous, this implies that the identity
holds uniformly in time, with $P$-probability one.

Based on the identity%
\[
\left\langle \omega_{t}^{N_{k}},\phi\right\rangle -\left\langle \omega
_{0}^{N_{k}},\phi\right\rangle -\int_{0}^{t}\left\langle H_{\phi},\omega
_{s}^{N_{k}}\otimes\omega_{s}^{N_{k}}\right\rangle ds=0
\]
and the general fact that $\left\vert x+y\right\vert \wedge1\leq\left(
\left\vert x\right\vert \wedge1\right)  +\left(  \left\vert y\right\vert
\wedge1\right)  $, one has the inequality%

\begin{align*}
& E\left[  \left\vert \left\langle \omega_{t},\phi\right\rangle -\left\langle
\omega_{0},\phi\right\rangle -\int_{0}^{t}\left\langle H_{\phi},\omega
_{s}\otimes\omega_{s}\right\rangle ds\right\vert \wedge1\right] \\
& \leq E\left[  \left(  \left\vert \left\langle \omega_{t},\phi\right\rangle
-\left\langle \omega_{t}^{N_{k}},\phi\right\rangle \right\vert \right)
\wedge1\right]  +E\left[  \left(  \left\vert \left\langle \omega_{0}%
,\phi\right\rangle -\left\langle \omega_{0}^{N_{k}},\phi\right\rangle
\right\vert \right)  \wedge1\right] \\
& +E\left[  \left(  \left\vert \int_{0}^{t}\left\langle H_{\phi},\omega
_{s}^{N_{k}}\otimes\omega_{s}^{N_{k}}\right\rangle ds-\int_{0}^{t}\left\langle
H_{\phi},\omega_{s}\otimes\omega_{s}\right\rangle ds\right\vert \right)
\wedge1\right]  .
\end{align*}
We have, for $\phi\in C^{\infty}\left(  \mathbb{T}^{2}\right)  $ and
$t\in\left[  0,T\right]  $,
\[
\lim_{k\rightarrow\infty}E\left[  \left(  \left\vert \left\langle \omega
_{t},\phi\right\rangle -\left\langle \omega_{t}^{N_{k}},\phi\right\rangle
\right\vert \right)  \wedge1\right]  =0
\]
simply because we have a.s. convergence in $C\left(  \left[  0,T\right]
;H^{-1-\delta}\left(  \mathbb{T}^{2}\right)  \right)  $. Hence it remains to
prove%
\[
\lim_{k\rightarrow\infty}E\left[  \left(  \left\vert \int_{0}^{t}\left\langle
H_{\phi},\omega_{s}^{N_{k}}\otimes\omega_{s}^{N_{k}}\right\rangle ds-\int%
_{0}^{t}\left\langle H_{\phi},\omega_{s}\otimes\omega_{s}\right\rangle
ds\right\vert \right)  \wedge1\right]  =0
\]
which is the most demanding part of the passage to the limit. Let us consider
a smooth (of class $H^{2+}$ is sufficient) approximation $H_{\phi}^{\delta}$
of $H_{\phi}$, $\delta>0$, with the property $H_{\phi}^{\delta}\left(
x,x\right)  =0$ (see Remark \ref{remark existence}). We have
\[
\lim_{n\rightarrow\infty}E\left[  \left(  \left\vert \int_{0}^{t}\left\langle
H_{\phi}^{\delta},\omega_{s}^{N_{k}}\otimes\omega_{s}^{N_{k}}\right\rangle
ds-\int_{0}^{t}\left\langle H_{\phi}^{\delta},\omega_{s}\otimes\omega
_{s}\right\rangle ds\right\vert \right)  \wedge1\right]  =0
\]
again because of a.s. convergence of $\omega^{N_{k}}$ to $\omega$ in $C\left(
\left[  0,T\right]  ;H^{-1-}\left(  \mathbb{T}^{2}\right)  \right)  $ and thus
of $\omega^{N_{k}}\otimes\omega^{N_{k}}$ to $\omega\otimes\omega$ in $C\left(
\left[  0,T\right]  ;H^{-2-}\left(  \mathbb{T}^{2}\times\mathbb{T}^{2}\right)
\right)  $. Therefore%
\begin{align*}
& \underset{k\rightarrow\infty}{\lim\sup}E\left[  \left(  \left\vert \int%
_{0}^{t}\left\langle H_{\phi},\omega_{s}^{N_{k}}\otimes\omega_{s}^{N_{k}%
}\right\rangle ds-\int_{0}^{t}\left\langle H_{\phi},\omega_{s}\otimes
\omega_{s}\right\rangle ds\right\vert \right)  \wedge1\right] \\
& \leq E\left[  \left(  \left\vert \int_{0}^{t}\left\langle H_{\phi}-H_{\phi
}^{\delta},\omega_{s}\otimes\omega_{s}\right\rangle ds\right\vert \right)
\wedge1\right]  +\underset{k\in\mathbb{N}}{\sup}E\left[  \left(  \left\vert
\int_{0}^{t}\left\langle H_{\phi}-H_{\phi}^{\delta},\omega_{s}^{N_{k}}%
\otimes\omega_{s}^{N_{k}}\right\rangle ds\right\vert \right)  \wedge1\right]
.
\end{align*}

We know that%
\begin{align*}
& E\left[  \left(  \left\vert \int_{0}^{t}\left\langle H_{\phi}-H_{\phi
}^{\delta},\omega_{s}\otimes\omega_{s}\right\rangle ds\right\vert \right)
\wedge1\right]  \leq\int_{0}^{t}E\left[  \left\vert \left\langle H_{\phi
}-H_{\phi}^{\delta},\omega_{s}\otimes\omega_{s}\right\rangle \right\vert
\right]  ds\\
& \leq C\int_{0}^{t}E\left[  \left\vert \left\langle H_{\phi}-H_{\phi}%
^{\delta},\omega_{s}\otimes\omega_{s}\right\rangle \right\vert ^{2}\right]
^{1/2}ds
\end{align*}
and the last term is arbitrarily small with $\delta$, due to Corollary
\ref{corollary WN} (a little argument is needed because $H_{\phi}-H_{\phi
}^{\delta}$ is not smooth but the computation is similar to the Cauchy
property of Theorem \ref{Thm Cauchy}). It remain to show that
\[
E\left[  \left(  \left\vert \int_{0}^{t}\left\langle H_{\phi}-H_{\phi}%
^{\delta},\omega_{s}^{N_{k}}\otimes\omega_{s}^{N_{k}}\right\rangle
ds\right\vert \right)  \wedge1\right]
\]
is small for small $\delta$, uniformly in $k$. But this case is similar to the
previous one, using now Lemma \ref{lemma integrabil pont vort}. The proof is complete.

\section{Proof of Theorem \ref{Thm AC rho}}

Recall the definitions of $\lambda_{N}^{0}\left(  d\theta\right)  $,
$\mathcal{T}_{N}$, $\mu_{N}^{0}\left(  d\omega\right)  $ from Remark
\ref{remark main measure on vortices}.

\begin{lemma}
Given a measurable function $\rho:H^{-1-\delta}\left(  \mathbb{T}^{2}\right)
\rightarrow\lbrack0,\infty)$ such that $\int_{H^{-1-\delta}\left(
\mathbb{T}^{2}\right)  }\rho\left(  \omega\right)  \mu_{N}^{0}\left(
d\omega\right)  <\infty$, the measure $\lambda_{N}^{\rho}\left(
d\theta\right)  :=\rho\left(  \mathcal{T}_{N}\left(  \theta\right)  \right)
\lambda_{N}^{0}\left(  d\theta\right)  $ on $\left(  \mathbb{R}\times
\mathbb{T}^{2}\right)  ^{N}$ has the property that its image measure $\mu
_{N}^{\rho}\left(  d\omega\right)  $ on $H^{-1-\delta}\left(  \mathbb{T}%
^{2}\right)  $ under the map $\mathcal{T}_{N}$ is $\rho\left(  \omega\right)
\mu_{N}^{0}\left(  d\omega\right)  $.
\end{lemma}

\begin{proof}
By definition of $\mu_{N}^{\rho}\left(  d\omega\right)  $ and $\lambda
_{N}^{\rho}\left(  d\theta\right)  $, for every non-negative measurable
function $F$ we have
\begin{align*}
\int_{H^{-1-\delta}\left(  \mathbb{T}^{2}\right)  }F\left(  \omega\right)
\mu_{N}^{\rho}\left(  d\omega\right)   & =\int_{\mathbb{R}^{N}\times
\mathbb{R}^{2N}}F\left(  \mathcal{T}_{N}\left(  \theta\right)  \right)
\lambda_{N}^{\rho}\left(  d\theta\right) \\
& =\int_{\mathbb{R}^{N}\times\mathbb{R}^{2N}}F\left(  \mathcal{T}_{N}\left(
\theta\right)  \right)  \rho\left(  \mathcal{T}_{N}\left(  \theta\right)
\right)  \lambda_{N}^{0}\left(  d\theta\right) \\
& =\int_{H^{-1-\delta}\left(  \mathbb{T}^{2}\right)  }F\left(  \omega\right)
\rho\left(  \omega\right)  \mu_{N}^{0}\left(  d\omega\right)  .
\end{align*}

\end{proof}

We may now prove Theorem \ref{Thm AC rho}. Given $\rho_{0}\in C_{b}\left(
H^{-1-}\left(  \mathbb{T}^{2}\right)  \right)  $, $\rho_{0}\geq0$, $\int%
\rho_{0}d\mu=1$ ($\mu$ here is the white noise Gaussian law on $H^{-1-}\left(
\mathbb{T}^{2}\right)  $), there is a constant $C_{N}>0$ such that $C_{N}%
\int_{H^{-1-\delta}\left(  \mathbb{T}^{2}\right)  }\rho_{0}\left(
\omega\right)  \mu_{N}^{0}\left(  d\omega\right)  =1$, for any $\delta>0$.
Since $\mu_{N}^{0}$ converges weakly to $\mu$ on $H^{-1-\delta}\left(
\mathbb{T}^{2}\right)  $\ and $\rho_{0}$ is continuous and bounded on
$H^{-1-\delta}\left(  \mathbb{T}^{2}\right)  $, we deduce $\lim_{N\rightarrow
\infty}C_{N}=1$. Let us consider, on Borel sets of $\left(  \mathbb{R}%
\times\mathbb{T}^{2}\right)  ^{N}$, the finite positive measure $C_{N}\rho
_{0}\left(  \mathcal{T}_{N}\left(  \theta\right)  \right)  \lambda_{N}%
^{0}\left(  d\theta\right)  $. By the lemma, its image measure on
$H^{-1-\delta}\left(  \mathbb{T}^{2}\right)  $ under the map $\mathcal{T}_{N}$
is $C_{N}\rho_{0}\left(  \omega\right)  \mu_{N}^{0}\left(  d\omega\right)  $
(we apply the lemma to $\rho\left(  \omega\right)  :=C_{N}\rho_{0}\left(
\omega\right)  $). The point vortex dynamics is well defined for a.e. $\left(
\left(  \xi_{1},X_{0}^{1}\right)  ,...,\left(  \xi_{N},X_{0}^{N}\right)
\right)  \in\left(  \mathbb{R}\times\mathbb{T}^{2}\right)  ^{N}$ with respect
to $C_{N}\rho_{0}\left(  \mathcal{T}_{N}\left(  \theta\right)  \right)
\lambda_{N}^{0}\left(  d\theta\right)  $, because this fact holds for
$\lambda_{N}^{0}\left(  d\theta\right)  $.\ Denote by $\omega_{t}^{N}$ the
vorticity of this point vortex dynamics; the law of $\omega_{0}^{N}$ is
$C_{N}\rho_{0}\left(  \omega\right)  \mu_{N}^{0}\left(  d\omega\right)  $.

Denote by $\Phi_{t}^{N}$ the map in $H^{-1-}\left(  \mathbb{T}^{2}\right)  $,
defined a.s. with respect to $\mu_{N}^{0}$, which gives $\omega_{t}^{N}%
=\Phi_{t}^{N}\omega_{0}^{N}$. The law of $\omega_{t}^{N}$ has the form%
\[
C_{N}\rho_{0}\left(  \left(  \Phi_{t}^{N}\right)  ^{-1}\left(  \omega\right)
\right)  \mu_{N}^{0}\left(  d\omega\right)
\]
where $\left(  \Phi_{t}^{N}\right)  ^{-1}$ is the inverse map of $\Phi_{t}%
^{N}$ and it is defined for $\mu_{N}^{0}$-a.e. $\omega\in H^{-1-}\left(
\mathbb{T}^{2}\right)  $. Indeed, for every non-negative measurable function
$F$ we have%
\begin{align*}
\mathbb{E}\left[  F\left(  \omega_{t}^{N}\right)  \right]   & =\mathbb{E}%
\left[  F\left(  \Phi_{t}^{N}\omega_{0}^{N}\right)  \right]  =\int%
_{H^{-1-\delta}\left(  \mathbb{T}^{2}\right)  }F\left(  \Phi_{t}%
^{N}\mathcal{\omega}\right)  C_{N}\rho_{0}\left(  \mathcal{\omega}\right)
\mu_{N}^{0}\left(  d\omega\right) \\
& =\int_{H^{-1-\delta}\left(  \mathbb{T}^{2}\right)  }F\left(  \mathcal{\omega
}\right)  C_{N}\rho_{0}\left(  \left(  \Phi_{t}^{N}\right)  ^{-1}\left(
\mathcal{\omega}\right)  \right)  \left(  \Phi_{t}^{N}\right)  _{\ast}\mu
_{N}^{0}\left(  d\omega\right)
\end{align*}
but $\left(  \Phi_{t}^{N}\right)  _{\ast}\mu_{N}^{0}=\mu_{N}^{0}$, see
Proposition \ref{propos point vortices}.

Therefore, for every non-negative measurable function $F$ on $H^{-1-}\left(
\mathbb{T}^{2}\right)  $, one has
\[
\mathbb{E}\left[  F\left(  \omega_{t}^{N}\right)  \right]  =\mathbb{E}\left[
C_{N}\rho_{0}\left(  \left(  \Phi_{t}^{N}\right)  ^{-1}\left(  \omega_{WN}%
^{N}\right)  \right)  F\left(  \omega_{WN}^{N}\right)  \right]
\]
where $\omega_{WN}^{N}$ denotes the random point vortices initial condition
with law $\mu_{N}^{0}$.

Let $Q^{N}$ be the law of $\omega^{N}$ on Borel subsets of the space
$\mathcal{X}$, as in the previous section. We want to prove that the family
$\left\{  Q^{N}\right\}  _{N\in\mathbb{N}}$ is tight in $\mathcal{X}$, by
proving that it is bounded in probability in $\mathcal{Y}$ (see previous
section). The family $\left\{  Q^{N}\right\}  _{N\in\mathbb{N}}$ is bounded in
probability in $L^{p_{0}}\left(  0,T;H^{-1-\delta}\left(  \mathbb{T}%
^{2}\right)  \right)  $,\ because%
\begin{align*}
\mathbb{E}\left[  \int_{0}^{T}\left\Vert \omega_{t}^{N}\right\Vert
_{H^{-1-\delta}}^{p_{0}}dt\right]   & =\int_{0}^{T}\mathbb{E}\left[
\left\Vert \omega_{t}^{N}\right\Vert _{H^{-1-\delta}}^{p_{0}}\right]  dt\\
& =\int_{0}^{T}\mathbb{E}\left[  C_{N}\rho_{0}\left(  \left(  \Phi_{t}%
^{N}\right)  ^{-1}\left(  \omega_{WN}^{N}\right)  \right)  \left\Vert
\omega_{WN}^{N}\right\Vert _{H^{-1-\delta}}^{p_{0}}\right]  dt\\
& \leq C_{N}\left\Vert \rho_{0}\right\Vert _{\infty}T\mathbb{E}\left[
\left\Vert \omega_{WN}^{N}\right\Vert _{H^{-1-\delta}}^{p_{0}}\right]  \leq
C_{p_{0},\delta}C_{N}\left\Vert \rho_{0}\right\Vert _{\infty}T
\end{align*}
(see the estimate of the previous section). It is bounded in probability in
$W^{1,2}\left(  0,T;H^{-\gamma}\left(  \mathbb{T}^{2}\right)  \right)  $, by
the same arguments given in the previous section, because%
\begin{align*}
& \mathbb{E}\left[  \left\vert \left\langle \omega_{t}^{N}\otimes\omega
_{t}^{N},H_{\phi}\right\rangle \right\vert ^{2}\right] \\
& =\mathbb{E}\left[  C_{N}\rho_{0}\left(  \left(  \Phi_{t}^{N}\right)
^{-1}\left(  \omega_{WN}^{N}\right)  \right)  \left\vert \left\langle
\omega_{WN}^{N}\otimes\omega_{WN}^{N},H_{\phi}\right\rangle \right\vert
^{2}\right] \\
& \leq C_{N}\left\Vert \rho_{0}\right\Vert _{\infty}\mathbb{E}\left[
\left\vert \left\langle \omega_{WN}^{N}\otimes\omega_{WN}^{N},H_{\phi
}\right\rangle \right\vert ^{2}\right] \\
& \leq C_{N}\left\Vert \rho_{0}\right\Vert _{\infty}C\left\Vert H_{\phi
}\right\Vert _{\infty}^{2}\leq C_{N}\left\Vert \rho_{0}\right\Vert _{\infty
}C\left\Vert D^{2}\phi\right\Vert _{\infty}^{2}%
\end{align*}
(all the other steps of the proof are the same). This proves tightness in
$\mathcal{X}$.

Repeating the arguments of the previous section (we use Prohorov and Skorokhod
theorems) we extract a subsequence $N_{k}$, construct a new probability space,
denoted by $\left(  \Xi,\mathcal{F},P\right)  $ and processes $\omega
_{t}^{N_{k}}$, $\omega_{t}$ with trajectories in $\mathcal{X}$, such that the
laws of $\omega^{N_{k}}$ and $\omega$ are $Q^{N_{k}}$ and $Q$ respectively,
and $\omega^{N_{k}}$ converges to $\omega$ in the topology of $\mathcal{X}$,
$P$-a.s.; and the structure of $\omega^{N_{k}}$ as sum of delta Dirac is
identified, namely Lemma \ref{lemma representation} is still true in the case
treated here (the proof does not require modifications). The only difference
is that here the law of (\ref{initial cond}) is $C_{N}\rho_{0}\left(
\omega\right)  \mu_{N}^{0}\left(  d\omega\right)  $. Let us first prove that
the law of $\omega_{t}$ on $H^{-1-}\left(  \mathbb{T}^{2}\right)  $, called
herewith $\mu_{t}$, is absolutely continuous with respect to $\mu$ (the law of
white noise) with bounded density. For every $F\in C_{b}\left(  H^{-1-}\left(
\mathbb{T}^{2}\right)  \right)  $, we have%
\begin{align*}
\int F\left(  \omega\right)  \mu_{t}\left(  d\omega\right)   & =\lim
_{N\rightarrow\infty}\mathbb{E}\left[  F\left(  \omega_{t}^{N}\right)
\right]  =\lim_{N\rightarrow\infty}\mathbb{E}\left[  C_{N}\rho_{0}\left(
\left(  \Phi_{t}^{N}\right)  ^{-1}\left(  \omega_{WN}^{N}\right)  \right)
F\left(  \omega_{WN}^{N}\right)  \right] \\
& \leq\left\Vert \rho_{0}\right\Vert _{\infty}\lim_{N\rightarrow\infty
}\mathbb{E}\left[  F\left(  \omega_{WN}^{N}\right)  \right]  =\left\Vert
\rho_{0}\right\Vert _{\infty}\int F\left(  \omega\right)  \mu\left(
d\omega\right)  .
\end{align*}
This implies $\mu_{t}<<\mu$ with bounded density, denoted in the sequel by
$\rho_{t}$.

We can pass to the limit as in the previous section. Inspection in that proof
reveals that we have only to explain why $E\left[  \left(  \left\vert \int%
_{0}^{t}\left\langle H_{\phi}-H_{\phi}^{\delta},\omega_{s}\otimes\omega
_{s}\right\rangle ds\right\vert \right)  \wedge1\right]  $ and
\begin{equation}
E\left[  \left(  \left\vert \int_{0}^{t}\left\langle H_{\phi}-H_{\phi}%
^{\delta},\omega_{s}^{N_{k}}\otimes\omega_{s}^{N_{k}}\right\rangle
ds\right\vert \right)  \wedge1\right] \label{second term}%
\end{equation}
are small for small $\delta$, uniformly in $k$ for the second term. We have%
\begin{align*}
& E\left[  \left(  \left\vert \int_{0}^{t}\left\langle H_{\phi}-H_{\phi
}^{\delta},\omega_{s}\otimes\omega_{s}\right\rangle ds\right\vert \right)
\wedge1\right]  \leq C\int_{0}^{t}E\left[  \left\vert \left\langle H_{\phi
}-H_{\phi}^{\delta},\omega_{s}\otimes\omega_{s}\right\rangle \right\vert
^{2}\right]  ^{1/2}ds\\
& =C\int_{0}^{t}E\left[  \rho_{s}\left(  \omega_{WN}\right)  \left\vert
\left\langle H_{\phi}-H_{\phi}^{\delta},\omega_{WN}\otimes\omega
_{WN}\right\rangle \right\vert ^{2}\right]  ^{1/2}ds\\
& \leq C\int_{0}^{t}E\left[  \left\vert \left\langle H_{\phi}-H_{\phi}%
^{\delta},\omega_{WN}\otimes\omega_{WN}\right\rangle \right\vert ^{2}\right]
^{1/2}ds
\end{align*}
that is arbitrarily small with $\delta$, due to Corollary \ref{corollary WN}.
The proof for (\ref{second term}) is similar.

\section{Proof of Theorem \ref{Thm intro}\label{Sect proof Thm intro}}

We have proved, see Theorem \ref{Thm AC} part (i), that there exist a
probability space $\left(  \Xi,\mathcal{F},P\right)  $ and a measurable map
$\omega_{\cdot}:\Xi\times\left[  0,T\right]  \rightarrow C^{\infty}\left(
\mathbb{T}^{2}\right)  ^{\prime}$ such that $\omega_{\cdot}$ is a
time-stationary white noise solution of Euler equations, in the sense of
Definition \ref{def WN sol}, and the random point vortex system, defined on
$\left(  \Xi,\mathcal{F},P\right)  $, has a subsequence which converges in law
to this solution, in $C\left(  \left[  0,T\right]  ;H^{-1-}\left(
\mathbb{T}^{2}\right)  \right)  $.

This means that:

\begin{itemize}
\item $\omega_{0}$ is distributed as a white noise, hence it takes values in
$H^{-1-}\left(  \mathbb{T}^{2}\right)  \backslash\left(  H^{-1}\left(
\mathbb{T}^{2}\right)  \cup\mathcal{M}\left(  \mathbb{T}^{2}\right)  \right)
$ and it is a full $\mu$-measure set, where $\mu$ is the enstrophy Gaussian measure;\ 

\item there exists a set $\Xi_{1}\in\mathcal{F}$ with $P\left(  \Xi
_{1}\right)  =1$ such that for all $\theta\in\Xi_{1}$ one has $\omega_{\cdot
}\left(  \theta\right)  \in C\left(  \left[  0,T\right]  ;H^{-1-}\left(
\mathbb{T}^{2}\right)  \right)  $.
\end{itemize}

Moreover, for every $\phi\in C^{\infty}\left(  \mathbb{T}^{2}\right)  $, the
following two claims hold true:

\begin{itemize}
\item for $P$-a.e. $\theta\in\Xi$, $s\mapsto\left\langle \omega_{s}%
\otimes\omega_{s},H_{\phi}\right\rangle \left(  \theta\right)  $ is well
defined as $L^{2}\left(  0,T\right)  $-limit of a subsequence of
$s\mapsto\left\langle \omega_{s}\left(  \theta\right)  \otimes\omega
_{s}\left(  \theta\right)  ,H_{\phi}^{n}\right\rangle $ (Definition
\ref{Def nonlin in t} identifies $s\mapsto\left\langle \omega_{s}\otimes
\omega_{s},H_{\phi}\right\rangle $ by an $L^{2}\left(  \Xi\right)  $-limit,
from which we can extract a subsequence which converges $P$-almost surely)

\item for $P$-a.e. $\theta\in\Xi$, we have the identity uniformly in time:
\[
\left\langle \omega_{t}\left(  \theta\right)  ,\phi\right\rangle =\left\langle
\omega_{0}\left(  \theta\right)  ,\phi\right\rangle +\int_{0}^{t}\left\langle
\omega_{s}\otimes\omega_{s},H_{\phi}\right\rangle \left(  \theta\right)  ds.
\]

\end{itemize}

Therefore, if $\mathcal{D}$\ is a countable set in $C^{\infty}\left(
\mathbb{T}^{2}\right)  $, applying a diagonal procedure to extract a single
subsequence with $P$-a.s. convergence of $\left\langle \omega_{s}\otimes
\omega_{s},H_{\phi}^{n}\right\rangle $, we can find a set $\Xi_{2}%
\in\mathcal{F}$ with $P\left(  \Xi_{2}\right)  =1$ such that for all
$\theta\in\Xi_{2}$:

\begin{itemize}
\item for every $\phi\in\mathcal{D}$, $s\mapsto\left\langle \omega_{s}%
\otimes\omega_{s},H_{\phi}\right\rangle \left(  \theta\right)  $ is well
defined as $L^{2}\left(  0,T\right)  $-limit of a subsequence of
$s\mapsto\left\langle \omega_{s}\left(  \theta\right)  \otimes\omega
_{s}\left(  \theta\right)  ,H_{\phi}^{n}\right\rangle $

\item for every $\phi\in\mathcal{D}$, we have the identity above uniformly in time.
\end{itemize}

Putting together $\Xi_{1,2}:=\Xi_{1}\cap\Xi_{2}$, for all $\theta\in\Xi_{1,2}$
the function $\omega_{\cdot}\left(  \theta\right)  $ satisfies the conditions
of Theorem \ref{Thm intro}, part (i), for all $\phi\in\mathcal{D}$. We have
thus proved such claim, limited to $\phi\in\mathcal{D}$.

Assume $\mathcal{D}$\ is also dense in $C^{\infty}\left(  \mathbb{T}%
^{2}\right)  $;\ precisely we shall use density in $H^{-\gamma}\left(
\mathbb{T}^{2}\right)  $ for some $\gamma>3$. Given $\phi\in H^{-\gamma
}\left(  \mathbb{T}^{2}\right)  $, take $\phi_{k}\rightarrow\phi$ in
$H^{-\gamma}\left(  \mathbb{T}^{2}\right)  $, $\phi_{k}\in\mathcal{D}$. We
have%
\begin{align*}
& \int_{0}^{T}\left\vert \left\langle \omega_{s}\left(  \theta\right)
\otimes\omega_{s}\left(  \theta\right)  ,H_{\phi}^{n}-H_{\phi}^{m}%
\right\rangle \right\vert ^{2}ds\\
& \leq2\int_{0}^{T}\left\vert \left\langle \omega_{s}\left(  \theta\right)
\otimes\omega_{s}\left(  \theta\right)  ,H_{\phi_{k}}^{n}-H_{\phi_{k}}%
^{m}\right\rangle \right\vert ^{2}ds+2\int_{0}^{T}\left\vert \left\langle
\omega_{s}\left(  \theta\right)  \otimes\omega_{s}\left(  \theta\right)
,H_{\phi_{k}-\phi}^{n}-H_{\phi_{k}-\phi}^{m}\right\rangle \right\vert ^{2}ds
\end{align*}
hence, to get that $s\mapsto\left\langle \omega_{s}\left(  \theta\right)
\otimes\omega_{s}\left(  \theta\right)  ,H_{\phi}^{n}\right\rangle $ is Cauchy
in $L^{2}\left(  0,T\right)  $ it is sufficient to prove that
\[
\int_{0}^{T}\left\vert \left\langle \omega_{s}\left(  \theta\right)
\otimes\omega_{s}\left(  \theta\right)  ,H_{\phi_{k}-\phi}^{n}\right\rangle
\right\vert ^{2}ds
\]
is small uniformly in $n$, if $k$ is large enough. Let us prove that this
property is true in a set $\Xi_{3}\in\mathcal{F}$ with $P\left(  \Xi
_{3}\right)  =1$. Then the proof of Theorem \ref{Thm intro}, part (i), will be
complete, considering $\theta\in\Xi_{1,2,3}:=\Xi_{1}\cap\Xi_{2}\cap\Xi_{3}$.

Consider the distribution $g_{s}^{n}\left(  \theta\right)  $ defined as%
\[
\left\langle g_{s}^{n}\left(  \theta\right)  ,\phi\right\rangle :=\left\langle
\omega_{s}\left(  \theta\right)  \otimes\omega_{s}\left(  \theta\right)
,H_{\phi}^{n}\right\rangle .
\]
We have%
\begin{align*}
\left\Vert g_{s}^{n}\left(  \theta\right)  \right\Vert _{H^{-\gamma}}^{2}  &
=\sum_{k}\left(  1+\left\vert k\right\vert ^{2}\right)  ^{-\gamma}\left\vert
\left\langle g_{s}^{n}\left(  \theta\right)  ,e_{k}\right\rangle \right\vert
^{2}\\
& =\sum_{k}\left(  1+\left\vert k\right\vert ^{2}\right)  ^{-\gamma}\left\vert
\left\langle \omega_{s}\left(  \theta\right)  \otimes\omega_{s}\left(
\theta\right)  ,H_{e_{k}}^{n}\right\rangle \right\vert ^{2}%
\end{align*}%
\begin{align*}
\mathbb{E}\left[  \int_{0}^{T}\left\Vert g_{s}^{n}\right\Vert _{H^{-\gamma}%
}^{2}ds\right]   & =\sum_{k}\left(  1+\left\vert k\right\vert ^{2}\right)
^{-\gamma}\mathbb{E}\left[  \int_{0}^{T}\left\vert \left\langle \omega
_{s}\otimes\omega_{s},H_{e_{k}}^{n}\right\rangle \right\vert ^{2}ds\right] \\
& \leq CT\sum_{k}\left(  1+\left\vert k\right\vert ^{2}\right)  ^{-\gamma
}\left\Vert e_{k}\right\Vert _{C^{2}}^{2}\leq CT\sum_{k}\left(  1+\left\vert
k\right\vert ^{2}\right)  ^{-\gamma}\left\vert k\right\vert ^{4}%
\end{align*}
and this is finite when $\gamma>3$. Hence there is a set $\Xi_{3}%
\in\mathcal{F}$ with $P\left(  \Xi_{3}\right)  =1$, such that $\int_{0}%
^{T}\left\Vert g_{s}^{n}\left(  \theta\right)  \right\Vert _{H^{-\gamma}}%
^{2}ds<\infty$ for all $\theta\in\Xi_{3}$. For such $\theta$ we have
\[
\int_{0}^{T}\left\vert \left\langle \omega_{s}\left(  \theta\right)
\otimes\omega_{s}\left(  \theta\right)  ,H_{\phi_{k}-\phi}^{n}\right\rangle
\right\vert ^{2}ds=\int_{0}^{T}\left\vert \left\langle g_{s}^{n}\left(
\theta\right)  ,\phi_{k}-\phi\right\rangle \right\vert ^{2}ds\leq C\left(
\theta\right)  \left\Vert \phi_{k}-\phi\right\Vert _{H^{\gamma}}^{2}%
\]
where $C\left(  \theta\right)  :=\int_{0}^{T}\left\Vert g_{s}^{n}\left(
\theta\right)  \right\Vert _{H^{-\gamma}}^{2}ds<\infty$. Hence we have the
required property.

As to claim (ii) of Theorem \ref{Thm intro}, we invoke the result of
\cite{MarPulv93}. First, let us recall Theorem \ref{Thm AC} part (ii): the
solution (not unique) provided by part (i) is the $P$-a.s. limit in $C\left(
\left[  0,T\right]  ;H^{-1-}\left(  \mathbb{T}^{2}\right)  \right)  $ of a
subsequence of the random point vortex system (\ref{vortex system}), defined
also on $\left(  \Xi,\mathcal{F},P\right)  $. This means that there is
$\left(  N_{k}\right)  _{k\in\mathbb{N}}$ and a subset $\Xi_{4}\in\mathcal{F}
$ of $\Xi_{1,2,3}$, still with $P\left(  \Xi_{4}\right)  =1$ such that for all
$\theta\in\Xi_{4}$ the function $\omega_{\cdot}\left(  \theta\right)  $ is the
$C\left(  \left[  0,T\right]  ;H^{-1-}\left(  \mathbb{T}^{2}\right)  \right)
$-limit of the sequence $\frac{1}{\sqrt{N_{k}}}\sum_{n=1}^{N_{k}}\xi
_{n}\left(  \theta\right)  \delta_{X_{t}^{n}\left(  \theta\right)  }$; with
the understanding that $\Xi_{4}$ is such that for all $\theta\in\Xi_{4}$ the
corresponding point vortex dynamics is well defined for all times, without
coalescence of points.

Taken $\theta\in\Xi_{4}$, the function $\omega_{\cdot}\left(  \theta\right)  $
satisfies the conditions of Theorem \ref{Thm intro}, part (i). In addition,
given any $\epsilon>0$, there is $k_{\epsilon}\in\mathbb{N}$ such that
\[
\sup_{t\in\left[  0,T\right]  }d_{H^{-1-}}\left(  \omega_{t}\left(
\theta\right)  ,\frac{1}{\sqrt{N_{k_{\epsilon}}}}\sum_{i=1}^{N_{k_{\epsilon}}%
}\xi_{i}\left(  \theta\right)  \delta_{X_{t}^{i}\left(  \theta\right)
}\right)  <\epsilon/2.
\]
Hence, for every $\phi\in C^{\infty}\left(  \mathbb{T}^{2}\right)  $ one has%
\[
\sup_{t\in\left[  0,T\right]  }\left\vert \left\langle \omega_{t}\left(
\theta\right)  ,\phi\right\rangle -\left\langle \frac{1}{\sqrt{N_{k_{\epsilon
}}}}\sum_{i=1}^{N_{k_{\epsilon}}}\xi_{i}\left(  \theta\right)  \delta
_{X_{t}^{i}\left(  \theta\right)  },\phi\right\rangle \right\vert <\epsilon/2.
\]
We now apply Theorem 2.1 of \cite{MarPulv93} to $\frac{1}{\sqrt{N_{k_{\epsilon
}}}}\sum_{i=1}^{N_{k_{\epsilon}}}\xi_{i}\left(  \theta\right)  \delta
_{X_{t}^{i}\left(  \theta\right)  }$, applicable because this solution of the
point vortex dynamics is global (namely without coalescence). It claims that
there exists a sequence $\omega_{\cdot}^{\left(  n\right)  }$ of solutions of
class $L^{\infty}\left(  \left[  0,T\right]  \times\mathbb{T}^{2}\right)  \cap
C\left(  \left[  0,T\right]  ;L^{p}\left(  \mathbb{T}^{2}\right)  \right)  $
for every $p\in\lbrack1,\infty)$, such that for every $\phi\in C\left(
\mathbb{T}^{2}\right)  $ one has%
\[
\lim_{n\rightarrow\infty}\sup_{t\in\left[  0,T\right]  }\left\vert
\left\langle \omega_{t}^{\left(  n\right)  },\phi\right\rangle -\left\langle
\frac{1}{\sqrt{N_{k_{\epsilon}}}}\sum_{i=1}^{N_{k_{\epsilon}}}\xi_{i}\left(
\theta\right)  \delta_{X_{t}^{i}\left(  \theta\right)  },\phi\right\rangle
\right\vert =0.
\]
Hence, given the value of $\epsilon$ above, there is $n_{0}$ such that for all
$n>n_{0}$%
\[
\sup_{t\in\left[  0,T\right]  }\left\vert \left\langle \omega_{t}^{\left(
n\right)  },\phi\right\rangle -\left\langle \frac{1}{\sqrt{N_{k_{\epsilon}}}%
}\sum_{i=1}^{N_{k_{\epsilon}}}\xi_{i}\left(  \theta\right)  \delta_{X_{t}%
^{i}\left(  \theta\right)  },\phi\right\rangle \right\vert <\epsilon/2.
\]
We deduce $\sup_{t\in\left[  0,T\right]  }\left\vert \left\langle \omega
_{t}\left(  \theta\right)  ,\phi\right\rangle -\left\langle \omega
_{t}^{\left(  n\right)  },\phi\right\rangle \right\vert <\epsilon$, concluding
the proof of Theorem \ref{Thm intro}, part (ii).

\section{Remarks on $\rho$-white noise solutions}

\subsection{The continuity equation}

Let $\mu$ be the law of white noise. Following \cite{DaPratoRoeckner},
\cite{DFR} and related literature, let us denote by $\mathcal{FC}_{b,T}^{1}$
the set of all functionals $F:\left[  0,T\right]  \times C^{\infty}\left(
\mathbb{T}^{2}\right)  ^{\prime}\rightarrow\mathbb{R}$ of the form $F\left(
t,\omega\right)  =\sum_{i=1}^{m}\widetilde{f}_{i}\left(  \left\langle
\omega,\phi_{1}\right\rangle ,...,\left\langle \omega,\phi_{n}\right\rangle
\right)  g_{i}\left(  t\right)  $, with $\phi_{1},...,\phi_{n}\in C^{\infty
}\left(  \mathbb{T}^{2}\right)  $, $\widetilde{f}_{i}\in C_{b}^{1}\left(
\mathbb{R}^{n}\right)  $, $g_{i}\in C^{1}\left(  \left[  0,T\right]  \right)
$ with $g_{i}\left(  T\right)  =0$. Given $F\in\mathcal{FC}_{b,T}^{1}$, denote
by $D_{\omega}F\left(  t,\omega\right)  $ the function
\[
\sum_{i=1}^{m}\sum_{j=1}^{n}\partial_{j}\widetilde{f}_{i}\left(  \left\langle
\omega,\phi_{1}\right\rangle ,...,\left\langle \omega,\phi_{n}\right\rangle
\right)  g_{i}\left(  t\right)  \phi_{j}.
\]

\begin{definition}
Given $F\in\mathcal{FC}_{b,T}^{1}$, we set%
\[
\left\langle D_{\omega}F\left(  t,\omega\right)  ,b\left(  \omega\right)
\right\rangle :=\sum_{i=1}^{m}\sum_{j=1}^{n}\partial_{j}\widetilde{f}%
_{i}\left(  \left\langle \omega,\phi_{1}\right\rangle ,...,\left\langle
\omega,\phi_{n}\right\rangle \right)  g_{i}\left(  t\right)  \left\langle
\omega\otimes\omega,H_{\phi_{j}}\right\rangle
\]
where $\left\langle \omega\otimes\omega,H_{\phi_{j}}\right\rangle $,
$j=1,...,n$, are the elements of $L^{2}\left(  \Xi\right)  $ given by Theorem
\ref{Thm Cauchy}. Hence $\left\langle D_{\omega}F\left(  t,\omega\right)
,b\left(  \omega\right)  \right\rangle $ is an element of $C\left(  \left[
0,T\right]  ;L^{2}\left(  \Xi\right)  \right)  $.
\end{definition}

\begin{definition}
We say that a bounded measurable function $\rho:\left[  0,T\right]  \times
H^{-1-}\left(  \mathbb{T}^{2}\right)  \rightarrow\lbrack0,\infty)$ is a
bounded weak solution of the continuity equation%
\begin{equation}
\partial_{t}\rho_{t}+\operatorname{div}_{\mu}\left(  \rho_{t}b\right)
=0\label{cont eq}%
\end{equation}
with initial condition $\rho_{0}$, if
\[
\int_{0}^{T}\int_{H^{-1-\delta/2}}\left(  \partial_{t}F\left(  t,\omega
\right)  +\left\langle D_{\omega}F\left(  t,\omega\right)  ,b\left(
\omega\right)  \right\rangle \right)  \rho_{t}\left(  \omega\right)
\mu\left(  d\omega\right)  dt=-\int_{H^{-1-\delta/2}}F\left(  0,\omega\right)
\rho_{0}\left(  \omega\right)  \mu\left(  d\omega\right)
\]
for all $F\in\mathcal{FC}_{b,T}^{1}$.
\end{definition}

\begin{proposition}
Any function $\rho$ given by Theorem \ref{Thm AC rho} is a bounded weak
solution of the continuity equation (\ref{cont eq}).
\end{proposition}

\begin{proof}
Let $\omega$ be a solution of Euler equations given by Theorem
\ref{Thm AC rho}, with the associated density function $\rho$. Given
$F\in\mathcal{FC}_{b,T}^{1}$ of the form $F\left(  t,\omega\right)
=\sum_{i=1}^{m}\widetilde{f}_{i}\left(  \left\langle \omega,\phi
_{1}\right\rangle ,...,\left\langle \omega,\phi_{n}\right\rangle \right)
g_{i}\left(  t\right)  $, we know that
\[
\left\langle \omega_{t},\phi_{j}\right\rangle =\left\langle \omega_{0}%
,\phi_{j}\right\rangle +\int_{0}^{t}\left\langle \omega_{s}\otimes\omega
_{s},H_{\phi_{j}}\right\rangle ds
\]
for every $j=1,...,n$. Here $P$-a.s. the function $s\mapsto\left\langle
\omega_{s}\otimes\omega_{s},H_{\phi_{j}}\right\rangle $ is of class
$L^{2}\left(  0,T\right)  $. Hence $\left\langle \omega_{t},\phi
_{j}\right\rangle $ is differentiable a.s. in time. We have, $P$-a.s., a.s. in
time,
\begin{align*}
& \partial_{t}\left(  F\left(  t,\omega_{t}\right)  \right) \\
& =\sum_{i=1}^{m}\sum_{j=1}^{n}\partial_{j}\widetilde{f}_{i}\left(
\left\langle \omega_{t},\phi_{1}\right\rangle ,...,\left\langle \omega
_{t},\phi_{n}\right\rangle \right)  g_{i}\left(  t\right)  \partial
_{t}\left\langle \omega_{t},\phi_{j}\right\rangle +\sum_{i=1}^{m}%
\widetilde{f}_{i}\left(  \left\langle \omega_{t},\phi_{1}\right\rangle
,...,\left\langle \omega_{t},\phi_{n}\right\rangle \right)  g_{i}^{\prime
}\left(  t\right) \\
& =\left\langle D_{\omega}F\left(  t,\omega_{t}\right)  ,b\left(  \omega
_{t}\right)  \right\rangle +\partial_{t}F\left(  t,\omega\right)
|_{\omega=\omega_{t}}%
\end{align*}
and thus%
\begin{align*}
& \int_{0}^{T}\int_{H^{-1-\delta/2}}\left(  \partial_{t}F\left(
t,\omega\right)  +\left\langle D_{\omega}F\left(  t,\omega\right)  ,b\left(
\omega\right)  \right\rangle \right)  \rho_{t}\left(  \omega\right)
\mu\left(  d\omega\right)  dt\\
& =\int_{0}^{T}\mathbb{E}\left[  \partial_{t}F\left(  t,\omega\right)
|_{\omega=\omega_{t}}+\left\langle D_{\omega}F\left(  t,\omega_{t}\right)
,b\left(  \omega_{t}\right)  \right\rangle \right]  dt\\
& =\int_{0}^{T}\mathbb{E}\left[  \partial_{t}\left(  F\left(  t,\omega
_{t}\right)  \right)  \right]  dt=\int_{0}^{T}\partial_{t}\mathbb{E}\left[
F\left(  t,\omega_{t}\right)  \right]  dt\\
& =\mathbb{E}\left[  F\left(  T,\omega_{T}\right)  \right]  -\mathbb{E}\left[
F\left(  0,\omega_{0}\right)  \right] \\
& =-\int_{H^{-1-\delta/2}}F\left(  0,\omega\right)  \rho_{0}\left(
\omega\right)  \mu\left(  d\omega\right)
\end{align*}
where the exchange of time-derivative and expectation is possible due to the
boundedness of terms in $F$; and we have used $g_{i}\left(  T\right)  =0$.
\end{proof}

The analysis of this continuity equation deserves more attention; we have just
mentioned here as a starting point of future investigations.

\subsection{An open problem\label{sect open}}

We have treated above the problem of approximating Albeverio-Cruzeiro solution
by smoother solutions of the Euler equations. Let us mention a sort of dual
problem, that can be formulated thanks to Theorem \ref{Thm AC rho}.

Given $\overline{\omega}_{0}\in L^{\infty}\left(  \mathbb{T}^{2}\right)  $,
there exists a unique solution $\overline{\omega}_{t}$ in $L^{\infty}\left(
\mathbb{T}^{2}\right)  $ of the Euler equations (point 1 of the Introduction).
For every $\epsilon>0$, consider the density
\[
\rho_{0}^{\left(  \epsilon\right)  }\left(  \omega\right)  =\frac{1}{Z_{R}%
}\exp\left(  -\frac{d_{H^{-1-}}\left(  \omega,\overline{\omega}_{0}\right)
^{2}}{2\epsilon}\right)
\]
defined on $H^{-1-}\left(  \mathbb{T}^{2}\right)  $, where
\[
Z_{R}=\int_{H^{-1-\delta}\left(  \mathbb{T}^{2}\right)  }\exp\left(
-\frac{d_{H^{-1-}}\left(  \omega,\overline{\omega}_{0}\right)  ^{2}}%
{2\epsilon}\right)  \mu\left(  d\omega\right)  .
\]
Let $\omega_{t}^{\left(  \epsilon\right)  }$ be a $\rho$-white noise solution,
provided by Theorem \ref{Thm AC rho}, corresponding to this initial density
$\rho_{0}^{\left(  \epsilon\right)  }$. Can we prove that $\omega_{t}^{\left(
\epsilon\right)  }$ converges, in a suitable sense, to $\overline{\omega}_{t}$?

We do not know the solution of this problem. Let us only remark that it looks
similar to the question of vortex point approximation of solutions of Euler
equations, solved in a smoothed Biot-Savart kernel scheme by \cite{MarPulv}
and in great generality by \cite{Shochet}. Also, very roughly, reminds large
deviation approximations of smooth paths by diffusion processes.

Theorems \ref{Thm AC} and \ref{Thm AC rho} give some intuition into
Albeverio-Cruzeiro solution and its variants, as a limit of random point
vortices. A positive solution of the previous problem would add more.

\begin{acknowledgement}
We thank Nicolai Tzvetkov for suggesting the problem of approximation by
solutions of the Euler equations, in contrast to the more traditional schemes
based on approximating equations (originally the author had presented a form
of Theorems \ref{Thm AC} and \ref{Thm AC rho} based on a Leray type
approximation; in \cite{AlbCruz} it is based on Galerkin approximations). The
idea of approximating Albeverio-Cruzeiro solution by random point vortices has
been discussed by the author with Hakima Bessaih, Marco Romito and Carl
M\"{u}ller some years ago; and recently with Romito again it was discussed the
possibility to use it as a bridge between Albeverio-Cruzeiro solution and more
smooth solutions; the author is deeply indebted to them for several ideas. We
thank Michael R\"{o}ckner for very important comments that allowed us to
improve Definition \ref{def WN sol} (see Definition \ref{Def nonlin in t}). We
also thank Giuseppe Da Prato, Ana Bela Cruzeiro, Sergio Albeverio, Benedetta
Ferrario and Francesco Grotto for important discussions and bibliographical comments.
\end{acknowledgement}

\end{document}